\documentclass[a4paper, 12pt, oneside, notitlepage]{amsart}
\usepackage[usenames,dvipsnames]{xcolor}
\usepackage[colorlinks=true,linkcolor=Red,citecolor=Green]{hyperref}
\usepackage{tikz-cd}
\usepackage[margin=3cm]{geometry}
\usepackage[normalem]{ulem}
\usepackage{amsmath,amssymb,amsthm,graphicx,mathrsfs,bbm,url}
\usepackage{amsthm}
\usepackage{wrapfig}
\usepackage{enumitem}
\usepackage{mathtools}
\usepackage[utf8]{inputenc}
 \usepackage[T1]{fontenc}
\usepackage[super]{nth}
\usepackage[open, openlevel=2, depth=3, atend]{bookmark}
\hypersetup{pdfstartview=XYZ}
\usepackage[font=footnotesize]{caption}
\usepackage{a4wide}
\usepackage[all]{xy}

\usepackage{epstopdf}
 
\usepackage{hyperref}

\theoremstyle{plain}
\newtheorem{theorem}{Theorem}[section]
\newtheorem*{theorem*}{Theorem}

\newtheorem{lemma}[theorem]{Lemma}
\newtheorem{proposition}[theorem]{Proposition}
\newtheorem{corollary}[theorem]{Corollary}

\theoremstyle{definition}
\newtheorem{definition}[theorem]{Definition}

\theoremstyle{remark}
\newtheorem{remark}[theorem]{Remark}

\numberwithin{equation}{section}

\newcommand{\fg}{\mathfrak{g}}
\newcommand{\fz}{\mathfrak{z}}

\newcommand{\C}{\mathbb{C}}
\newcommand{\R}{\mathbb{R}}

\newcommand{\Z}{\mathbb{Z}}

\newcommand{\V}{\mathbb{V}}
\newcommand{\X}{\mathbf{X}}

\newcommand{\HH}{\mathbb{H}}

\newcommand{\eps}{\varepsilon}

\newcommand{\mc}{\mathcal}

\newcommand{\dd}{\mathrm{d}}
\newcommand{\n}{\mathbf{n}}
\newcommand{\kk}{\mathbf{k}}
\newcommand{\btheta}{\boldsymbol{\theta}}
\newcommand{\comp}{\mathrm{comp}}
\newcommand{\bk}{\mathbf{k}}
\newcommand{\Lk}{\mathbf{L}^{\otimes \mathbf{k}}}

\DeclareMathOperator{\vol}{vol}
\DeclareMathOperator{\Ell}{ell}

\DeclareMathOperator{\Op}{Op}
\DeclareMathOperator{\WF}{WF}
\DeclareMathOperator{\Var}{Var}

\DeclareMathOperator{\id}{id}

\DeclareMathOperator{\hol}{hol}

\newcommand{\be}{\begin{equation}}
\newcommand{\ee}{\end{equation}}

\author{Mihajlo Cekić}
\address{CNRS and Universit\'e Paris-Est Cr\'eteil (LAMA), 94010 Cr\'eteil, France}
\email{mihajlo.cekic@cnrs.fr}

\author{Thibault Lefeuvre}
\address{Universit\'e Paris-Saclay, Laboratoire de math\'ematiques d’Orsay, 91405, Orsay, France.}
\email{thibault.lefeuvre1@universite-paris-saclay.fr}

\author{Sebastián Muñoz-Thon}
\address{Universit\'e Paris-Saclay, Laboratoire de math\'ematiques d’Orsay, 91405, Orsay, France.}
\email{sebastian.munoz-thon@universite-paris-saclay.fr}

\title[Decay of correlations on Abelian covers]{Decay of correlations on Abelian covers of isometric extensions of volume-preserving Anosov flows}

\begin{document}

\begin{abstract}
    We establish an asymptotic expansion in inverse powers of time of the correlation function of isometric extensions of volume-preserving Anosov flows on Abelian covers of closed manifolds.
\end{abstract}

\maketitle

\section{Introduction}

\subsection{Decay of correlations for Anosov flows on Abelian covers}

Let $M_0$ be a smooth closed connected manifold equipped with a smooth Anosov flow $\varphi_t^0 \colon M_0 \to M_0$ that preserves a smooth probability measure $\vol_{M_0}$. The Anosov property yields a continuous flow-invariant decomposition of $TM_0$ into
\[
    TM_0 = \R X_{M_0} \oplus E_s \oplus E_u,
\]
where $X_{M_0} \in C^\infty(M_0,TM_0)$ is the generator of the flow, and $E_s$ and $E_u$ are respectively the stable and the unstable vector bundles.

Let $\rho \colon \pi_1(M_0) \to \Z^d$ be a (surjective) representation. This defines a $\Z^d$-cover
\begin{equation}
\label{equation:zd-cover} 
\pi \colon M \to M_0
\end{equation}
by considering the associated bundle $M:=\widetilde{M}_0 \times_{\rho} \Z^d$, where $\widetilde{M}_0$ denotes the universal cover of $M_0$. We will refer to this construction as an Abelian cover (see \S\ref{ssection:abelian-covers} for further details). The flow $(\varphi_t^0)_{t \in \R}$ can be lifted to a $\Z^d$-equivariant flow $(\varphi_t)_{t \in \R}$ on $M$. Similarly, the measure $\vol_{M_0}$ can be lifted to a smooth flow-invariant measure $\vol_{M}$ on $M$. However, note that $\vol_{M}$ has infinite volume if $d > 0$. Unless the context is clear, and there is no risk of confusion, objects on $M_0$ will be distinguished by an index $0$.

Let $\alpha \in C^0(M_0,T^*M_0)$ be the Anosov $1$-form on $M_0$ defined by $\alpha(E_s \oplus E_u) = 0$ and $\alpha(X_{M_{0}})=1$. It is merely continuous, but $d\alpha$ is well-defined as a distributional $2$-form. In particular, we have that $d\alpha = 0$ if and only if $E_s \oplus E_u$ is jointly integrable (see e.g. \cite[Section 4.2.2]{Cekic-Lefeuvre-24}).

We are interested in the decay of correlations for the flow $(\varphi_t)_{t \in \R}$ with respect to the measure $\vol_{M}$ on the Abelian cover $M$. The following statement relies on a scale of spaces $(B^{s,r}(M))_{s,r \geq 0}$ on the Abelian cover $M$. To avoid burdening the introduction with technical details, we postpone their precise definition to \S\ref{ssection:functional-spaces}. For the time being, let us simply note that functions in $B^{s,r}$ are locally $H^s$ and exhibit increasing decay at infinity in $M$ as the parameter $r$ grows. Our first result is the following:

\begin{theorem}[Decay of correlations on Abelian covers]
\label{theorem:main1}
Let $\pi \colon M \to M_0$ be a $\Z^d$-cover corresponding to a surjective representation $\rho \colon \pi_1(M_0) \to \Z^d$ as in \eqref{equation:zd-cover}, and let $(\varphi_t)_{t \in \R}$ be a volume-preserving $\Z^d$-equivariant Anosov flow as above.

If $d\alpha \neq 0$, then there exist a constant $\kappa > 0$, and continuous bilinear forms $(C_j)_{j \geq 1}$ on $C^\infty_{\comp}(M)$ such that for all $f,g \in C^\infty_{\comp}(M)$, for all integers $N \geq 1$:
\begin{equation}
    \label{equation:decay-correlation-0}
\begin{split}
t^{d/2} \int_M f\circ \varphi_{-t} \cdot g~ \dd \vol_{M} =  \kappa \int_M f ~ \dd \vol_{M} \int_M g  ~\dd \vol_{M} + \sum_{j =1}^{N-1} t^{-j} C_j(f,g) + R_N(t,f,g)
\end{split}
\end{equation}
where the following estimates hold: for all $s > 0$, there exists a constant $C > 0$ such that
\begin{equation}
    \label{equation:pfiou}
\begin{split}
|C_j(f,g)| & \leq C\|f\|_{B^{s,2j}}\|g\|_{B^{s,2j}},\quad 1 \leq j \leq N - 1,\\
|R_N(t,f,g)| & \leq C \langle t \rangle^{-N}\left(\|f\|_{B^{\vartheta N,0}}\|g\|_{B^{\vartheta N,0}} + \|f\|_{B^{s,2N+d+1}}\|g\|_{B^{s,2N+d+1}}\right), \quad \forall t \geq 0,
\end{split}
\end{equation}
where the constant $\vartheta > 0$ is explicit and only depends on $\dim M_0$.

In addition, \eqref{equation:decay-correlation-0} holds more generally for all $f,g \in B^{\vartheta N, 0}(M) \cap B^{s,2N+d+1}(M)$.
\end{theorem}

That \eqref{equation:decay-correlation-0} holds for functions in $B^{\vartheta N,0}(M) \cap B^{s,2N+d+1}(M)$ is a straightforward consequence of the density of $C^\infty_{\comp}(M)$ in this space, together with the bounds \eqref{equation:pfiou}. The constant $\kappa > 0$ is explicit and given by a factor $(2\pi)^{d/2}$ times the inverse square root of the determinant of a covariance matrix, see \eqref{equation:expression-cj} applied with $j=0$. We establish that the bilinear forms $C_j$ are all non-zero for $j \geq 1$ (see Lemma \ref{lemma:cj-non-zero}), together with an explicit expression for $C_1$, see \eqref{equation:c1}.

One can also allow for a \emph{linear drift} in \eqref{equation:decay-correlation-0}, that is one can replace $f$ by $\tau_{\bk}^*f$ provided $\bk \in \Z^d$ is bounded by $\mc{O}(t^{1/2-\eps})$ for some $\eps > \eps_d > 0$, where $\eps_d$ depends on the order $d \geq 1$. In the regime $|\bk| \simeq t^{1/2-\eps}$, the second term in the expansion gets considerably simplified, see the discussion in \S\ref{ssection:linear-drift}.

The rate of mixing $d/2$ in the leading term $t^{-d/2}$ of the asymptotic expansion appears as half of the dimension of the connected component of the trivial representation of $\Z^d$ in its unitary dual $\widehat{\Z^d} \simeq \mathrm{U}(1)^d$ (the equivalence class of all unitary irreducible representations). This is  a general principle that should hold for any $G$-extension of a hyperbolic dynamical system, where $G$ is a non-compact Lie group, provided the trivial representation is not isolated (that is Kazhdan's property (T) does not hold). See \S\ref{ssection:related-results} for a related discussion.

Theorem \ref{theorem:main1} was established for the geodesic flow in negative curvature in \cite{Dolgopyat-Pene-Nandori-22}, see also \S\ref{ssection:related-results} where this is further discussed. We have here the sole assumption that $d\alpha \neq 0$. We believe that the argument could be adapted to treat any equilibrium state of $(\varphi_t^0)_{t \in \R}$; this is left for future investigation.

\subsection{Decay of correlations for isometric extensions of Anosov flows on Abelian covers} Theorem \ref{theorem:main1} is a consequence of a more general result on isometric extensions that we now explain. Let $p \colon P_0 \to M_0$ be a principal $G$-bundle where $G$ is a compact Lie group, and denote by $R_g \colon P_0 \to P_0$ the fiberwise right-action of $g \in G$ on $P_0$. Let $(\psi_t^0)_{t \in \R}$ be a $G$-extension of the Anosov flow $(\varphi_t^0)_{t \in \R}$ to $P_0$ in the sense that it satisfies the following two properties:
\[
\varphi_t^0 \circ p = p \circ \psi_t^0, \qquad R_g \circ \psi_t^0 = \psi_t^0 \circ R_g,\quad t \in \mathbb{R}, g \in G.
\]

The flow $(\psi_t^0)_{t \in \R}$ is called an \emph{isometric extension} of $(\varphi_t^0)_{t \in \R}$ and is a partially hyperbolic flow, see \cite[Chapter 12]{Lefeuvre-book} for instance. It preserves a smooth probability measure $\nu_0$ obtained (locally) as the product of $\vol_{M_{0}}$ with the Haar measure on $G$.

Associated with $(\psi_t^0)_{t \in \R}$, there is an important object on $P_0$, called the \emph{dynamical connection}. This is a $G$-equivariant continuous connection $\nabla^{\mathrm{dyn}}$ that is functorially induced by $(\psi_t^0)_{t \in \R}$. For precise definitions and further details we refer to \S\ref{ssection:dynamical-connection} below; we now state a few of its key properties. The \emph{transitivity group} $H \leqslant G$ (or the \emph{Brin group}) of $(\psi_t^0)_{t \in \R}$ is the closure of the holonomy group of the dynamical connection; it is well-defined up to conjugation in $G$. This group plays a key role as it describes the ergodic components of $(\psi_t^0)_{t \in \R}$. More precisely, $(\psi_t^0)_{t \in \R}$ is ergodic with respect to $\nu_0$ if and only if $H = G$ (see \cite{Brin-75-1,Brin-75-2,Lefeuvre-23}).

Let $\mathfrak{g}$ be the Lie algebra of $G$, and denote by $\mathrm{Ad} \colon G \to \mathrm{End}(\mathfrak{g})$ the adjoint representation. Let $\mathrm{Ad}(P_0) = P_0 \times_{\mathrm{Ad}} \mathfrak{g} \to M_0$ be the adjoint bundle associated with $P_0$. By the general structure theory of compact Lie groups, the Lie algebra $\fg$ of $G$ splits into
\[
\fg = \fz \oplus [\fg,\fg],
\]
where $\fz \simeq \R^a$ for some $a \geq 0$ is the Lie algebra of the center of $G$, and the derived algebra $[\fg,\fg]$ corresponds to the Lie algebra of the semisimple part of $G$. The previous decomposition being invariant by the $\mathrm{Ad}(G)$ action, it implies a decomposition for the vector bundle $\mathrm{Ad}(P_0)$ into
\[
\mathrm{Ad}(P_0) = M_0 \times \fz \oplus E,
\]
where $E$ is the vector bundle corresponding to the semisimple part. The curvature of the dynamical connection is an $\mathrm{Ad}(P_0)$-valued distributional $2$-form on $M_0$
\[
F_{\nabla^{\mathrm{dyn}}} \in \mc{D}'(M_0, \Lambda^2 T^*M_0 \otimes \mathrm{Ad}(P_0)).
\]
Projecting $F_{\nabla^{\mathrm{dyn}}}$ onto the trivial bundle $M_0 \times \fz$, and taking an arbitrary trivializing frame, we obtain $a$ distributional $2$-forms $F_1,\dotsc , F_a \in \mc{D}'(M_0, \Lambda^2 T^*M_0)$.

Let $P := \pi^*P_0 \to M$ be the pullback bundle over $M$ (we recall $\pi \colon M \to M_0$ is the $\Z^d$-cover). Note that $P$ is also a $(G \times \Z^d)$-bundle over $M_0$, and $\mathbb{Z}^d$-bundle over $P_0$ (see \S \ref{subsection:topology-abelian-covers} for further details), that is the following commutative diagram holds:
\[
\begin{tikzcd}
P = \pi^*P_0 \arrow{r}{} \arrow[swap]{dr}{} \arrow[swap]{d}{} & P_0 \arrow{d}{p} \\
M \arrow{r}{\pi} & M_0
\end{tikzcd}
\]
The flow $(\psi_t^0)_{t \in \R}$ can be lifted equivariantly to a flow on $P$, denoted by $(\psi_t)_{t \in \R}$. Similarly, the measure $\nu_0$ can be lifted to a measure $\nu$ (of infinite volume when $d > 0$). In the following result, given $f \in C^\infty_{\comp}(P)$, we let $f_{\bk=\mathbf{0}} \in C^\infty_{\comp}(M)$ be the function obtained by averaging in the $G$-fibers of $P$, namely
\[
f_{\bk=\mathbf{0}}(x) := \int_{P_x} f(w \cdot g) \dd g, \qquad x \in M
\]
where $w \in P_x$ is arbitrary and $\dd g$ denotes the normalized Haar measure on $G$. The meaning of the index $\bk=\mathbf{0}$ will become clear later, when introducing the Borel-Weil calculus on principal bundles, see \S\ref{ssection:bw-calculus}. We shall establish the following:

\begin{theorem}
\label{theorem:main2}
Let $P_0 \to M_0$ be a $G$-principal bundle where $G$ is a compact Lie group, and $P := \pi^*P_0 \to M$ be the pullback bundle on the $\Z^d$ Abelian cover $\pi \colon M \to M_0$, equipped with the lifted flow $(\psi_t)_{t \in \R}$ as defined above.

If the transitivity group $H$ is equal to $G$, and $(d\alpha, F_1,\dotsc,F_a)$ are linearly independent, then for all $f,g \in C^\infty_{\comp}(P)$, for all integers $N \geq 1$:
\begin{equation}
    \label{equation:decay-correlation-1}
\begin{split}
t^{d/2}\int_P f\circ \psi_{-t} \cdot g~ \dd \nu = \kappa   \int_P f \dd \nu \int_P g \dd \nu + \sum_{j =1}^{N-1} t^{-j} C_j(f_{\bk=\mathbf{0}},g_{\bk=\mathbf{0}}) + R_N(t,f,g),
\end{split}
\end{equation}
where for all $s > 0$, there exists $C > 0$ such that
\[
R_N(t,f,g) \leq C \langle t \rangle^{-N}\left(\|f\|_{B^{\vartheta N,0}(P)}\|g\|_{B^{\vartheta N,0}(P)} + \|f_{\bk = \mathbf{0}}\|_{B^{s,2N+d+1}(M)}\|g_{\bk = \mathbf{0}}\|_{B^{s,2N+d+1}(M)}\right),
\]
and the constant $\vartheta > 0$ is explicit and only depends on $\dim M_0$ and $\dim G$.

In addition, \eqref{equation:decay-correlation-1} holds more generally for all $f,g \in B^{\vartheta N,0}(P)$ such that $f_{\bk = \mathbf{0}}, g_{\bk = \mathbf{0}} \in B^{s,2N+d+1}(M)$.
\end{theorem}

The constant $\kappa$, as well as the bilinear forms $C_j$, are the same as the ones appearing in Theorem \ref{theorem:main1}. Also note that
\[
\int_P f ~\dd \nu \int_P g ~\dd \nu = \int_M f_{\bk=\mathbf{0}} ~\dd \vol_{M} \int_M g_{\bk=\mathbf{0}} ~\dd \vol_{M}.
\]
All relevant terms in the expansion \eqref{equation:decay-correlation-1} thus depend only on the behavior of the $0$-th Fourier mode of $f$ and $g$.

\subsection{Application to frame flows}

We now apply the previous result to the frame flow. Let $(N_0,g_0)$ be a smooth closed connected oriented Riemannian $n$-manifold of negative sectional curvature. Let $M_0 := SN_0$ and let $(\varphi_{t}^{0})_{t \in \mathbb{R}}$ denote the geodesic flow on the unit tangent bundle $SN_0$. Let $P_0 := FN_0$ be the frame bundle over $N_0$. It can be seen as a principal $\mathrm{SO}(n)$-bundle over $N_0$, or as a principal $\mathrm{SO}(n-1)$-bundle over $SN_0$. The frame flow is defined on $N_0$ by:
\[
\psi_{t}^{0}(v, w) := (\varphi_{t}^{0}(v), \tau_{v \to \varphi_{t}^{0} v} w),\quad t \in \mathbb{R},
\]
where $w$ is a frame in the orthogonal complement of $v$ and $\tau_{v \to \varphi_{t}^{0} v}$ denotes the parallel transport along the geodesic $(\varphi_{s}^{0} v)_{s \in [0,t]}$ with respect to the Levi-Civita connection.

Let $\rho \colon \pi_1(N_0) \to \Z^d$ be a surjective representation and $N := \widetilde{N_0}\times_\rho \Z^d$ be the corresponding Abelian cover. Let $g$ be the equivariant lift of $g_0$ to $N$, and $P := FN$ be the frame bundle of $(N,g)$. An immediate consequence of Theorem \ref{theorem:main2} is the following result:

\begin{corollary}
\label{theorem:frame-flows}
Assume that the frame flow is ergodic on $N_0$. Then there exist $\kappa > 0$, and continuous bilinear forms $(C_j)_{j \geq 1}$ on $C^\infty_{\comp}(SN)$, such that for all $f,g \in C^\infty_{\comp}(FN)$, for all integers $N \geq 1$:
\begin{equation}
    \label{equation:decay-correlation-2}
\begin{split}
t^{d/2}\int_{FN} f\circ \psi_{-t} \cdot g~ \dd \nu = \kappa  \int_{FN} f \dd \nu \int_{FN} g \dd \nu + \sum_{j =1}^{N-1} t^{-j} C_j(f_{\bk=\mathbf{0}},g_{\bk=\mathbf{0}}) + R_N(t,f,g),
\end{split}
\end{equation}
where $C_j$ satisfies the bounds \eqref{equation:pfiou}, and for all $s > 0$, there exists $C > 0$ such that for all $t \geq 0$:
\[
R_N(t,f,g) \leq C \langle t \rangle^{-N}\left(\|f\|_{B^{\vartheta N,0}(FN)}\|g\|_{B^{\vartheta N,0}(FN)} + \|f_{\bk = \mathbf{0}}\|_{B^{s,2N+d+1}(SN)}\|g_{\bk = \mathbf{0}}\|_{B^{s,2N+d+1}(SN)}\right).
\]
The constant $\vartheta > 0$ is explicit and only depends on $\dim N_0$.

In addition, \eqref{equation:decay-correlation-2} holds more generally for all $f,g \in B^{\vartheta N,0}(FN)$ such that $f_{\mathbf{k} = \mathbf{0}}, g_{\bk = \mathbf{0}} \in B^{s,2N+d+1}(SN)$.
\end{corollary}

Ergodicity of the frame flow is known when $n := \dim(N_0)$ is odd and different from $7$ \cite{Brin-Gromov-80}, or when the metric is $\delta(n)$-pinched\footnote{That is the sectional curvature is contained in an interval $[-C,-\delta(n)C)$ for some $C > 0$.} \cite{Cekic-Lefeuvre-Moroianu-Semmelmann-24} (see also \cite{Brin-Karcher-84, Burns-Pollicott-03}) and:
\begin{itemize}
\item $n=4$ or $n=4k + 2$ and $\delta(n)\sim 0.28$;
\item $n=4k$ and $\delta(n) \sim 0.55$;
\item $n=7$ and $\delta(7) \sim 0.5$.
\end{itemize}
See also the survey \cite{Cekic-Lefeuvre-Moroianu-Semmelmann-22} for further discussion.

\subsection{Related results} \label{ssection:related-results}

The result closest to Theorem \ref{theorem:main1} in the existing literature appears to be \cite{Dolgopyat-Pene-Nandori-22}, which addresses $\Z^d$-covers of hyperbolic flows. Specifically, in the context of the geodesic flow, Theorem \ref{theorem:main1} was first established in \cite[Theorem~4.1]{Dolgopyat-Pene-Nandori-22}. However, the general non-integrability condition $d\alpha \neq 0$ does not explicitly appear in their analysis, up to our understanding.

The study of decay of correlations or central/local limit theorems for Abelian extensions (i.e. $\R^d$- or $\Z^d$-extensions) of chaotic dynamical systems such as Markov chains, Anosov maps, or Sinai billiards is now a classical topic in the field, see \cite{Balint-Bruin-Terhsiu-23, Guivarch-Hardy-88,Guivarch-89,Szasz-Varju-04,Gouezel-11,Iwata-08,Melbourne-Terhesiu-13,Liverani-Terhesiu-16,Aaronson-Nakada-17,Dolgopyat-Nandori-17a,Oh-Pan-17,Melbourne-Terhesiu-17,Dolgopyat-Nandori-19,Pene-19,Terhesiu-19,Dolgopyat-Nandori-25} among other references. See also \cite{Bruin-Fougeron-Terhesiu-24, Castorrini-Ravotti-24} for results about asymptotic expansions of ergodic averages on Abelian extensions of parabolic dynamical systems.

However, Theorem \ref{theorem:main2} which deals with $(G \times \Z^d)$-extensions, where $G$ is a compact Lie group, and derives a full asymptotic expansion of the correlation function, seems completely new.

We also point out that $\Z^d$ could be replaced by $\R^d$ at the expense of adding a new semiclassical parameter $\lambda \in \R^d$, thus replacing the parameter $\btheta \in \mathrm{U}(1)^d$ in the analysis developed in \S\ref{section:abelian-extension-anosov} and \S\ref{section:decay-isometric}. This is due to the fact that the unitary dual of $\R^d$ is $\R^d$ itself, which is noncompact (unlike $\widehat{\Z^d}\simeq \mathrm{U}(1)^d)$. The problem of understanding decay of correlations would then be more involved, as one would have to prove uniform high-frequency estimates as $|\lambda| \to +\infty$.

The study of $\R^d$- or $(\R^d \times G)$-extensions (with $G$ compact) has emerged recently in the context of Anosov representations, see \cite{Thirion-09,Sambarino-15,Chow-Sarkar-23O,Bonthonneau-Lefeuvre-Weich-26}. In this setting, one considers a discrete subgroup $\Gamma < \mathbf{G}$ of a noncompact semisimple Lie group $\mathbf{G}$, which generalizes convex co-compact subgroups of $\mathrm{Isom}(\HH^n) \simeq \mathrm{SO}(n,1)$. The associated dynamical system is a multiflow (a hyperbolic $\R^k$-action), which can be viewed as a higher-rank analogue of the geodesic flow on the unit tangent bundle of a hyperbolic manifold. By fixing a direction for the $\R^k$-action (i.e. restricting to an $\R$-factor), one obtains the so-called \emph{refraction flow}.

The first three articles \cite{Thirion-09,Sambarino-15,Chow-Sarkar-23O} only derive the leading term in the asymptotic expansion of the correlation function for the refraction flow. However, \cite{Bonthonneau-Lefeuvre-Weich-26} achieves a full asymptotic expansion, analogous to the results presented in Theorems \ref{theorem:main1} and \ref{theorem:main2}. In this context, the mixing rate $d/2$ corresponds to half of $\mathrm{rank}(\mathbf{G}) - 1$.

\subsection{Organization of the article} In \S\ref{section:abelian-covers}, we recall standard facts on Abelian covers, Floquet theory, and isometric extensions of Anosov flows. In \S\ref{section:semiclassical}, we provide the necessary background in semiclassical analysis required for the proofs in subsequent sections. In particular, we recall the main results on the Borel-Weil calculus developed in \cite{Cekic-Lefeuvre-24}, which allows to treat equivariant (pseudo)differential operators on $G$-bundles. In \S\ref{section:abelian-extension-anosov}, we prove Theorem \ref{theorem:main1}. Finally, Theorem \ref{theorem:main2} is established in \S\ref{section:decay-isometric}. \\

\noindent \textbf{Acknowledgement:} This project is supported by the European Research Council (ERC) under the European Union’s Horizon 2020 research and innovation programme (Grant agreement no. 101162990 — ADG). 

\section{Abelian covers and isometric extensions}

\label{section:abelian-covers}

\subsection{Abelian covers}

\label{ssection:abelian-covers}

Let $M_0$ be a smooth closed manfold, $x_0 \in M_0$ be an arbitrary fixed point, and let
\begin{equation}
    \label{equation:representation-rho}
    \rho \colon \pi_1(M_0,x_0) \to \Z^d
\end{equation}
be a surjective representation onto $\Z^d$ for some $d \geq 1$. Let
\[
M := \widetilde{M}_0 \times_{\rho} \Z^d
\]
be the associated Abelian cover, where $\widetilde{M}_0$ denotes the universal cover of $M_0$. That is $M = (\widetilde{M}_0 \times \Z^d)/\sim$ where

\[
    (x,\mathbf{n}) \sim (\gamma.x, - \rho(\gamma)+\mathbf{n}),\quad \forall x \in \widetilde{M}_0, \n \in \Z^d, \gamma \in \pi_1(M_0, x_0).
\]
Note that $M$ is connected as $\rho$ is surjective. Additionally, the representation \eqref{equation:representation-rho} factors through 
\[
    \pi_1(M_0,x_0) \xrightarrow{\rho_{\mathrm{ab}}} H_1(M_0, \Z) \to H_1(M_0,\Z)/\mathrm{Tor} \to \Z^d,
\]
where $\rho_{\mathrm{ab}}$ denotes the Abelianization map and $\mathrm{Tor}$ denotes the torsion part, see \cite{Hatcher-02}. Denote the induced map $H_1(M_0, \Z) \to \mathbb{Z}^d$ by $\rho'$; then for \emph{any} closed loop $\gamma \subset M_0$ (not necessarily based at $x_0$) we can define $\rho(\gamma) := \rho'([\gamma]_{H_1})$ where $[\gamma]_{H_1}$ denotes the class of $\gamma$ in $H_1(M_0, \mathbb{Z})$.

The manifold $M$ is a principal $\Z^d$-bundle over $M_0$, namely there is a surjective projection $\pi \colon M \to M_0$ such that the preimage of each point can be identified with $\Z^d$, and it is equipped with a $\Z^d$-action $\tau \colon M \times \Z^d \to M$. On the universal cover, this action is given by 
\[
    \tau_\kk(x,\n) = (x,\n + \kk),\quad \kk, \n \in \mathbb{Z}^d, x \in \widetilde{M}_0,
\]
and it is immediate to verify that it descends to the quotient $M$.

We will need the following lemma:

\begin{lemma}
Let $\widehat{x}_0 \in M$ be an arbitrary preimage of $x_0 \in M_0$ by the projection map $\pi$, and let $\gamma \subset M$ be a loop based at $\widehat{x}_0$. Then $\pi_* \gamma \in \ker \rho$.
\end{lemma}

\begin{proof}
Write $\widehat{x}_0=[(\widetilde{x}_0,\mathbf{n}_{0})] \in M$ where $\widetilde{x}_0 \in \widetilde{M}_0$ is a lift of $x_{0}$ by $p \colon \widetilde{M}_0 \to M_{0}$. Consider $\pi_{*}(\gamma)$, where $\gamma$ is a loop in $M$ based at $\widehat{x}_0$, and write $\gamma(t)=[(\widetilde{\gamma}(t),\mathbf{n}(t))]$, where $\widetilde{\gamma}$ and $\mathbf{n}$ are paths in $\widetilde{M}_0$ and $\mathbb{Z}^{d}$ respectively, such that $\widetilde{\gamma}(0)=\widetilde{x}_0$ and $\mathbf{n}(0)=\mathbf{n}_{0}$ (here we use the path lifting property, see \cite{Hatcher-02}). Note that since $\Z^{d}$ is discrete, we have $\mathbf{n}(t) \equiv \mathbf{n}_{0}$. Now, 
\[
[(\widetilde{x}_0,\mathbf{n}_{0})]=\gamma(0)=\gamma(1)=[(\widetilde{\gamma}(1),\mathbf{n}_{0})]=[(\gamma. \widetilde{x}_0,\mathbf{n}_{0})]=[(\widetilde{x}_0,\rho(\gamma)+\mathbf{n}_{0})],
\]
where the first three equalities follow since $\gamma$ is a loop based at $\widehat{x}_0$, the fourth one follows from the fact that deck transformations on $\widetilde{M}_0$ are given by $\pi_{1}(M_{0},x_{0})$ (the endpoints of lifted paths on the universal cover are given by the action of the corresponding element in the fundamental group on the initial point of the lifted path), and we used the definition of $M$ in the last step. The equality implies that $(\widetilde{x}_0,\n_0) = (\omega. \widetilde{x}_0, -\rho(\omega)+\rho(\gamma)+\n_0)$ for some $\omega \in \pi_1(M_0, x_0)$. Since $\pi_1(M_0, x_0)$ acts freely, this forces that $\omega$ is the neutral element, and $\rho(\gamma)=0$.
\end{proof}

Let $\mathrm{U}(1) := \R/2\pi\Z$. Given $\btheta \in \mathrm{U}(1)^d$, we introduce the representation
\[
    \alpha_{\btheta} \colon \Z^d \to \mathrm{U}(1), \qquad \alpha_{\btheta}(\n) :=  \btheta \cdot \n ~ \text{ mod } 2\pi.
\]
We will write $H^1_{\mathrm{dR}}(M_0, \mathbb{R})$ for the first de Rham cohomology group of $M_0$, and denote by $H^1_{\mathrm{dR}}(M_0, 2\pi \mathbb{Z})$ the lattice inside of it, whose elements $[\eta]$ satisfy for any closed loops $\gamma \subset M_0$
\[
    \int_\gamma [\eta] \in 2\pi \mathbb{Z}.
\]

\begin{proposition}\label{prop:eta-theta}
    There exists a smooth map
    \[
    \mathrm{U}(1)^d \ni \btheta \mapsto [\eta_{\btheta}] \in H^1_{\mathrm{dR}}(M_0, \mathbb{R})/H^1_{\mathrm{dR}}(M_0, 2\pi \mathbb{Z})
    \]
    such that for all smooth closed curves $\gamma \subset M_0$:
    \begin{equation}
        \label{equation:eta-theta}
            \int_\gamma [\eta_{\btheta}] = \rho(\gamma) \cdot \btheta ~ \mod 2\pi.
    \end{equation}
    Moreover, the dependence on $\btheta$ is linear in the sense that the map $\btheta \mapsto [\eta_{\btheta}]$ descends from a linear map $\mathbb{R}^d \ni \btheta \mapsto [\widetilde{\eta}_{\btheta}] \in H^1_{\mathrm{dR}}(M_0, \mathbb{R})$, and so is in particular smooth. As a consequence, for any $v \in T_{\btheta} \mathrm{U}(1)^d \cong \mathbb{R}^d$, we have $D[\eta_{\btheta}](v) = [\widetilde{\eta}_{v}]$.
\end{proposition}
\begin{proof}
    Note that as $\mathbb{Z}^d$ is Abelian $\rho$ factors as $\rho = \rho' \circ \rho_{ab}$
    \[
        \pi_1(M_0, x_0) \to H_1(M_0, \mathbb{Z}) \to \mathbb{Z}^d.
    \]
    where we denote by $\rho_{\mathrm{ab}}$ the Abelianisation map, and by $\rho'$ the induced map $H_1(M_0, \mathbb{Z}) \to \mathbb{Z}^d$. For any $\btheta \in \mathbb{R}^d$, consider the homomorphism $\beta_{\btheta} \colon \mathbb{Z}^d \to \mathbb{R}$ given by $\beta_{\btheta}(\mathbf{n}) := \btheta \cdot \mathbf{n}$. Then
    \[
        \beta_{\btheta} \circ \rho' \in \mathrm{Hom}(H_1(M_0, \mathbb{Z}), \mathbb{R}) \cong H^1_{\mathrm{dR}}(M_0, \mathbb{R}),
    \]
    where the isomorphism is given by de Rham's theorem, and so there exists a unique first cohomology class $[\widetilde{\eta}_{\btheta}]$ that equals $\beta_{\btheta} \circ \rho'$ in this correspondence, i.e. by definition, for any loop $\gamma \subset M_0$ we have
    \begin{equation}\label{eq:de-rham}
        \int_\gamma [\widetilde{\eta}_{\btheta}] = \beta_{\btheta} \circ \rho'([\gamma]_{H^1}) = \btheta \cdot \rho([\gamma]_{\pi_1}),
    \end{equation}
    where in the last equality we write $\rho_{\mathrm{ab}}([\gamma]_{\pi_1}) = [\gamma]_{H^1}$. Moreover, the map $\mathbb{R}^d \ni \btheta \mapsto [\widetilde{\eta}_{\btheta}] \in H^1_{\mathrm{dR}}(M_0, \mathbb{R})$ is linear since $\btheta \mapsto \beta_{\btheta}$ is linear and de Rham's isomorphism is linear. We next observe that by \eqref{eq:de-rham}, if $\btheta \in (2\pi \mathbb{Z})^d$, then $[\widetilde{\eta}_{\btheta}] \in H^1(M_0, 2\pi \mathbb{Z})$; therefore the map $\btheta \mapsto [\widetilde{\eta}_{\btheta}]$ descends to a map of tori. Finally, the derivative claim follows from linearity, and the fact that the derivative can be identified as a section of endomorphisms between $\mathbb{R}^d$ and $H^1_{\mathrm{dR}}(M_0, \mathbb{R})$. 
\end{proof}

\begin{remark}\rm
    We remark that in general the map $\btheta \mapsto [\eta_{\btheta}]$ does not lift to a map to $H^1_{\mathrm{dR}}(M_0, \mathbb{R})$ (unless $H^1_{\mathrm{dR}}(M_0, \mathbb{R}) = \{0\}$ or $d = 0$ but in this case the claim is trivial). Indeed, the existence of such a lift is equivalent to $[\widetilde{\eta}_{\btheta}] = 0$ in $H^1_{\mathrm{dR}}(M_0, \mathbb{R})$ for all $\btheta \in (2\pi \mathbb{Z})^d$. Equivalently, by \eqref{eq:de-rham} we have $\rho(\gamma) = 0$ for all closed loops $\gamma$, which contradicts the fact that $\rho$ is surjective.
\end{remark}

We now introduce sections that will eventually enable us to parametrize Fourier modes in $\btheta$ and write them as functions on $M_0$.

\begin{lemma} \label{lemma:stheta}
Let $\eta_{\btheta}$ be a smooth closed $1$-form representing the class of $[\eta_{\btheta}]$ from Proposition \ref{prop:eta-theta}. We then introduce
\begin{equation}
\label{equation:stheta}
    s_{\btheta}(x) := \exp\left(i \int_\gamma \eta_{\btheta} \right), \qquad x \in M,
\end{equation}
where $\gamma := \pi(\gamma')$ and $\gamma'$ is an arbitrary path joining $\widehat{x}_0$ to $x$ in $M$. Then $s_{\btheta}$ is well-defined, and  
\[
    s_{\btheta} \circ \tau_\n = e^{i \n \cdot \btheta} s_{\btheta},\quad s_{\btheta}^{-1} d s_{\btheta} = i \pi^* \eta_{\btheta},\quad \forall \mathbf{n} \in \mathbb{Z}^d.
\]
\end{lemma}

\begin{proof}
Firstly, we show that $s_{\btheta}$ is well defined, i.e. independent of the choice of path $\gamma$. Let $x \in M$, and consider two paths $\gamma_{1},\gamma_{2}$ from $\widehat{x}_0$ to $x$. Since $\pi_*(\gamma_{2}^{-1} \cdot \gamma_{1}) \in \ker \rho$, we find that 
\[ \int_{\pi(\gamma_{2}^{-1} \cdot \gamma_{1})} \eta_{\btheta} \equiv \rho (\pi_{*}(\gamma_{2}^{-1} \cdot \gamma_{1})) \cdot \btheta \equiv 0,
\]
where $\equiv$ stands for equality modulo $2\pi$. This shows that $s_{\btheta}$ is well defined. 

We then study how $s_{\btheta}$ interacts with $\tau_{\mathbf{n}}$. Since we have the freedom to choose any path joining $\widehat{x}_0$ with $\tau_{\n}(x)$, we will take a path which is the result of the concatenation of $\gamma_{1}$ joining $\widehat{x}_0$ to $x$, with the path $\gamma_{2}$ from $x$ to $\tau_{\n}(x)$. Although we can choose $\gamma_{1}$ arbitrary, we will take $\gamma_{2}$ in a very specific way. If we write $y=\pi(x)$, the surjectivity of $\rho \colon \pi_{1}(M_{0},y) \to \Z^{d}$ gives the existence of a path $\gamma \in \pi_{1}(M_{0},y)$ with $\rho(\gamma)=\n$. Let $\gamma'$ be a lift of $\gamma$ to $M$, with base point $x$. We claim that $\gamma'(1)=\tau_{\n}(x)$. Indeed, let us write $x=[(\widetilde{y},\kk)]$, where $\widetilde{y} \in \widetilde{M}_0$ is such that $p(\widetilde{y})=y$. Hence, $\gamma'(t)=[(\widetilde{\gamma}(t),\kk)]$, for some $\kk \in \Z^{d}$, and $\widetilde{\gamma}$ a path in $\widetilde{M}_0$. Then, using that the endpoints of paths in $\widetilde{M}_0$ are given by the action of the corresponding element on $\pi_{1}(M_{0},y)$ in the initial point, we have
\[
\gamma'(1)=[(\gamma.\widetilde{y}, \mathbf{k})]=[(\widetilde{y},\rho(\gamma)+\kk)]=[(\widetilde{y},\kk+\n)]=\tau_{\n}(x).
\]
Hence, to compute $s_{\btheta}(\tau_{\n}(x))$, we will use any path $\gamma_{1}$ from $x_{0}$ to $x$, and we take $\gamma_{2}=\gamma'$ as above, which joins $x$ with $\tau_{\n}(x)$. Therefore, 
\[ 
s_{\btheta}(\tau_{\n}(x))=\exp \left(i \int_{\pi(\gamma')} \eta_{\btheta} \right) \exp \left(i \int_{\pi(\gamma_{1})} \eta_{\btheta} \right)=\exp(i \rho(\pi(\gamma'))\cdot \btheta)s_{\btheta}(x)=e^{i\n\btheta}s_{\btheta}(x).
\]
Finally, the last relation on the statement of the lemma follows directly by differentiation.
\end{proof}

Unfortunately, the previous construction depends on the choice of the representative $\eta_{\btheta}$ and so is not globally defined. We can do better over contractible open sets.

\begin{proposition}\label{prop:contractible}
Let $U \subset \mathrm{U}(1)^d$ be an open contractible set. Then there exists a smooth map
\[
    U \ni \btheta \mapsto \eta_{\btheta} \in C^\infty(M_0, T^*M_0) \cap \ker d, 
\]
whose de Rham cohomology class agrees with the one of $[\widetilde{\eta}_{\btheta}]$ constructed in Proposition \ref{prop:eta-theta}.  Moreover, if $\mathbf{0} \in U$, we may assume that $\eta_{\mathbf{0}} \equiv 0$, and we can assume that $(\eta_{\btheta})_{\btheta \in U}$ are harmonic with respect to a background Riemannian metric $g_0$ on $M_0$. Consequently, the conclusions of Lemma \ref{lemma:stheta} are valid with smooth dependence on $\btheta \in U$.
\end{proposition}
\begin{proof}
For simplicity denote the map $\btheta \mapsto [\eta_{\btheta}]$ constructed in Proposition \ref{prop:eta-theta} by $\Xi$ and the linear map on the cover as $\widetilde{\Xi} \colon \mathbb{R}^d \to H^1_{\mathrm{dR}}(M_0, \mathbb{R})$. Notice that $\widetilde{\Xi}$ is injective: indeed, if $\widetilde{\Xi}(\btheta) = 0$, then by \eqref{eq:de-rham}, and since $\rho$ is by assumption surjective, we get $\btheta = 0$. By invariance of rank, we thus must have $\widetilde{\Xi}((2\pi \Z)^d) = H^1(M_0, 2\pi \Z) \cap \widetilde{\Xi}(\R^d)$, and so it follows that $\Xi$ is an embedding and that $\Xi(\mathrm{U}(1)^d) = \widetilde{\Xi}(\mathbb{R}^d)/\widetilde{\Xi}(2\pi \mathbb{Z}^d)$ is a sub-torus of $J_T = H^1(M_0, \mathbb{R})/H^1(M_0, 2\pi \mathbb{Z})$ of dimension $d$. We conclude that $\Xi(U)$ is open and contractible inside $\Xi(\mathrm{U}(1)^d)$. We may restrict the $H^1_{\mathrm{dR}}(M_0, 2\pi \Z)$-bundle (or covering space) $H^1_{\mathrm{dR}}(M_0, \mathbb{R})$ to $\Xi(\mathrm{U}(1)^d)$; since $\Xi(U) \subset \Xi(\mathrm{U}(1)^d)$ is contractible there exists a local smooth section $S_{\Xi(U)} \colon \Xi(U) \to H^1_{\mathrm{dR}}(M_0, \mathbb{R})$ of that bundle. Let $g_0$ be an arbitrary Riemannian metric on $M_0$ and denote by $\Pi_{\mc{H}^1}$ the Hodge projection onto $g_0$-harmonic $1$-forms; we will use Hodge theory. Using the notation of Proposition \ref{prop:eta-theta}, we then set for $\btheta \in U$, $\eta_{\btheta} := \Pi_{\mc{H}^1} S_{\Xi(U)} ([\eta_{\btheta}])$ which clearly satisfies the required properties. Finally, if $\mathbf{0} \in U$, the claim about $\eta_{\mathbf{0}}$ can be achieved by replacing $\eta_{\btheta}$ by $\eta_{\btheta} - \eta_{\mathbf{0}}$. This completes the proof.
\end{proof}

\subsection{Floquet theory} \label{ssection:floquet-theory} For $\btheta \in \mathrm{U}(1)^d$ we introduce the associated line bundles
\[
    L_{\btheta} := \widetilde{M}_0 \times_{e^{-i \alpha_{\btheta} \circ \rho}} \mathbb{C} = M \times_{e^{-i\alpha_{\btheta}}} \mathbb{C},
\]
where as usual we have 
\[
    M \times_{e^{-i\alpha_{\btheta}}} \C = M \times \C/\sim,\quad (x, z) \sim (\tau_{\n} x, e^{i \btheta \cdot \n}z), \forall \n \in \Z^d.
\]
Observe that $L_{\btheta}$ is a flat Hermitian (i.e. equipped with an inner product in its fibres and a compatible flat connection) line bundle with holonomy along a curve $\gamma \in \pi_1(M_0, x_0)$ given by
\[
\exp\left(-i \alpha_{\btheta} \circ \rho(\gamma)\right) = \exp\left(-i \btheta \cdot \rho(\gamma)\right).
\]
From the definition of $L_{\btheta}$, we see that the sections in $C^\infty(M_0, L_{\btheta})$ correspond bijectively to equivariant functions in
\begin{equation}
    \label{equation:equivariant-space}
    C^\infty_{\btheta}(M) := \{F \in C^\infty(M) \mid \forall \mathbf{n} \in \mathbb{Z}^d,\, F \circ \tau_\n = e^{i \btheta \cdot \n} F\},
\end{equation}
the space Fourier modes of frequency $\btheta$. Given a section $s \in C^\infty(M_0, L_{\btheta})$ we will denote its equivariant lift by $\widetilde{s} \in C^\infty_{\btheta}(M)$.

We now observe that the line bundles $L_{\btheta}$ are topologically trivial.

\begin{proposition}
\label{proposition:ltheta-trivial}
    Let $U \subset \mathrm{U}(1)^d$ be an open contractible set. Then, the family $(L_{\btheta})_{\btheta \in U}$ of line bundles is smoothly trivial, i.e. there exists a smooth family $(s_{\btheta})_{\btheta \in U}$ such that $s_{\btheta} \in C^\infty(M_0, L_{\btheta})$ is nowhere vanishing. 
\end{proposition}
    By a smooth family of sections $(s_{\btheta})_{\btheta \in U}$ here we mean that the family of equivariant lifts $\widetilde{s}_{\btheta} \in C^\infty_{\btheta}(M) \subset C^\infty(M)$ depend smoothly on $\btheta$.
\begin{proof}
    By Proposition \ref{prop:contractible}, there is a smooth family of $1$-forms $(\eta_{\btheta})_{\btheta \in U}$ on $M_0$ satisfying \eqref{equation:eta-theta}. By Lemma \ref{lemma:stheta}, there are equivariant functions $\widetilde{s}_{\btheta} \in C^\infty_{\btheta}(M)$ of pointwise unit norm, which therefore descend to the desired sections $s_{\btheta}$ of $L_{\btheta}$ of unit pointwise norm. This completes the proof.
\end{proof}
\begin{remark}\rm
    We note that the family of line bundles $\{L_{\btheta} \mid \btheta \in \mathrm{U}(1)^d\}$ is not trivial as a \emph{family}, i.e. we cannot choose a smooth family $\{s_{\btheta} \mid \btheta \in \mathrm{U}(1)^d\}$ of trivialising sections. Indeed, if there was such a family $(s_{\btheta})_{\btheta}$, then we would also get a smooth family of $1$-forms $(\eta_{\btheta})_{\btheta} \in C^\infty(M_0, T^*M_0)$ derived from the formula
    \[
        s_{\btheta}^{-1} ds_{\btheta} = i \pi^*\eta_{\btheta},
    \]
    see Lemma \ref{lemma:stheta}, which satisfy that for any loop $\gamma \subset M_0$
    \[
        \int_\gamma \eta_{\btheta} = \rho(\gamma) \cdot \btheta \mod 2\pi.
    \]
    Differentiating, we get $\int_\gamma D_v\eta_{\btheta} = \rho(\gamma)\cdot v$, where $D_v\eta_{\btheta}$ is seen as the directional derivative in the $\btheta$-variable in the direction of $v \in \mathbb{R}^d$. Let $\delta$ be the closed geodesic on the torus $\mathrm{U}(1)^d$ starting at the origin tangent to $v \in (2\pi \mathbb{Z})^d$ and of length $|v|$. Integrating this previous equality along $\delta$, we get $0 = \rho(\gamma)\cdot v$. This would imply that $\rho \equiv 0$, contradicting our assumption that $\rho$ is surjective.
\end{remark}

After choosing a representative $\eta_{\btheta}$ of $[\eta_{\btheta}]$ as in Proposition \ref{prop:eta-theta}, by Lemma \ref{lemma:stheta}, there exists an equivariant function $s_{\btheta} \in C_{\btheta}^\infty(M)$ of pointwise unit norm, and an isomorphism
\[
C^\infty(M_0) \to C^\infty_{\btheta}(M), \qquad f \mapsto F_{\btheta} := (\pi^*f) s_{\btheta}.
\]
Conversely, the inverse map is denoted by
\[
C^\infty_{\btheta}(M) \to C^\infty(M_0), \qquad F_{\btheta} \mapsto f_{\btheta},
\]
where $f_{\btheta} \in C^\infty(M_0)$ is the unique function such that $F_{\btheta}/s_{\btheta} = \pi^*f_{\btheta}$. We note that since $s_{\btheta}$ has pointwise unit norm, it is uniquely determined up to multiplication by a function of the form $\pi^*h$, where $h \in C^\infty(M_0)$ has pointwise unit norm; therefore the identification above is also unique up to multiplication by $\pi^*h$. Note also that $C^\infty_{\mathbf{0}}(M)$ corresponds to $\Z^d$-periodic functions, so they are pullbacks of functions on $M_0$, that is $C^\infty_{\mathbf{0}}(M) = \pi^* C^\infty(M_0)$. By Proposition \ref{prop:contractible}, we can do this procedure \emph{smoothly} for $\btheta$ varying in a contractible open set $U \subset \mathrm{U}(1)^d$.

We introduce the following $L^2$-norm on $C^\infty_{\btheta}(M)$ which coincides with the $L^2$ norm on sections of $L_{\btheta}$ under the identification above:
\[
    \|F_{\btheta}\|^2_{L^2_{\btheta}(M)} := \|F_{\btheta}\|_{L^2(M, L_{\btheta})}^2 = \|f_{\btheta}\|^2_{L^2(M_0)},
\]
where $L^{2}(M_{0})$ is with respect to the volume measure $\vol_{M_0}$. Here we abuse the notation slightly by identifying sections of $L_{\btheta}$ with equivariant functions in $C_{\btheta}^\infty(M)$. The space $L^2_{\btheta}(M)$ is then defined as the completion of $C^\infty_{\btheta}(M)$ with respect to this norm. 

We now express an arbitrary smooth function in terms of its Fourier modes.

\begin{proposition}\label{eq:fourier-theory}
    Let $f \in C^\infty_{\comp}(M)$. Then 
    \begin{equation}
        \label{equation:floquet-decomposition}
            f = \dfrac{1}{(2\pi)^d} \int_{\mathrm{U}(1)^d} F_{\btheta} ~\dd \btheta,
    \end{equation}
    where $F_{\btheta} \in C_{\btheta}^\infty(M)$ is defined as
\begin{equation}\label{eq:theta-frequency}
    F_{\btheta}(x) := \sum_{\n \in \Z^d} f(\tau_{\n}(x)) e^{-i \n \cdot \btheta}, \quad x \in M.
\end{equation}
Finally, we have the following Parseval identity: for all $f,g \in C^\infty_{\comp}(M)$,
\begin{equation}
\label{equation:parseval}
\langle f,g\rangle_{L^2(M)} = \dfrac{1}{(2\pi)^d} \int_{\mathrm{U}(1)^d} \langle F_{\btheta}, G_{\btheta}\rangle_{L^2(M_0, L_{\btheta})}\, \dd\btheta.
\end{equation}
\end{proposition}
Here, $L^{2}(M)$ is defined with respect to the measure $\vol_{M}$.
\begin{proof}
For the first part, note that the sum defining $F_{\btheta}$ converges because the $\mathbb{Z}^d$-action is proper (i.e. for each $x \in M$ and $K \subset M$ compact, there are only finitely many $\n \in K$ such that $\tau_{\n}x \in K$) and $f$ has compact support; it clearly satisfies $F_{\btheta} \in C^\infty_{\btheta}(M)$. In addition, \eqref{equation:floquet-decomposition} is easily verified.

For the second part, let $y_0 \in M_{x_0}$ be arbitrary. Then
\begin{align*}
    &\int_{M_{x_0}} f(y) \overline{g}(y)\, \dd y = \sum_{\mathbf{n} \in \mathbb{Z}^d} f(\tau_{\mathbf{n}} y_0) \overline{g}(\tau_{\mathbf{n}} y_0)\\
    &= \frac{1}{(2\pi)^{d}} \int_{\mathrm{U}(1)^d} F_{\btheta}(y_0) \overline{G}_{\btheta}(y_0)\, \dd \btheta = \dfrac{1}{(2\pi)^d} \int_{\mathrm{U}(1)^d} \langle{F_{\btheta}(x_0), G_{\btheta}(x_0)}\rangle_{L_{\btheta}(x_0)}\, \dd\btheta,
\end{align*}
where $F_{\btheta}$ and $G_{\btheta}$ are $\btheta$-frequency parts of $f$ and $g$ respectively (defined in \eqref{eq:theta-frequency}), and in the last equality we identified those with sections of $L_{\btheta}$. Since $x_0 \in M_0$ was arbitrary, integrating in $M_0$ completes the proof.
\end{proof}

We end this section with an auxiliary claim about the result of deriving $f_{\btheta}$ in $\btheta$ and the corresponding operation on $f$. We fix an open contractible set $U \subset \mathrm{U}(1)^d$ as above and work over $U$. We define
\[
    H_j(x) := \int_{\widehat{x}_0}^x \pi^* \partial_{\btheta_j} \eta_{\btheta},\quad x \in M,\quad 1 \leq j \leq d,
\]
where the integration is along an arbitrary path between the basepoint $\widehat{x}_0 \in M$ and $x$. That this is well-defined is checked in the same way as \eqref{equation:stheta}. Alternatively, we may observe that
\begin{equation}\label{eq:auxiliary-Hj}
    s_{\btheta}^{-1} \partial_{\btheta_j} s_{\btheta} = i H_j
\end{equation} showing directly that $H_j$ is well-defined. For a multi index $\alpha$ we may set $\mathbf{H}^\alpha := \prod_{j = 1}^d H_j^{\alpha_j}$. We then have

\begin{proposition}\label{prop:differentiation-btheta}
    For an arbitrary multi index $\alpha$ we have
    \[
        \partial^\alpha_{\btheta} f_{\btheta} = (-i)^{|\alpha|}(\mathbf{H}^\alpha f)_{\btheta},\quad f \in C^\infty_{\comp}(M),\quad \btheta \in U.
    \]
\end{proposition}
\begin{proof}
    It suffices to consider the case $|\alpha| = 1$ since the other cases follow by iteration. Fix an index $j$; then by differentiating the relation $s_{\btheta} \circ \tau_{\mathbf{n}} = e^{i \mathbf{n} \cdot \btheta} s_{\btheta}$ from Lemma \ref{lemma:stheta}, valid for any $\mathbf{n} \in \mathbb{Z}^d$, we get
    \[
        i H_j \circ \tau_{\mathbf{n}} \cdot s_{\btheta} \circ \tau_{\mathbf{n}} = \partial_{\btheta_j} s_{\btheta} \circ \tau_{\mathbf{n}} = i n_j e^{i \mathbf{n} \cdot \btheta} s_{\btheta} + e^{i\mathbf{n} \cdot \btheta} \partial_{\btheta_j} s_{\btheta} = i(n_j + H_j) e^{i \mathbf{n} \cdot \btheta} s_{\btheta},
    \]
    where we used \eqref{eq:auxiliary-Hj} in the first and last equalities. We conclude that
    \begin{equation}\label{eq:auxiliary-equivariance-Hj}
        H_j \circ \tau_{\mathbf{n}} = n_j + H_j,\quad \mathbf{n} \in \mathbb{Z}^d.
    \end{equation}
    By differentiating the equality $F_{\btheta} = s_{\btheta} \pi^*f_{\btheta}$ we get
    \begin{align*}
        s_{\btheta} \pi^* \partial_{\btheta_j} f_{\btheta} = \partial_{\btheta_j}F_{\btheta} - \partial_{\btheta_j} s_{\btheta} \pi^*f_{\btheta} = (-i n_j) F_{\btheta} - iH_j s_{\btheta}\pi^*f_{\btheta} =  (-i)(n_j + H_j) F_{\btheta} = -i (fH_j)_{\btheta},
    \end{align*}
    where in the last equality we used \eqref{eq:auxiliary-equivariance-Hj}. This completes the proof.
\end{proof}

\subsection{Topology of Abelian extensions} 
\label{subsection:topology-abelian-covers}

As above, $M_0$ denotes a smooth closed connected manifold and $\pi \colon M \to M_0$ a $\Z^d$-extension defined via a surjective representation $\rho \colon \pi_1(M_0,x_0) \to \Z^d$. Let $p \colon P_0 \to M_0$ be a $G$-principal bundle over $M_{0}$, where $G$ is a compact Lie group. Let $\widetilde{P}_{0}$, $\widetilde{M}_{0}$ be, respectively, the universal covers of $P_{0}$ and $M_{0}$. We introduce
\[
P:=\widetilde{P}_{0} \times _{\rho \circ p_{*}}\Z^{d},
\]
the Abelian cover associated with the representation $\rho \circ p_{*}$, where $p_{*} \colon \pi_{1}(P_{0},z_0) \to \pi_{1}(M_{0},x_0)$ is the map induced by $p$, where $z_0 \in P_{x_0}$ is arbitrary.

\begin{lemma} \label{lemma:g-z-bundles-relation}
    The following holds:
    \begin{enumerate}[label=\roman*)]
        \item $P$ is a $G$-bundle over $M$, and a $(G \times \Z^d)$-bundle over $M_0$;
        \item $P \simeq \pi^*P_0$ as $G$-bundles over $M$, and as $(G \times \Z^d)$-bundles over $M_0$.
    \end{enumerate}
\end{lemma}

\begin{proof}
Recall that, by construction, the universal cover of a topological space is defined as the set of continuous paths based at a given point, modulo homotopies with fixed ends. Let $z_0 \in P$ and $x_0 := p(z_0) \in M_0$, where $p : P_0 \to M_0$ is the projection. Let $\gamma_{P_0} \colon [0,1]\to P_0$ be a path with $\gamma_{P_0}(0)=z_0$ and $\gamma_{P_0}(1)=z \in P$. This path projects to a path $\gamma_{M_0} := p \circ \gamma_{P_0}$ with $\gamma_{M_0}(0)=x_0$ and $\gamma_{M_0}(1) = x \in M_0$, which shows that there is a well-defined map $\widetilde{p} \colon \widetilde{P}_0 \to \widetilde{M}_0$. In addition, there is a well-defined action of $G$ on $\widetilde{P}_0$ given by $\widetilde{R}_g(\gamma_{P_0}) := R_g \circ \gamma_{P_0}$ (where $R_g \colon P_0 \to P_0$ is the right-action of $g \in G$ on $P_0$), and $\widetilde{p} \colon \widetilde{P}_0 \to \widetilde{M}_0$ is a $G$-bundle map with respect to this action. Finally, by construction, for $\gamma \in \pi_1(P_0,z_0)$ and $z \in \widetilde{P}_0$, it is immediate that $\widetilde{p}(\gamma.z) = p_*(\gamma).\widetilde{p}(z)$, where $p_* \colon \pi_1(P_0,z_0) \to \pi_1(M_0,x_0)$ is the induced map in homotopy. We now drop the dependence of $\pi_1(M_0, x_0)$ (resp. $\pi_1(P_0,z_0)$) on the basepoint $x_0 \in M_0$ (resp. $z_0 \in P_0$) to simplify notation.
\medskip

\emph{Step 1. $P$ is a $G$-bundle over $M$.} We claim that the map $\widetilde{p} \times \id \colon \widetilde{P}_{0} \times \Z^{d} \to \widetilde{M}_{0}\times \Z^{d}$ descends to a $G$-bundle map $\overline{p} \colon P \to M$. First, we check that $\widetilde{p} \times \mathrm{id}$ is equivariant with respect to the actions of $\pi_{1}(P_{0})$ and $\pi_{1}(M_{0})$ on the domain and codomain, respectively. Let $\widetilde{z} \in \widetilde{P}_{0}$, $\n \in \Z^{d}$, and let $\gamma \in \pi_{1}(P_{0})$. Then:
\[
\begin{split}
    (\widetilde{p} \times \id)(\gamma .(\widetilde{z},\n)) &=(\widetilde{p} \times \id)(\gamma . \widetilde{z},-(\rho \circ p_{*})(\gamma)+\n) \\
    &=(\widetilde{p}(\gamma .\widetilde{z}),-(\rho \circ p_{*})(\gamma)+\n) \\
    &=(p_{*}(\gamma). \widetilde{p}(\widetilde{z}),-(\rho(p_{*}(\gamma)))+\n) \\
    &=p_{*}(\gamma).(\widetilde{p}(\widetilde{z}),\n),
\end{split}
\]
where in the third equality we used that $\tilde{p}$ intertwines deck transformations. In addition, the $G$-action on $\widetilde{P}_0$ descends to $P$ as it commutes with the action of $\pi_1(P_0)$. This shows that the induced map $\overline{p} \colon P \to M$ is a $G$-bundle map.

\medskip

\emph{Step 2. $P$ is a $(G \times \Z^{d})$-bundle over $M_{0}$.} The right $(G \times \Z^{d})$-action on $P$ is given by $[\widetilde{z},\n] \cdot (g,\kk) := [\widetilde{z}\cdot g,\n+\kk]$, which is well-defined on $P$ following the same proof as in Step 1. Recall that $\pi : M \to M_0$ is the projection. It is also immediate to verify that $\pi \circ \overline{p}(\bullet \cdot (g,\kk)) = \pi \circ \overline{p}(\bullet)$. Therefore, $(\pi \circ \overline{p}) \colon P \to M_{0}$ is a $(G \times \Z^{d})$-bundle map.

\medskip

\emph{Step 3. $\pi^{*}P_{0}$ and $P$ are isomorphic as $G$-bundles over $M$.} Let $q_{M} \colon \widetilde{M}_{0} \to M_{0}$ and $q_P \colon \widetilde{P}_0 \to P_0$ be the projections. Define
\[
\varphi \colon \widetilde{P}_{0} \times \Z^{d}  \to \pi^{*}P_{0}, \qquad \varphi(\widetilde{z},\n) := ([\widetilde{p}(\widetilde{z}),\n],q_{P}(\widetilde{z})).
\]
A quick calculation shows that $\varphi$ descends to a map $\overline{\varphi} \colon P \to \pi^{*}P_{0}$. Indeed, let $\gamma \in \pi_{1}(P_{0})$, $\widetilde{z} \in \widetilde{P}_{0}$, $\n \in \Z^{d}$; then, as in Step 1, one finds that $\varphi(\gamma .(\widetilde{z},\n)) = ([\widetilde{p}(\widetilde{z}),\n],q_{P}(\widetilde{z})) = \varphi(\widetilde{z},\n)$. In addition, it is also straightforward to verify that $\overline{\varphi}$ intertwines the two $G$-actions on $\pi^*P_0$ and $P$, and that $\mathrm{pr} \circ \overline{\varphi}=\overline{p}$, where $\mathrm{pr} \colon \pi^{*}P_{0} \to M$ is the projection.

We now show that $\overline{\varphi}$ is bijective. First, we deal with the injectivity. Let $[\widetilde{z}_{j},\n_{j}] \in P$, with $j=1,2$ be such that 
\begin{equation} \label{equation:bar-phi-injective}
    \overline{\varphi}([\widetilde{z}_{1},\n_{1}])=\overline{\varphi}([\widetilde{z}_{2},\n_{2}]).
\end{equation}
This implies that $q_{P}(\widetilde{z}_{2})=q_{P}(\widetilde{z}_{1})$. Hence, there exists $\gamma \in \pi_{1}(P_{0})$ such that $\widetilde{z}_{2}=\gamma.\widetilde{z}_{1}$. Then, using the equality of the first entries of \eqref{equation:bar-phi-injective} yields
\[ [\widetilde{p}(\widetilde{z}_{1}),\n_{1}]=[\widetilde{p}(\widetilde{z}_{2}),\n_{2}]=[\widetilde{p}(\gamma.\widetilde{z}_{1}),\n_{2}] =[p_{*}(\gamma).\widetilde{p}(\widetilde{z}_{1}),\n_{2}] =[\widetilde{p}(\widetilde{z}_{1}),\rho(p_{*}(\gamma))+\n_{2}]. \]
Comparing the left and right-hand side, we conclude that $\n_{2}=-\rho(p_{*}(\gamma))+\n_{1}$, which shows that $[\widetilde{z}_{1},\n_{1}]=[\widetilde{z}_{2},\n_{2}]$. 

We now show the surjectivity. Consider $([\widetilde{x},\n], z) \in \pi^*P_0$, where $z \in P_x$ and $x := q_M(\widetilde{x})$. By construction, $\widetilde{x}$ corresponds to a path connecting the basepoint $x_0$ to $x$, which can be lifted to a path $\widetilde{z} \in \widetilde{P}$ from $z_0 \in P_{x_0}$ to $z$. Then $\varphi(\widetilde{z},\n) = ([\widetilde{p}(\widetilde{z}),\n],q_P(\widetilde{z})) = ([\widetilde{x},\n],z)$. This concludes the proof of the third step.

\medskip

\emph{Step 4. $(\pi \circ \mathrm{pr}) \colon \pi^{*}P_{0} \to M_{0}$ and $(\pi \circ \overline{p}) \colon P \to M_0$ are isomorphic as $(G \times \Z^{d})$-bundles over $M_0$.} The $(G \times \Z^{d})$-action on $\pi^{*}P_{0}$ is given by
\[
([x,\n],y)\cdot(g,\kk) := ([x,\n+\kk],y \cdot g).
\]
We may then define $\overline{\varphi}$ as in Step 3. A similar proof to that step shows that $\overline{\varphi}$ is an isomorphism,  that $(\pi \circ \mathrm{pr}) \overline{\varphi}=\pi \circ \overline{p}$, and that $\overline{\varphi}$ intertwines both actions. This concludes the proof.
\end{proof}

Recall that a connected $\Z^{d}$-cover $M \to M_{0}$ is the data of a surjective representation $\rho \colon \pi_{1}(M_{0}) \to \Z^{d}$. Let $p \colon E_{0} \to M_{0}$ be a fiber bundle with typical fiber $F$; then $\rho \circ p_{*} \colon \pi_1(E_0) \to \Z^d$ defines a representation of $\pi_{1}(E_{0})$, hence a $\Z^{d}$-cover of $E_{0}$. A natural question is if we can go the other way around; that is given a $\Z^{d}$-cover of $E_{0}$, can we obtain a $\Z^{d}$-cover of $M_{0}$? Under certain assumptions on $F$, we can prove that it is indeed the case:

\begin{lemma}
\label{lemma:ehoui}
    Assume that $p \colon E_0 \to M_0$ is a fiber bundle with typical fiber $F$, such that $F$ is connected and $\pi_1(F)$ is finite. Then there is a one-to-one correspondence between $\Z^d$-covers of $E_0$ and $\Z^d$-covers of $M_0$. 
\end{lemma}

In particular, the lemma applies to all principal bundles $p \colon P_0 \to M_0$ with $G$ compact and semisimple as $\pi_1(G)$ is finite, see \cite[Remark 7.13]{Brocker-Dieck-85} for instance.

\begin{proof}
This boils down to showing that there exists a map $\pi_{1}(M_{0}) \to \Z^{d}$ such that the following diagram commutes:
\[
\xymatrix{
\pi_{1}(E_{0}) \ar[rr]^{\rho_{E_0}} \ar[dr]_{p_{*}} 
  & & \mathbb{Z}^{d} \\
& \pi_{1}(M_{0}) \ar@{-->}[ur] &
}
\]
However, we can factorize $\rho_{E_0}$ through the first homology group of $E_{0}$ (using the Hurewicz map). Hence, it is enough to obtain an isomorphism between the torsion free part of the homologies, that is, we will show that $H_{1}(E_{0}, \Z)/\mathrm{Tor} \cong H_{1}(M_{0}, \Z)/\mathrm{Tor}$. The fibration $F \hookrightarrow E_0 \xrightarrow{p}M_0$ gives the long exact sequence
\[ \cdots \rightarrow \pi_{1}(F) \rightarrow \pi_{1}(E_{0}) \xrightarrow{p_{*}} \pi_{1}(M_{0}) \to \pi_{0}(F) \rightarrow \cdots \]
Since $F$ is connected, $\pi_{0}(F)$ is trivial, so $p_{*}$ is surjective. Taking Abelianizations, the Hurewicz Theorem gives 
\[ \pi_{1}(F)^{\mathrm{Ab}} \to H_{1}(E_{0}, \Z) \xrightarrow{p_{*}^{\mathrm{Ab}}} H_{1}(M_{0}, \Z), \]
where $\bullet^{\mathrm{Ab}}$ denotes the Abelianization of the group, and $p_{*}^{\mathrm{ab}}$ is the map induced by $p_{*}$, which is still surjective as Abelianizing is a right-exact functor. As $\pi_1(F)^{\mathrm{Ab}}$ is finite, its image in $H_1(E_0, \Z)$ lies in the torsion part; hence $\ker p_{*}^{\mathrm{Ab}}$ is contained in the torsion-part so the induced map
\[
H_{1}(E_{0}, \Z)/\mathrm{Tor} \xrightarrow{p_{*}^{\mathrm{Ab}}} H_{1}(M_{0}, \Z)/\mathrm{Tor}
\]
is an isomorphism.
\end{proof}

Finally, we focus on the particular case of Abelian extensions of Riemannian manifolds. For a Riemannian manifold $N_0$, we let $SN_0$ denote its unit tangent bundle, and $FN_0$ its frame bundle.

\begin{lemma}
Let $N_{0}$ be a closed connected Riemannian manifold with $\dim N_{0} \geq 3$, $\rho \colon \pi_{1}(N_{0}) \to \Z^{d}$ a surjective representation, $q \colon \widetilde{N}_{0} \to N_{0}$ be the universal cover, and let $\pi \colon N=\widetilde{N}_{0} \times_{\rho} \Z^{d} \to N_{0}$ be the associated Abelian cover. Then $SN \cong \widetilde{SN}_{0} \times_{\rho \circ (p_{SN_0})_*}\Z^{d}$ and $FN \cong \widetilde{FN}_{0} \times_{\rho \circ (p_{FN_0})_{*}} \Z^{d}$, where $p_{FN_0} \colon FN_{0} \to N_{0}$ and $p_{SN_0} \colon SN_0 \to N_0$ are the projections. In addition, $FN$ is a $\mathrm{SO}(n-1)$-bundle over $SN$ and a $(\mathrm{SO}(n-1)\times\Z^d)$-bundle over $SN_0$.
\end{lemma}

\begin{proof}
As $n := \dim N_{0} \ge 3$, $\pi_1(S^{n-1})=\{0\}$ and $\pi_1(\mathrm{SO}(n-1))$ is finite. By Lemma \ref{lemma:ehoui}, there is a one-to-one correspondence between Abelian covers of $N_0$, $SN_0$ and $FN_0$. The claim then easily follows from this observation.
\end{proof}

\subsection{Functional spaces on Abelian covers}  \label{ssection:functional-spaces}

Let $g_0$ be an arbitrary background metric on $M_0$ and let $g$ denote its lift to $M$. Let $\nabla$ be the Levi-Civita connection induced by $g$ on $M$. Recall that, given $u \in C^\infty(M)$ and an integer $s \geq 0$, $\nabla^s u$ is a section of the symmetric $s$-tensor bundle $\mathrm{Sym}^s(T^*M) \to M$ such that
\[
\nabla^s u (x ; v, \dotsc, v) := \partial^s_t u(\gamma(t))|_{t = 0}, \qquad \forall (x,v) \in TM,
\]
where $t \mapsto \gamma(t)$ is the geodesic generated by $(x,v)$. The vector bundle $\mathrm{Sym}^s(T^*M) \to M$ carries a natural inner product in its fibres induced by the Riemannian metric $g$. We note that by definition, for every $\mathbf{n} \in \mathbb{Z}^d$, we have that $\tau_{\mathbf{n}}: (M, g) \to (M, g)$ is an isometry and that the differential of $\tau_{\mathbf{n}}$ is an isometry of the fibres of $\mathrm{Sym}^s(T^*M)$.

Given an integer $s \geq 0$, and an open subset $U \subset M$, the $H^s$-norm on $U$ is defined as:
\[
\|u\|^2_{H^s(U)} := \|u\|^2_{L^2(U)} + \|\nabla^s u\|^2_{L^2(U, \mathrm{Sym}^s T^*U)},
\]
and the $C^s$-norm is defined as
\[
\|u\|_{C^s(U)} := \|u\|_{C^0(U)} + \|\nabla^s u\|_{C^0(U, \mathrm{Sym}^s T^*U)}.
\]
Since $\tau_{\mathbf{n}}$ acts by isometries of $\mathrm{Sym}^s(T^*M)$, we have 
\begin{equation}\label{eq:symmetry-CHs}
    \|\tau_{\mathbf{n}}^*u\|_{F^s(U)} = \|u\|_{F^s(\tau_{\mathbf{n}}U)},\quad F \in \{C, H\}. 
\end{equation}
Let $\widehat{x}_0 \in M$ be an arbitrary point in $M$. Let $D \subset M$ be a relatively compact open subset with smooth boundary such that $\widehat{x}_0 \in D$ and $\pi \colon D \to M_0$ is surjective. For nonnegative integers $r, s \geq 0$, we introduce the norm
\[
\|u\|_{B^{s,r}(M)} := \sum_{\n \in \Z^d} \langle \n \rangle^{r} \|u\|_{H^s(\tau_{\mathbf{n}}D)},
\]
where $\langle n \rangle := (1+|\n|^2)^{1/2}$ and $|\bullet|$ is the Euclidean norm in $\Z^d$. The space $B^{s,r}(M)$ is defined as the completion of $C^\infty_{\comp}(M)$ with respect to the previous norm. Similarly, we let
\[
\begin{split}
\|u\|_{C^{s,r}(M)} & := \sup_{\n \in \Z^d} \langle \n \rangle^{r} \|u\|_{C^s(\tau_{\mathbf{n}}D)},
\end{split}
\]
and define $C^{s,r}(M)$
as the completion of $C^\infty_{\comp}(M)$ with respect to the norm $\|\bullet\|_{C^{s,r}(M)}$.
It is straightforward to verify that the functional space $F^{s,r}(M)$ is independent of the choice of background metric $g_0$ on $M_0$ and domain $D \subset M$, for $F \in \{B,C\}$ (a different choice of $D$ would give equivalent norms). For non integer values of $r,s \geq 0$, and $F \in \{B,C\}$, the spaces $F^{s,r}(M)$ are defined by real interpolation. The following lemma is a consequence of standard results.

\begin{lemma}
\label{lemma:relation-spaces}
    Let $r, s \geq 0$. For all $\eps > 0$, there exists a constant $C > 0$ such that:
    \[
    \begin{array}{lll}
    \|u\|_{B^{s,r}(M)} & \leq C \|u\|_{C^{s,r+d+\eps}(M)}, & \qquad \forall u \in C^{s,r+d+\eps}(M), \\
    \|u\|_{C^{s,r}(M)} & \leq C \|u\|_{B^{s+n/2+\eps,r}(M)}, & \qquad \forall u \in H^{s+n/2+\eps,r}(M).
    \end{array}
    \]
\end{lemma}

\begin{proof}
    For the first estimate, for $u \in C^\infty_{\comp}(M)$ we have
    \begin{align*}
        \|u\|_{B^{s, r}(M)} = \sum_{\n \in \mathbb{Z}^d} \langle \n \rangle^r \|u\|_{H^s(\tau_{\n}D)} &= \sum_{\n \in \mathbb{Z}^d} \langle \n \rangle^{-d - \varepsilon} \langle \n \rangle^{r + d + \varepsilon} \|u\|_{H^s(\tau_{\n} D)}\\ 
        &\leq \|u\|_{C^{s, r + d + \varepsilon}} \sum_{\n \in \mathbb{Z}^d} \langle \n \rangle^{-d - \varepsilon},
    \end{align*}
    which proves the bound since the sum over $\n \in \Z^d$ converges.

    For the second one, for some $C> 0$, and for all $u \in C^\infty_{\comp}(M)$ and $\n \in \mathbb{Z}^d$ we have
    \begin{align*}
        \|u\|_{C^{s}(\tau_{\n} D)} = \|\tau_{\n}^* u\|_{C^s(D)} \leq C\|\tau_{\n}^*u\|_{H^{s + \frac{n}{2} + \varepsilon}(D)} = C\|u\|_{H^{s + \frac{n}{2} + \varepsilon}(\tau_{\n} D)},
    \end{align*}
    where in the equalities we used \eqref{eq:symmetry-CHs}, and in the inequality we used the classical Sobolev embedding theorem on $D$. The second estimate follows immediately, completing the proof. 
\end{proof}

An open subset $U \subset \mathrm{U}(1)^d$ is called \emph{good} if there exists an open neighbourhood $V \subset \mathrm{U}(1)^d$ of $\overline{U}$ such that $V$ is contractible. In the next lemma, given $f \in C^\infty_{\comp}(M)$, and $U \subset \mathrm{U}(1)^d$, a good open subset, we write $f_\bullet \in C^\infty(U,C^\infty(M_0))$ to denote the function $U \to C^\infty(M_0), \btheta \mapsto f_{\btheta}$, defined in \S \ref{ssection:floquet-theory}. For $k \geq 0$ an integer, we will write 
\[
    \|f_{\bullet}\|_{C^k(U, H^s(M_0))} := \max_{|\alpha| \leq k} \sup_{\btheta \in U} \|\partial^{\alpha}_{\btheta} f_{\btheta}\|_{H^s(M_0)}.
\]

\begin{lemma} \label{lemma:spaces_theta_sobolev}
Let $U \subset \mathrm{U}(1)^d$ be a good open subset, $k \geq 0$ an integer, and $s \geq 0$ a non negative real number. Then, there exists $C > 0$ such that for all $f \in B^{s,k}(M)$
\[
\|f_\bullet\|_{C^k(U,H^s(M_0))} \leq C \|f\|_{B^{s,k}(M)}.
\]
\end{lemma}

\begin{proof}
    We begin by proving the claim for $k=0$. It suffices to prove the claim for integers $s \geq 0$, as the general statement follows by interpolation. For $x \in M$, using \eqref{eq:theta-frequency}, we write
    \begin{equation}
        \label{equation:dimanche}
    F_{\btheta}(x) = \pi^* f_{\btheta}(x) s_{\btheta}(x) = \sum_{\n \in \Z^d} f(\tau_{\n}(x)) e^{-i\n \cdot \btheta}.
    \end{equation}
    It follows that:
    \[
    \begin{split}
    \|f_{\btheta}\|_{H^s(M_0)} & \leq \|\pi^* f_{\btheta}\|_{H^s(D)} = \|s_{\btheta}^{-1}F_{\btheta}\|_{H^s(D)}  \leq \sum_{n \in \Z^d} \|s_{\btheta}^{-1} f(\tau_{\mathbf{n}}\bullet)\|_{H^s(D)} \\
    & \leq C \sum_{n \in \Z^d} \|s_{\btheta}^{-1}\|_{C^s(D)} \|f(\tau_{\mathbf{n}}\bullet)\|_{H^s(D)} \\
    & \leq C \sum_{n \in \Z^d} \|f(\bullet)\|_{H^s(\tau_{\n}D)} = C \|f\|_{B^{s,0}(M)},
    \end{split}
    \]
    where in the last line we used \eqref{eq:symmetry-CHs}. For general values of $k \geq 1$, the estimate is obtained by differentiating \eqref{equation:dimanche} with respect to $\btheta$ and applying the previous argument.
\end{proof}

\subsection{Dynamical connection} 

\label{ssection:dynamical-connection}

We recall now-standard material on the dynamical connection. We refer to \cite[Section 4.2.1]{Cekic-Lefeuvre-24} or \cite[Chapter 12]{Lefeuvre-book} for further discussion.

\subsubsection{Affine case}

\label{sssection:dynamical-connection-linear}

Let $E \to M_0$ be a Hermitian vector bundle. Recall that $X_{M_0}$ denotes the vector field generating the volume-preserving Anosov flow $(\varphi_t^0)_{t \in \R}$. Let
\[
\X \colon C^\infty(M_0,E) \to C^\infty(M_0,E)
\]
be a differential operator of order $1$ such that
\[
\X(f \otimes s) = X_{M_{0}} f \otimes s + f \otimes \X s,
\]
for all $f \in C^\infty(M_0)$, $s \in C^\infty(M_0,E)$. Such an operator is called a \emph{lift} of $X_{M_{0}}$ to sections of $E$.

There is a well-defined fiberwise linear cocycle $C$ such that for all $s \in C^\infty(M_0,E)$:
\[
(e^{-t\X} s) (\varphi_t^0 x) = C(t,x) s(x), \qquad \forall t \in \R, x \in M_0,
\]
where $(e^{t\X})_{t \in \R}$ is the propagator of $\X$, and $C(x,t) \colon E_x \to E_{\varphi_t^0 x}$ is a linear map, depending smoothly on both parameters. The map $C(t,x)$ is an isometry (for all $t,x$) if and only if $\X$ is skew-adjoint on $L^2(M_0,E)$ (or, equivalently, $(e^{t\X})_{t \in \R}$ is unitary on $L^2(M_0,E)$). From now on, we will work under this standing assumption.

The operator $\X$ induces a partially hyperbolic flow $(\Phi_t)_{t \in \R}$ on $E$ such that for all $x \in M_0, s \in E_{x}$, $\Phi_t(s) := C(t,x) s$ (see \cite[Section 12.1.1.2]{Lefeuvre-book} for instance). In particular, there are well-defined stable and unstable manifolds $W^{s,u}_E$ on $E$. The \emph{stable holonomy} is the linear map $\mathbf{H}^s \colon E_x \to E_y$, defined for all $x \in M_0$ and $y \in W^s(x)$, such that $\mathbf{H}^s s \in E_y$ is the (unique) point of intersection $W^s_E(s) \cap E_y = \{\mathbf{H}^s s\}$. The \emph{unstable holonomy} map $\mathbf{H}^u \colon E_x \to E_y$ is defined similarly for $y \in W^u(x)$. Finally, the \emph{central holonomy} is defined using the flow $(\Phi_t)_{t \in \R}$ itself, that is, for $x \in M_0, y = \varphi_t^0 x$, $\mathbf{H}^c s := \Phi_t s$. As the central, stable, and unstable bundles span the whole of $TM_0$ (that is, $TM_0 = \R X_{M_{0}} \oplus E_s \oplus E_u$), these holonomies can easily be seen to derive from a unitary connection $\nabla^{\mathrm{dyn}}$ on $E$, called the \emph{dynamical connection} (see \cite[Section 4.2.1]{Cekic-Lefeuvre-24}). In general, this connection is only Hölder-continuous (see \cite{Beaufort-25} for further discussion about the regularity of this connection). However, if $\X = \nabla^E_{X_{M_{0}}}$, where $\nabla^E$ is a \emph{flat unitary} connection in $E$, a quick calculation reveals that $\nabla^{\mathrm{dyn}} = \nabla^E$. 

There is a well-defined curvature $F_{\nabla^{\mathrm{dyn}}}$ associated with $\nabla^{\mathrm{dyn}}$, which is a distributional $2$-form with values in skew-adjoint endomorphisms of $E$, that is
\[
F_{\nabla^{\mathrm{dyn}}} \in \mc{D}'(M_0, \Lambda^2 T^*M_0 \otimes \mathrm{End}_{\mathrm{sk}}(E)).
\]

\subsubsection{Principal bundles}

The same construction can be carried out on principal bundles. Let $p \colon P_0 \to M_0$ be a $G$-principal bundle, where $G$ is a compact Lie group, and $(\psi_t^0)_{t \in \R}$ be an equivariant extension of $(\varphi_t^0)$ to $P_0$ in the sense that $\varphi_t^0 \circ p = p \circ \psi_t^0$ and $R_g \circ \psi_t^0 = \psi_t^0 \circ R_g$ for all $t \in \mathbb{R}, g \in G$. The flow $(\psi_t^0)_{t \in \R}$ being partially hyperbolic on $P_0$ (see \cite[Section 12.1.1.2]{Lefeuvre-book}), central, stable and unstable holonomy maps $\mathbf{H}^{c,s,u} \colon (P_0)_x \to (P_0)_y$ can be defined as above for $x \in M_0$ and $y$ in the central, stable or unstable manifold of $x$ respectively. These holonomies correspond to the parallel transport of a $G$-equivariant connection on $P_0$, called the dynamical connection.

The holonomy group of this connection (which is a subgroup of $G$) is called the Brin (pre-)transitivity group. It is dense in $G$ if and only if the flow $(\psi_t^0)_{t \in \R}$ is ergodic with respect to the probability measure $\nu_0$, obtained locally as a product of $\vol_{M_{0}}$ with the Haar measure on $G$, see \cite[Section 12.4]{Lefeuvre-book}. The curvature of the dynamical connection is a distributional $2$-form with values in the adjoint bundle of $P_0$, that is
\[
F \in \mc{D}'(M_0, \Lambda^2 T^*M_0 \otimes \mathrm{Ad}(P_0)).
\]

Let $\rho \colon G \to \mathrm{End}(V)$ be a unitary representation of $G$, and $E := P_0 \times_{\rho} V$ be the associated bundle. The flow $(\psi_t^0)_{t \in \R}$ induces a flow $(\Phi_t)_{t \in \R}$ on $E$, defined as $\Phi_t([\xi,s]) := [\psi_t^0 \xi, s]$. Similarly, the $G$-equivariant dynamical connection induces functorialy a unitary associated connection $\nabla$ on $E$. It is then immediate to verify that the dynamical connection $\nabla^{\mathrm{dyn}}$ induced by $(\Phi_t)_{t \in \R}$ on $E$, as introduced in \S\ref{sssection:dynamical-connection-linear}, coincides with $\nabla$.

\section{Preliminaries on semiclassical analysis}

\label{section:semiclassical}

\subsection{Semiclassical analysis}

In this section we recall and state the key notions and tools from semiclassical analysis. For more details on the content of this section, see \cite[Appendix E]{Dyatlov-Zworski-19}, \cite{Zworski-12,Lefeuvre-book}.

\subsubsection{Definitions and basic properties} 

Recall that for $\xi \in \R^n$, the Japanese bracket is defined as $\langle \xi \rangle := (1+|\xi|^2)^{1/2}$, where $|\bullet|$ is the Euclidean norm. Given $m \in \R$, we let $S^m(\R^n)$ be the class of symbols on $\R^n$ of order $m$, defined as the set of functions $a \in C^\infty(T^*\R^n)$ such that for all $\alpha,\beta \in \Z^n_{\geq 0}$, there exists $C > 0$ such that
\begin{equation}
    \label{equation:symbols}
|\partial^\alpha_\xi \partial^\beta_x a(x,\xi)| \leq C\langle\xi\rangle^{m-|\alpha|}, \qquad \forall (x,\xi) \in T^*\R^n.
\end{equation}
On a smooth closed manifold $M_0$, the class $S^m(M_0) \subset C^\infty(T^*M_0)$ of symbols of order $m \in \R$ can be defined analogously by requiring \eqref{equation:symbols} to hold in any local chart. It is also possible to let $a$ depend on the additional semiclassical parameter $h > 0$; however, in this case, we require \eqref{equation:symbols} to hold uniformly with respect to $h$.

The standard (left) semiclassical quantization in $\R^n$ is given for $a \in S^m(\R^n)$ and $f \in C^\infty_{\comp}(\R^n)$ by
\begin{equation}
\label{equation:quantization}
\Op_h^{\R^n}(a) f (x) := \dfrac{1}{(2\pi h)^n} \int_{\R^n_\xi \times \R^n_y} e^{\tfrac{i}{h}\xi\cdot(x-y)} a(x,\xi) f(y)\, \dd y \dd \xi.
\end{equation}
It is standard that, using local charts, one can define a quantization procedure $\Op_h$ on $M_0$, mapping $S^m(M_0)$ into the class of semiclassical pseudodifferential operators \cite[Chapter 14]{Zworski-12}. We then let
\[
\Psi^m_h(M_0) := \{\Op_h(a) + R ~:~ a\in S^m(M_0), R \in h^{\infty}\Psi^{-\infty}_h(M_0)\},
\]
be the set of semiclassical pseudodifferential operators of order $m$. Here, $R$ is a $\mc{O}(h^\infty)$-smoothing operator, that is for any $N > 0$ it maps boundedly $H^{-N}_h(M_0) \to H^N_h(M_0)$ with norm $\leq C_N h^N$.

For $s \in \R$, the Sobolev spaces $H^s_h(M_0)$ are defined as the completion of $C^\infty(M_0)$ with respect to the norm
\[
\|f\|_{H^s_h(M_0)} := \|(1+h^2\Delta_{g_0})^s f\|_{L^2(M_0)},
\]
where $\Delta_{g_0} \geq 0$ is the Laplace operator induced by an arbitrary background Riemannian metric $g_0$ on $M_0$, and $(1+h^2\Delta_{g_0})^s$ is defined using the spectral theorem. The Calder\'on-Vaillancourt theorem asserts that any $A = \Op_h(a) \in \Psi^m_h(M_0)$ is bounded as a map
\[
    A \colon H^{s+m}_h(M_0) \to H^s_h(M_0)
\]
with norm $\|A\| \leq \|\langle\xi\rangle^{-m}a\|_{L^\infty(T^*M_0)} + \mc{O}(h)$.

\subsubsection{Main properties}

We now summarize some standard results from semiclassical analysis that will be used throughout the proof of Theorem \ref{theorem:main1}. We will use the notion of semiclassical wavefront set $\WF_h$, see \cite[Section E.2]{Dyatlov-Zworski-19} for an introduction. Additionally, it will be convenient to work with the radial compactification $\overline{T^*M_0} = T^*M_0 \sqcup \partial_\infty T^*M_0$ of the cotangent bundle, where $T^*M_0/\R_+^*$ and $\R_+^*$ denotes the radial dilation $(x,\xi) \mapsto (x,\lambda \xi)$ for $(x,\xi) \in T^*M_0$, $\lambda > 0$.

\begin{lemma} [Elliptic Regularity] \label{lemma:elliptic-rgularity}
    Let $A \in \Psi_{h}^{m}(M_{0})$, $B \in \Psi_{h}^{m'}(M_{0})$. Assume that $\WF_{h}(B) \subset \Ell(A)$. Then, there exist $Q_L, Q_R  \in \Psi_{h}^{m'-m}(M_{0})$ such that
    \[
    B = Q_LA+\mathcal{O}_{\Psi_{h}^{-\infty}(M_{0})}(h^{\infty}) = AQ_R+\mathcal{O}_{\Psi_{h}^{-\infty}(M_{0})}(h^{\infty}).
    \]
\end{lemma}

\begin{lemma}[Egorov's Theorem] \label{lemma:egorov}
    Let $P \in \Psi_{h}^{1}(M_{0})$ with real principal symbol $\sigma_{P}$, $(\Phi_{t})_{t \in \mathbb{R}}$ be the Hamiltonian flow on $T^* M_{0}$ generated by the Hamiltonian $\sigma_{P}$, and let $A \in \Psi_{h}^{k}(M_{0})$. Then
    \[ 
    B(t):=e^{i t P/h} A e^{-i t P/h} \in \Psi_{h}^{k}(M_0), \quad \sigma_{B(t)}(x, \xi)=\sigma_A\left(\Phi_t(x, \xi)\right), \quad \forall t \in \mathbb{R}.
    \]
\end{lemma}

In what follows, we let $\btheta \in \mathrm{U}(1)^d$ and consider the first order differential operator $X_{\btheta}$ acting on $C^\infty(M_0)$ which is introduced in \S\ref{section:abelian-extension-anosov} below. Its principal symbol agrees with the one of the flow generator $X_{M_{0}}$, i.e. $X_{\btheta} - X_{M_{0}}$ is of order zero (and imaginary-valued).

\begin{lemma}[Propagation of Singularities] \label{lemma:propagation-singularities}
    Let $A,A' \in \Psi_{h}^{0}(M_{0})$ be such that the following holds: for all $(x, \xi) \in \WF_{h}(A)$, there exists $t \in \mathbb{R}$ such that $\Phi_t(x, \xi) \in \Ell (A')$. Then, there exists $B \in \Psi_{h}^{0}(M_{0})$, such that for all $s \in \R$, $N>0$, $\btheta \in \mathrm{U}(1)^d$, there exists $C>0$ such that for all $f_{h} \in C^{\infty}(M_{0})$, the following holds:
    \[
    \|A f_{h}\|_{H_{h}^{s}} \leq C\left(\|BX_{\btheta}f_{h}\|_{H_{h}^{s}}+\|A' f_{h}\|_{H_{h}^{s}}+h^{N}\|f_{h}\|_{H_{h}^{-N}}\right)
    \]
    Moreover, if $f_{h}$ is merely a distribution and the right-hand side is finite, then $A f_{h} \in H_{h}^{s}$ and the same estimate holds.
\end{lemma}

In what follows, we note that the so-called \emph{thresholds} $\omega_\pm(X_{\btheta})$ in the radial estimates are zero since $X_{\btheta}$ is skew-adjoint (see \cite[Remark 11.1.3]{Lefeuvre-book} for the definition).

\begin{lemma}[Source/Sink Estimates] \label{lemma:source-sink-estimates}
    \begin{enumerate}[label=\roman*),itemsep=4pt]
        \item \emph{Source estimate.} For all $s>0$, $N \gg 0$ large enough, for all $A \in \Psi_{h}^{0}(M_{0})$ with wavefront set near $\overline{E_{s}^{*}} \cap \partial_{\infty}T^{*}M_0$, there exists $B \in \Psi_{h}^{0}(M_{0})$ with wavefront set and elliptic on a slightly larger neighborhood than $\WF_{h}(A)$, and a constant $C>0$ such that for all $h>0$, $\btheta \in \mathrm{U}(1)^d$, and for all families $f_{h} \in C^{\infty}(M_{0})$,
        \[
        \|A f_{h}\|_{H_{h}^{s}} \leq C\left(\|BX_{\btheta}f_{h}\|_{H_{h}^{s}}+h^{N}\|f_{h}\|_{H_{h}^{-N}}\right).
        \]
        If $f_{h} \in \mathcal{D}'(M_{0})$ is merely a family of distributions such that $BX_{\btheta} f_{h} \in H_{h}^{s_{0}}$ for some $s_{0}>0$, then $Af_{h} \in H_{h}^{s}$ for all $s>0$ and the previous estimate holds.
        
        \item \emph{Sink estimate.} For all $s<0$, $N \gg 0$ large enough, for all $A \in \Psi_{h}^{0}(M_{0})$ with wavefront set near $\overline{E_{u}^{*}} \cap \partial_{\infty}T^{*}M_0$, there exists $B \in \Psi_{h}^{0}(M_{0})$ with wavefront set and elliptic on a slightly larger neighborhood than $\WF_{h}(A)$, $A' \in \Psi_{h}^{0}(M_{0})$ such that $\WF_{h}(A') \subset \WF_{h}(B)$ and $\WF_{h}(A') \cap \overline{E_{u}^{*}} \cap \partial_{\infty}T^{*}M=\emptyset$, and a constant $C>0$ such that for all $h>0$, $\btheta \in \mathrm{U}(1)^d$, for all families $f_{h} \in C^{\infty}(M_{0})$,
        \[
        \|A f_{h}\|_{H_{h}^{s}} \leq C\left(\|BX_{\btheta} f_{h}\|_{H_h^s}+\|A' f_{h}\|_{H_{h}^{s}}+h^{N}\|f_{h}\|_{H_{h}^{-N}}\right) .
        \]
        If $f_{h} \in \mathcal{D}'(M_{0})$ is merely a family of distributions such that $BX_{\btheta} f_{h} \in H_{h}^{s_{0}}$ and $A' f_{h} \in H_{h}^{s_0}$ for some $s_{0}<0$, then $A f_{h} \in H_{h}^{s}$ for all $s<0$ and the previous estimate holds.
    \end{enumerate}
\end{lemma}

We refer to \cite{Dyatlov-Zworski-16}, \cite[Appendix E]{Dyatlov-Zworski-19}, or \cite{Lefeuvre-book} for a proof. We note that the constants $C > 0$ in the previous statements can be taken to be uniform with respect to $\btheta$. Indeed, if not, arguing by contradiction there is a sequence of families $(f_{n, h})_{n, h > 0}$ which satisfies the reversed inequality with the values of $C = n$ and for the values of $\btheta_n \to_{n \to \infty} \btheta$. Suitably rescaling, and using that $X_{\btheta_n} - X_{\btheta}$ is an operator of order zero with $C^N(M_0)$ norm $o_N(1)$ for any $N > 0$, this contradicts the estimate for $X_{\btheta}$.

\subsection{Borel-Weil calculus} \label{ssection:bw-calculus}

The Borel-Weil calculus was introduced in \cite{Cekic-Lefeuvre-24} as a key tool to study $G$-equivariant (pseudo)differential operators on $G$-principal bundles $P_0 \to M_0$, when $G$ is a compact Lie group. Given a unitary representation $\rho \colon G \to \mathrm{U}(V_\rho)$, we denote by $E_\rho := M_0 \times_\rho V_\rho$ the corresponding associated vector bundle, defined as the set of equivalence classes $[z,\xi] \sim [z.g, \rho(g)^{-1}\xi]$, where $z \in P_0, \xi \in V_\rho$. 

\subsubsection{General description} Before detailing the construction, let us first explain the general idea of the Borel-Weil semiclassical calculus. Let $\mathbf{Q} \colon C^\infty(P_0) \to C^\infty(P_0)$ be a $G$-equivariant differential operator acting on functions on $P_0$. That is $\mathbf{Q}$ commutes with the right-action of $G$, i.e. $R_g^*\mathbf{Q}=\mathbf{Q}R_g^*$ for all $g \in G$. By the Peter-Weil theorem, any function $f \in C^\infty(P_0)$ can be decomposed into Fourier modes in each fiber of the principal bundle. Namely to any such $f$, it is possible to associate its Fourier transform
\begin{equation}
    \label{equation:fourier-decomposition}
\mc{F}(f) := \bigoplus_{\rho \in \widehat{G}} f_\rho,
\end{equation}
where the sum runs over $\widehat{G}$, the set all unitary irreducible representations of $G$, $f_\rho \in C^\infty(M_0, E_\rho^{\oplus \dim E_\rho})$ and $E_\rho \to M_{0}$ is the corresponding associated vector bundle. 

As $\mathbf{Q}$ is equivariant, it acts diagonally on the above decomposition, that is, it induces operators $Q_\rho \colon C^\infty(M_0,E_\rho) \to C^\infty(M_0,E_\rho)$. In certain problems, one is interested in the behaviour of the operator $Q_\rho$ as $\rho$ ``tends to infinity'' in the space of irreducible representations. However, this raises a number of issues, the first one being that the rank of $E_\rho$ diverges to $+\infty$ as $\rho$ ``increases'' unless $G$ is Abelian, thus making it very hard to keep track effectively of the behaviour of the operator.

A convenient solution is to realize geometrically $E_\rho$, using the Borel-Weil theorem, as a space of holomorphic sections of line bundles defined over a space associated with $P_0$ called the \emph{flag bundle}. More precisely, $V_\rho$ can be explicitly realized as $V_\rho \simeq H^0(G/T, J_\rho)$, where $T < G$ is a maximal torus and $J_\rho$ is a holomorphic line bundle induced by the representation. A section of $E_\rho \to M_0$ can then be identified with a section of the flag bundle $F_0 := P_0/T \to M_0$ with values in a certain complex line bundle $L_\rho$ which is \emph{fiberwise holomorphic}, that is holomorphic in restriction to every fiber $P_0(x_0)/T \simeq G/T$ where $x_0 \in M$. Here, $L_\rho$ is a complex line bundle whose restriction to $P_0(x_0)/T \simeq G/T$ is holomorphically equivalent to $J_\rho \to G/T$.

\subsubsection{Lie-theoretic preliminaries} \label{sssection:lie} See \cite[Section 2.1.2]{Cekic-Lefeuvre-24} for further discussion. Let $G$ be a compact Lie group and $\widehat{G}$ denote the set of all irreducible representations. The Lie algebra $\fg$ decomposes as
\[
\fg = \mathfrak{z}(\fg) \oplus \fg' = \mathfrak{z}(\fg) \oplus [\fg,\fg],
\]
where $\mathfrak{z}(\fg)$ denotes the centre of $\fg$. Let $T < G$ be a maximal torus. Its Lie algebra then splits as
\[
\mathfrak{t} = \mathfrak{z}(\fg) \oplus \mathfrak{t}',
\]
where $\mathfrak{t}' := \mathfrak{t} \cap \fg'$. Let $\mathfrak{a}^{\mathrm{ss}}_{+} \subset (i\mathfrak{t}')^{*}$ be the polyhedral convex positive cone spanned by the set of positive roots $\Delta^{+}(\mathfrak{g}_{\mathbb{C}})$ of the group
(the superscript $\mathrm{ss}$ stands for semisimple, and $(i\bullet)^*$ denotes real-linear functionals from $\bullet$ to $i\mathbb{R}$), and set
\[
\mathfrak{a}_{+} := (i\mathfrak{z}(\mathfrak{g}))^{*} \oplus \mathfrak{a}^{\mathrm{ss}}_{+}
\subset (i\mathfrak{t})^{*},
\]
the positive Weyl chamber.

We introduce
\[
a := \dim \mathfrak{z}(\fg), \qquad b:= \dim \mathfrak{t}', \qquad d := a+b = \mathrm{rank}(G).
\]
The set $\mc{A} \subset \mathfrak{t}$ of \emph{analytically integral weights} corresponds to
\[
\mc{A} := \{\alpha \in (i\mathfrak{t})^* ~:~ \alpha(H) \in 2\pi i \Z, \forall H \in \mathfrak{t}, \exp(H)=1\}.
\]
Let $\{\lambda_1,\dots,\lambda_a\}$ be a system of generators of $\mathfrak{z}(\mathfrak{g})$, chosen such that any $\alpha \in (i\mathfrak{z}(\mathfrak{g}))^{*} \cap \mathcal{A}$ can be written as
\[
\alpha = \sum_{i=1}^{a} k_i \lambda_i , \qquad k_i \in \mathbb{Z}.
\]
Let $\{\lambda_{a+1},\dots,\lambda_d\}$ be a system of generators of
\(\mathfrak{a}^{\mathrm{ss}}_{+} \cap \mathcal{A}\), chosen such that any
\(\alpha \in \mathfrak{a}^{\mathrm{ss}}_{+} \cap \mathcal{A}\) can be written as
\[
\alpha = \sum_{i=a+1}^{d} k_i \lambda_i , \qquad k_i \in \mathbb{Z}_{\geq 0}.
\]
Finally, we can consider the following surjective map
\[
\phi \colon \mathbb{Z}^{a} \times \mathbb{Z}_{\ge 0}^{\,d-a} \longrightarrow \mathfrak{a}_{+} \cap \mathcal{A},
\qquad
\bk \longmapsto \sum_{i=1}^{d} k_i \lambda_i.
\]
The map $\phi$ may fail to be injective. To any element $\phi(\bk)$, one can associate the (unique) irreducible representation of $G$ with highest weight is $\phi(\bk)$. Hence, $\widehat{G}$ is parametrized by $\mathbb{Z}^{a} \times \mathbb{Z}_{\ge 0}^{b}/\sim$, where $\bk \sim \bk'$ if and only if $\phi(\bk)=\phi(\bk')$.

The homogeneous space $G/T$ is called the \emph{flag manifold} associated with $G$; it is equipped with a natural complex structure. The Borel-Weil theorem asserts that the representation with highest weight $\phi(\bk)$ can be realized as
\[
H^0(G/T, \mathbf{J}^{\otimes \mathbf{k}}), \qquad \mathbf{J}^{\otimes \mathbf{k}} := J_1^{\otimes k_1} \otimes \dotsb \otimes J_d^{\otimes k_d}, 
\]
where $J_i \to G/T$ is a holomorphic Hermitian line bundle over $G/T$. For $i=1,\dotsc,a$, $J_i \to G/T$ is trivial. The left action of $G$ on $G/T$ lifts naturally to a fiberwise (isometric) action on $H^0(G/T,\mathbf{J}^{\otimes \mathbf{k}})$. This left action corresponds precisely to the irreducible representation of highest weight $\phi(\bk)$.

We emphasize that if there exists $k_i \neq 0$ for some $i \in \{a+1,\dotsc,d\}$, then $\mathbf{J}^{\otimes \mathbf{k}} \to G/T$ is not topologically trivial (see \cite[Proposition 2.1.8]{Cekic-Lefeuvre-24}). If $k_{a+1}=\dotsb=k_d=0$, then $\mathbf{J}^{\otimes \mathbf{k}} \to G/T$ is trivial and the space of holomorphic sections $H^0(G/T, \mathbf{J}^{\otimes \mathbf{k}})$ reduces to $\C$.

\subsubsection{Definition} The previous construction can be extended to the principal bundle $P_0 \to M_0$ as each fiber is isomorphic to $G$. We let $F_0 := P_0/T$ be the flag bundle over $M_0$ and $\mathbf{L}^{\otimes \bk} \to F_0$ be the line bundles obtained by the above construction. Each fiber $P_0(x)/T \simeq G/T$ for $x \in M$ is a complex manifold. In addition, the line bundles $\mathbf{L}^{\otimes \bk} \to F_0$ are holomorphic in restriction to each fiber. We can thus define $C^\infty_{\hol}(F_0,\Lk)$ as the space of \emph{fiberwise holomorphic} sections on $F_0$ with values in $\Lk$, that is such that the restriction to each fiber $P_0(x)/T$ is holomorphic. There is also an $L^2$-orthogonal projector onto the space of fiberwise holomorphic sections $C^\infty_{\hol}(F_0,\Lk)$
\[
\Pi_{\bk} \colon C^\infty(F_0,\Lk) \to C^\infty_{\hol}(F_0,\Lk), \qquad L^2(F_0,\Lk) \to L^2_{\hol}(F_0,\Lk).
\]

The Fourier transform \eqref{equation:fourier-decomposition} should then be seen as a map
\begin{equation}
    \label{equation:fourier-group}
\mc{F} \colon C^\infty(P_0) \to \bigoplus_{\mathbb{Z}^{a} \times \mathbb{Z}_{\ge 0}^{b}/\sim} C^\infty_{\hol}(F_0,\Lk)^{\oplus d_{\bk}},
\end{equation}
where $d_{\bk}$ denotes the dimension of the corresponding irreducible representation with highest weight $\phi(\bk)$. It is an isometry for the $L^2$-scalar product.

We now further assume that $P_0$ is equipped with a $G$-equivariant connection. This connection induces an Ehresmann connection on $F_0$, that is a splitting
\[
TF_0 = \HH \oplus \V,
\]
where $\V = \ker d\pi_{F_0}$ is the vertical fiber, $p_{F_0} \colon F_0 \to M_0$ is the footpoint projection, and $\HH$ is the horizontal space provided by the connection. We also introduce $\HH^*,\V^* \subset T^*F_0$ such that $\HH^*(\V)=0=\V^*(\HH)$. The map $dp_{F_0} \colon \HH \to TM_0$ is an isomorphism and therefore
\begin{equation}
\label{equation:isomorphism-h}
    dp_{F_0}^\top \colon T^*M_0 \to \HH^*
\end{equation}
is an isomorphism as well.

In addition, the $G$-equivariant connection on $P_0$ also induces a connection on each line bundle $\Lk \to F_0$. We are therefore in the framework of twisted quantization where one can consider the class $\Psi^m_{h,\bk}(F_0,\mathbf{L})$ of twisted semiclassical pseudodifferential operators originally introduced by Charles \cite{Charles-00} for tensor powers of a single line bundle. See also \cite[Section 3.2]{Cekic-Lefeuvre-24} for a complete discussion. An element $\mathbf{A} \in \Psi^m_{h,\bk}(F_0,\mathbf{L})$ is a \emph{family} of operators 
\[
\mathbf{A}_{h,\bk} \colon C^\infty(F_0, \Lk) \to C^\infty(F_0,\Lk)
\]
defined for $h > 0$ and $\bk \in \Z^{a} \times \Z^b_{\geq 0}$ such that $h|\bk|\leq 1$, with the following property. For any contractible open subset $U \subset F_0$, if $s_i \in C^\infty(F_0,L_i)$ are local trivializing sections of pointwise norm $|s_i| = 1$, then there exists a family of standard $h$-semiclassical pseudodifferential operators $A_{h,\bk} \in \Psi^m_h(U)$ such that for all $f \in C^\infty_{\comp}(U)$:
\[
\mathbf{A}_{h,\bk}(f \mathbf{s}^{\bk}) = A_{h,\bk}(f) \mathbf{s}^{\bk},
\]
where equality holds on $U$.

To introduce the Borel-Weil calculus, we need a preliminary notion:

\begin{definition}[Admissible operators]
Let $\mathbf{A} \in \Psi^m_{h,\bk}(F_0,\mathbf{L})$. The operator is \emph{admissible} if $[\mathbf{A},\Pi_{\bk}] \in h^\infty \Psi^{-\infty}_{h,\bk}(F_0,\mathbf{L})$.
\end{definition}

This means that the family of operators $\mathbf{A}_{h,\bk}$ preserve fiberwise holomorphic sections modulo negligible remainders. This leads to the following definition:

\begin{definition}[Borel-Weil operators]
    For $m \in \R$, the class of Borel-Weil semiclassical operators of order $m$ on $P_0$ is defined as
    \[
    \Psi^m_{h,\mathrm{BW}}(P_0) := \{(\Pi_{\bk} \mathbf{A} \Pi_{\bk})_{h > 0, \bk \in \Z^a \times \Z^b_{\geq 0}, h|\bk|\leq 1} ~:~ \mathbf{A} \in \Psi^m_{h,\bk}(F_0,\mathbf{L}) \text{ is admissible}\}.
    \]
\end{definition}

It can be verified that $\Psi^\bullet_{h,\mathrm{BW}}(P_0)$ is indeed an algebra, see \cite[Lemma 3.3.11]{Cekic-Lefeuvre-24}. All $G$-equivariant differential operators naturally belong to $\Psi^m_{h,\mathrm{BW}}(P_0)$. The phase space corresponding to this quantization is $\HH^* \subset T^*F_0$. Observe that when the bundle $P_0 \to M_0$ is trivial, that is $P_0 \simeq M_0 \times G$, the phase space is isomorphic to $\HH^* \simeq T^*M_0 \times G/T$. One should think of operators in the Borel-Weil calculus as standard semiclassical pseudodifferential operators in the base variable (i.e. on $M_0$) \emph{with values in} Toeplitz operators on the complex manifold $G/T$. This statement can be made precise, see \cite[Theorem 3.3.4]{Cekic-Lefeuvre-24}.

To any $\mathbf{A} \in \Psi^m_{h,\mathrm{BW}}(P_0)$, one can associate a principal symbol $\sigma_{\mathbf{A}} \in S^m(\HH^*)$. This is nothing but the restriction of the principal symbol of $\mathbf{A}$ to $\HH^*$, where $\mathbf{A}$ is seen as an operator in the twisted algebra $\Psi^m_{h,\bk}(F_0,\mathbf{L})$. Conversely, given $a \in S^m(\HH^*)$, there is a well-defined quantization procedure $\mathbf{A} := \Op_h^{\mathrm{BW}}(a) \in \Psi^m_{h,\mathrm{BW}}(P_0)$ constructing an operator in the Borel-Weil calculus such that $\sigma_{\mathbf{A}} = a$. All standard semiclassical results hold for Borel-Weil operators (invertibility of elliptic operators, propagation of singularities, Calderon-Vaillancourt, etc.), see \cite[Section 3.3]{Cekic-Lefeuvre-24} for further discussion. Finally, we note that the above construction can be carried out more generally for operators acting on sections of a vector bundle, see \cite[Remark 3.3.17]{Cekic-Lefeuvre-24}.

\section{Abelian extensions of Anosov flows}

\label{section:abelian-extension-anosov}

In this section, we prove the decay of correlations for Abelian extensions of Anosov flows under the assumption that $d\alpha \neq 0$ (Theorem \ref{theorem:main1}).

\subsection{Preliminaries}

Recall that $X_M$ denotes the generator of the lifted flow on $M$ and that we denote by $X_{M_0}$ the generator of the flow on $M_0$.

\subsubsection{Resolvents}\label{sssection:conjugacy} Using the isomorphism
\begin{equation}
    \label{equation:iso}
C^\infty(M_0) \to C^\infty_{\btheta}(M), \qquad f \mapsto \pi^*f \cdot s_{\btheta},
\end{equation}
introduced in \S\ref{ssection:floquet-theory}, we can compute, for $F_{\btheta} = \pi^*f \cdot s_{\btheta}$:
\begin{equation}\label{eq:conjugacy-operators}
X_{M}F_{\btheta}=X_{M}(\pi^{*}f)s_{\btheta}+\pi^{*}f X_{M}s_{\btheta}=\pi^{*}(X_{M_{0}}f+i\iota_{X_{M_{0}}}\eta_{\btheta}f)s_{\btheta},
\end{equation}
where in the last equality we also used Lemma \ref{lemma:stheta}. In other words, the induced operator $X_M \colon C^\infty_{\btheta}(M) \to C^\infty_{\btheta}(M)$ is conjugate \emph{via} \eqref{equation:iso} to the operator
\[ 
X_{\btheta} \colon C^{\infty}(M_{0}) \to C^{\infty}(M_{0}), \qquad X_{\btheta}f=X_{M_{0}}f+i\iota_{X_{M_{0}}}\eta_{\btheta}f.
\]
We also record a consequence of \eqref{eq:conjugacy-operators} on the corresponding propagators for future use:
\begin{equation}\label{eq:conjugacy-propagators}
    e^{tX_M} F_{\btheta} = \pi^*(e^{t X_{\btheta}} f) s_{\btheta},\quad t \in \mathbb{R}.
\end{equation}

Making a different choice of $s_{\btheta}$ amounts to considering $s_{\btheta}' = \pi^*q \cdot s_{\btheta}$ for some $q \in C^\infty(M_0)$ of pointwise unit norm. The corresponding $1$-form then satisfies $i\eta_{\btheta}' = i \eta_{\btheta} + \frac{dq}{q}$, while the operator $X_{\btheta}'$ becomes
\begin{equation}\label{eq:conjugated}
    X_{\btheta}' = X_{\btheta} + \frac{dq}{q} = q^{-1} X_{\btheta} q,
\end{equation}
i.e. $X_{\btheta}'$ is conjugate to $X_{\btheta}$ by the multiplication operator by the smooth function $q$. 

Since $\eta_{\btheta}$ is real and $X_{M_{0}}$ is volume-preserving, the operator $X_{\btheta}$ is skew-adjoint and its propagator $(e^{tX_{\btheta}})_{t \in \R}$ is unitary on $L^2(M_0)$. The resolvents
\begin{equation} \label{eq:resolvent_def}
    R_{\btheta}^{\pm}(z) := (\mp X_{\btheta}-z)^{-1}=-\int_0^{+\infty} e^{-t z} e^{\mp t X_{\btheta}}\, \dd t
\end{equation}
are therefore well-defined, bounded and holomorphic on $L^{2}(M_{0})$ for $\Re(z) > 0$. Note that, equivalently, the resolvents $R_{\btheta}^{\pm}(z)$ can be seen as operators acting on $C^\infty(M_0,L_{\btheta})$ using the correspondence \eqref{eq:fourier-theory}. For simplicity, we will always consider these operators as acting on functions on $M_0$ rather than on sections of the line bundle $L_{\btheta} \to M_0$.

\subsubsection{Anisotropic Sobolev spaces} \label{sssection:anisotropic-spaces}
We recall the constructions of dynamical anisotropic Sobolev spaces. For more details, see \cite[Chapter 9]{Lefeuvre-book}. In what follows, $E_u^* \subset T^*M_0$ and $E_s^* \subset T^*M_0$ are, respectively, the annihilators of $E_u \oplus \mathbb{R}X_{M_0}$ and $E_s \oplus \mathbb{R} X_{M_0}$. Recall that $g_{0}$ induces a natural fiberwise inner product on $T^* M_{0}$, which we denote by the same letter. We will work on the radial compactification $\overline{T^{*}M_{0}}=T^{*}M_{0} \sqcup \partial_{\infty}T^{*}M_{0}$. Fix an order function $m$, that is, a symbol in $S^{0}(T^{*}M_{0})$, which is 0-homogeneous in the $\xi$-variable for $|\xi| \gg 1$ large enough such that $m \equiv 1$ near $\overline{E_s^*} \cap \partial_{\infty} T^* M_0$, $m \equiv-1$ near $\overline{E_u^*} \cap \partial_{\infty} T^* M_0$. Let $H$ be the Hamiltonian vector field associated with the Hamiltonian $p(x,\xi) := \xi(X_{M_{0}}(x))$, that is, the generator of the the symplectic lift of $(\varphi_t^0)_{t \in \mathbb{R}}$ to $T^*M_0$, given by $\Phi_t^0(x, \xi) = (\varphi_t^0 x, (d\varphi^0_{-t})^{\top}\xi)$. We also introduce the \emph{escape function}
\[
G_m(x, \xi):=m(x, \xi) \log \langle\xi\rangle_{g_0}, 
\]
where $\langle \xi \rangle_{g_0} := (1+|\xi|^2_{g_0})^{1/2}$ is the Japanese bracket. This function can be constructed such that the following properties hold (see \cite[Lemma 9.1.9]{Lefeuvre-book}):
\begin{enumerate}[label=\roman*)]
    \item $H G_m(x, \xi) \leq 0$ near $\partial_{\infty} T^{*} M_{0}$;
    \item $H G_m(x, \xi) \leq-C<0$ near $\partial_{\infty} T^{*} M_{0} \cap (\overline{E_s^*} \cup \overline{E_u^*})$, for some constant $C>0$.
\end{enumerate}
We now introduce $A_{s}:=\Op_{h}(e^{s G_{m}})$ and the \emph{semiclassical anisotropic Sobolev spaces}
\[
\mathcal{H}_{h,\pm}^s(M_{0}):=A_{s}^{\mp 1}(L^2(M_{0}))
\]
These spaces are Hilbert spaces when equipped with the natural choice of norm
\[
\|u\|_{\mathcal{H}_{h,\pm}^s(M_0)} := \|A_s^{\pm 1} u\|_{L^2(M_0)}.
\]
Additionally, we have the dense inclusion $C^{\infty}(M_{0}) \subset \mathcal{H}_{h,\pm}^{s}(M_{0})$ and the latter is a subspace of distributions $\mathcal{D}'(M_{0})$.

When $h=1$, we will simply write $\mathcal{H}_{\pm}^s(M_0)$. For different values of $h > 0$, the spaces $\mc{H}_{h,\pm}^s(M_0)$ are all equal to $\mc{H}_{\pm}^s(M_0)$ as spaces, with equivalent norms. More precisely, the following holds: there exists $C > 0$, such that for all $u \in \mc{H}_{h, \pm}^s(M_0)$:
\begin{equation}
    \label{equation:equivalence-norms}
    \|u\|_{\mathcal{H}_{h,\pm}^s(M_0)} \leq Ch^{-s}  \|u\|_{\mathcal{H}_{\pm}^s(M_0)}, \qquad \|u\|_{\mathcal{H}_{\pm}^s(M_0)} \leq C h^{-s}\|u\|_{\mathcal{H}_{h,\pm}^s(M_0)},
\end{equation}
see \cite[Lemma 4.3]{Guillarmou-Kuster-21} or \cite[Page 126]{Cekic-Lefeuvre-24}.

\subsubsection{Meromorphic extension of the resolvent}

The usefulness of these spaces can be seen through the following result due to Faure-Sjöstrand \cite{Faure-Sjostrand-11} (see also \cite{Dyatlov-Zworski-16}): 

\begin{lemma} \label{theorem:meromorphic-extension}
    There exists a constant $c>0$ such that for any $s>0$, any $\btheta \in \mathrm{U}(1)^d$, and for any $h > 0$ the families of operators 
    \[ z \mapsto R^{\pm}_{\btheta}(z)  \colon \mathcal{H}_{h,\pm}^{s}(M_{0}) \to \mathcal{H}_{h,\pm}^{s}(M_{0}), \]
    admit a meromorphic extension to $\{\Re(z)>-cs\}$. The poles are contained in the half-plane $\{\Re(z) \leq 0\}$, and they are independent of the choices made in the construction.
\end{lemma}

We will say that $z\in \C$ is a \emph{resonance} (resp. \emph{coresonance}) of $X_{\btheta}$ if there exists $s$ large enough such that $\ker(-X_{\btheta}-z) \neq \{0\}$ on $\mathcal{H}_{h, +}^{s}(M_{0})$ (resp. $\ker(X_{\btheta}-z) \neq \{0\}$ on $\mathcal{H}_{h, -}^{s}(M_{0})$). That $z \in \C$ is a resonance is independent of the choices made in the construction of these spaces, on the value of $s$ (provided it is large enough), and on the value of $h > 0$ (see \cite[Section 9.1.1.7]{Lefeuvre-book}). It is also independent on the choice of the trivialising section $s_{\btheta}$ since other choices give conjugate operators. A distribution $u \in \ker(-X_{\btheta}-z)|_{\mathcal{H}_{h, +}^{s}}$ (resp. $u \in \ker(X_{\btheta}-z)|_{\mathcal{H}_{h, -}^{s}}$) is called a \emph{resonant state} (resp. \emph{coresonant state}); it can be shown that any such $u$ actually satisfies $\WF(u) \subset E_u^*$ (resp. $\WF(u) \subset E_s^*$), that is $u \in \mc{D}'_{E_u^*}(M_0):=\{u \in \mc{D}'(M_0) \mid \WF(u) \subset E_u^*\}$ (resp. $u \in \mc{D}'_{E_s^*}(M_0):=\{u \in \mc{D}'(M_0) \mid \WF(u) \subset E_s^*\}$).

\subsection{Resonances close to the imaginary axis} We begin by studying resonances on the imaginary axis and close to $0$.

\subsubsection{Resonances on the imaginary axis} We recall that the smooth volume form $\vol_{M_{0}}$ preserved by the flow on $M_0$ is normalized such that $\vol(M_0) = 1$. In what follows we make the assumption that $\eta_{\mathbf{0}} \equiv 0$; then $X_{\mathbf{0}} = X_{M_0}$. 

\begin{lemma} \label{lemma:no-resonances-imaginary-axis}
    Assume that $d\alpha \neq 0$. If $\btheta \neq 0$, there are no resonances for $X_{\btheta}$ on $i\R$. If $\btheta = \mathbf{0}$, then $0$ is the only resonance with $1$-dimensional resonant space spanned by the constant function $\mathbf{1}$ on $M_0$. The associated spectral projector is $\Pi_{\mathbf{0}} = \langle \bullet, \mathbf{1}_{M_0}\rangle \mathbf{1}_{M_0}$.
\end{lemma}

\begin{proof}
For $\btheta=\mathbf{0}$, by assumption we have $X_{\mathbf{0}} = X_{M_0}$. By ergodicity of $(\varphi_t^0)_{t \in \R}$ on $M_0$, $0$ is a resonance for $X_{M_{0}}$ with $1$-dimensional resonant space spanned by the constant function, see \cite[Lemma 9.2.5]{Lefeuvre-book}. In addition, by the Anosov alternative, there are other resonances on $i\R$ if and only if the flow is a suspension of an Anosov diffeomorphism by a constant roof function, see \cite[Lemma 9.2.8]{Lefeuvre-book}. However, this is excluded as $d\alpha \neq 0$, that is the stable and unstable bundles are not jointly integrable over $M_0$.

Assume now that $\btheta \neq \mathbf{0}$ and that there is a resonance on $i\R$. That is there exists a non-zero $u \in \mc{D}'_{E_u^*}(M_0)$, such that $(-X_{\btheta}-i\lambda)u = 0$ for some $\lambda \in \R$. Then by \cite[Proposition 9.2.3, item (ii)]{Lefeuvre-book} (see also \cite[Lemma 2.3]{Dyatlov-Zworski-17}), $u$ is smooth. After integration this implies that
\[
    (\varphi_t^0)^*u = e^{-i\lambda t} e^{-i\int_0^t \iota_{X_{M_{0}}} (\varphi_s^0)^*\eta_{\btheta}\, \dd s}u,\quad t \in \mathbb{R},
\]
and after applying the exterior derivative, as well as some manipulations we get that
\[
    (\varphi_t^0)^*du + i ((\varphi_t^0)^* \eta_{\btheta}) u e^{-i\lambda t} e^{-i\int_0^t \iota_{X_{M_{0}}} (\varphi_s^0)^*\eta_{\btheta}\, \dd s} = e^{-i\lambda t} e^{-i\int_0^t \iota_{X_{M_{0}}} (\varphi_s^0)^*\eta_{\btheta}\, \dd s}(du + i \eta_{\btheta} u),
\]
where we also used the fact that $d\eta_{\btheta} = 0$ as well as Cartan's magic formula. Applying this to some vector $v \in E_s$, taking $t \to +\infty$, and using the definition of Anosov flows, gives that $(du + i \eta_{\btheta} u)(v) = 0$. This works similarly if $v \in E_u$ by taking $t \to -\infty$. (Alternatively, here we could use an argument similar to Lemma 4.7 below which shows horocyclic invariance, combined with the fact that the connection $d + i\eta_{\btheta}$ is flat.) 

From the discussion above it follows that
\begin{equation}
    \label{equation:no-res}
    (-d-i\eta_{\btheta}\wedge-i\lambda \alpha \wedge) u = 0,
\end{equation}
where $\alpha$ is the Anosov $1$-form. It follows that $d(|u|^2) = 0$ and so we may assume that $|u| \equiv 1$. Hence $\lambda \alpha + \eta_{\btheta} = i du/u$. In particular, applying $d$, we find that $\lambda d \alpha = 0$ which if $\lambda \neq 0$ implies $d\alpha \equiv 0$ contradicting our assumption. If $\lambda = 0$, then we find that $\eta_{\btheta} = i du/u$, so $\int_\gamma \eta_{\btheta} \in 2\pi \mathbb{Z}$ for every closed curve $\gamma \in H_1(M_0,\Z)$. But by construction $\int_\gamma \eta_{\btheta} - \rho(\gamma)\cdot\btheta \in 2\pi \mathbb{Z}$, so since $\rho$ is surjective we get $\btheta = \mathbf{0}$ inside $\mathrm{U}(1)^d$. However, for $\btheta = \mathbf{0}$, there are no resonances on $i\R$ except $0$ by the first paragraph, which completes the proof. (For an alternative proof not using the first paragraph, simply observe that the equation $du = 0$ has a unique solution $u = \mathbf{1}_{M_0}$ up to multiplication by a constant, which shows that $\ker(-X_{\mathbf{0}})|_{\mc{H}^s_+}$ is $1$-dimensional. Also, since $\|(-X_{\mathbf{0}} - z)^{-1}\|_{L^2 \to L^2} = \frac{1}{\Re(z)}$ for $\Re z > 0$, there cannot be any Jordan blocks for $-X_{\mathbf{0}}$, showing that the multiplicity of the resonant space at $z = 0$ is one.)
\end{proof}

\subsubsection{Variance. Laurent expansion near $0$}

As seen above in Lemma \ref{lemma:no-resonances-imaginary-axis}, for $\btheta = \mathbf{0}$, $z = 0$ is a simple resonance of $X_{\mathbf{0}}$. We have the following Laurent expansions near $0$:
\begin{equation}\label{eq:resolvents-at-zero}
R^{+}_{\btheta=\mathbf{0}}(z) = R^{+, \mathrm{hol}}_{\btheta=\mathbf{0}}(z) - \dfrac{\Pi_{\mathbf{0}}^{+}}{z}, \qquad R^{-}_{\btheta=\mathbf{0}}(z) = R^{-, \mathrm{hol}}_{\btheta=\mathbf{0}}(z) - \dfrac{\Pi_{\mathbf{0}}^{-}}{z},
\end{equation}
where the resolvent was introduced in \eqref{eq:resolvent_def}. Note also that, applying $\mp X_{M_{0}} - z$ to the two identities we get
\begin{equation}\label{eq:holomorphic-part-at-zero-identity}
    -X_{M_{0}} R_{\mathbf{0}}^{+, \mathrm{hol}}(0) = \id - \Pi_{\mathbf{0}}^+,\quad X_{M_{0}} R_{\mathbf{0}}^{-, \mathrm{hol}}(0) = \id - \Pi_{\mathbf{0}}^{-}.
\end{equation}
This allows to introduce Guillarmou's covariance operator (see \cite{Guillarmou-17-1} where it was first studied):
\begin{equation}
\label{equation:pi}
\Pi := -\left(R^{+, \mathrm{hol}}_{\btheta=\mathbf{0}}(0) + R^{-, \mathrm{hol}}_{\btheta=\mathbf{0}}(0)\right)
\end{equation}

This operator has a dynamical interpretation which we now describe. Given $f \in C^\infty(M_0)$, recall that its \emph{variance} (computed with respect to the volume measure) is defined as:
\[
\mathrm{Var}_{\vol}(f) := \lim_{T \to +\infty} \dfrac{1}{T} \int_{M_0} \left( \int_0^T f(\varphi_t x)\, \dd t - T \int_{M_0} f(x)\, \dd\vol(x)\right)^2 \dd\vol(x) \geq 0.
\]
It is then possible to show that:
\begin{equation} \label{eq:variance_pi}
    \mathrm{Var}_{\vol}(f) = \langle \Pi f, f\rangle_{L^2(M_0,\vol_{M_{0}})}.
\end{equation}
See \cite[Equation (2.5)]{Guillarmou-Knieper-Lefeuvre-22} for a proof. Also, variance zero functions need to be coboundaries, namely
\begin{equation}\label{eq:variance-zero-coboundary}
    \langle{\Pi f, f}\rangle_{L^2} = 0 \iff \exists u \in C^\infty(M_0),\, \mathrm{s.t.}\, f = Xu,
\end{equation}
see \cite[Theorem 1.1]{Guillarmou-17-1}.

\subsubsection{Leading resonance}\label{ssec:leading-resonance}

For $\btheta = \mathbf{0}$, $\lambda_{\btheta = \mathbf{0}} = 0$ is a resonance of $X_{\btheta=\mathbf{0}}$ with $1$-dimensional resonant space. It follows from the perturbation theory of operators, see \cite[Corollary 2]{Bonthonneau-19} and \cite{Cekic-Lefeuvre-24-stability}, that there exists a smooth map $\btheta \mapsto \lambda_{\btheta} \in \C$, well-defined for $\btheta \in U \subset \mathrm{U}(1)^d$ close to $\mathbf{0}$, such that $\lambda_{\btheta}$ is the only resonance of $X_{\btheta}$ near $z = 0$. We call $\lambda_{\btheta}$ the \emph{leading resonance} of $X_{\btheta}$. As $X_{\btheta}$ is skew-adjoint on $L^2(M_0)$, we have $\lambda_{\btheta} \in \{\Re(z) \leq 0\}$.

We also denote by $\Pi_{\btheta}^+ \colon C^\infty(M_0) \to \mc{H}^s_{+}$ the spectral projector onto this resonant state. We recall from \cite[Lemma 3.5]{Cekic-Lefeuvre-24-stability} that there exists a (small) closed positively oriented curve $\gamma_0 \subset \mathbb{C}$ enclosing $0$ such that for $\btheta$ in a neighbourhood of $\mathbf{0} \in \mathrm{U}(1)^d$, we have
\begin{equation}\label{eq:projector-near-the-origin}
    \Pi_{\btheta}^+ = -\frac{1}{2\pi i}\oint_{\gamma_0} R^+_{\btheta}(z) \dd z,
\end{equation}
Note also that we can write $\Pi_{\btheta}^+ = \langle \bullet, v_{\btheta}\rangle u_{\btheta}$, where $u_{\btheta} \in \mc{H}^s_-$ is a resonant state such that $(-X_{\btheta}-\lambda_{\btheta}) u_{\btheta} = 0$, $v_{\btheta} \in \mc{H}^s_-$ is a coresonant state satisfying $(X_{\btheta} - \overline{\lambda_{\btheta}}) v_{\btheta} = 0$, and $\langle u_{\btheta},v_{\btheta}\rangle = 1$. In addition, $u_{\mathbf{0}} = \mathbf{1}_{M_0}$ is the constant function equal to $1$ everywhere, $v_{\mathbf{0}} = \mathbf{1}_{M_0}$, and the maps $\btheta \mapsto u_{\btheta} \in \mc{H}^s_+$ and $\btheta \mapsto v_{\btheta} \in \mc{H}^s_-$ are smooth by \cite[Lemma 3.5]{Cekic-Lefeuvre-24-stability}. Here we write $\langle{u, v}\rangle = \int_{M_0} uv ~\dd \vol_{M_0}$ which integrates $u$ and $v$ against the flow-invariant smooth measure. We now record a simple lemma about the symmetries of $\lambda_{\btheta}$.

\begin{lemma}\label{lemma:leading-resonance-symmetry}
    For $\btheta \in U$, we have $\overline{\lambda_{\btheta}} = \lambda_{-\btheta}$. When $(\varphi_t^0)_{t \in \mathbb{R}}$ is an Anosov geodesic flow, then $\lambda_{\btheta} \in \mathbb{R}$. 
\end{lemma}
\begin{proof}
    First of all, since $X_{\btheta}^* = -X_{\btheta}$, we have $(R_{\btheta}^+(z))^* = R_{\btheta}^-(\bar{z})$, so $(\Pi^+_{\btheta})^* = \Pi_{\btheta}^-$ and we conclude that $\overline{\lambda_{\btheta}}$ is a coresonance of $-X_{\btheta}$, i.e. $\ker(X_{\btheta} - \overline{\lambda_{\btheta}})|_{\mc{H}^s_-}$ is one dimensional. By applying complex conjugation to $(-X_{\btheta} - \lambda_{\btheta}) u_{\btheta} = 0$ we get
    \[
        (-X_{M_0} + i\iota_{X_{M_{0}}} \eta_{\btheta} - \overline{\lambda_{\btheta}}) \overline{u_{\btheta}} = 0,
    \]
    and we have by the properties of the wavefront set that $\WF(\overline{u_{\btheta}}) \subset E_u^*$. Thus $\overline{u_{\btheta}}$ is a resonant state for $-X_{M_0} + i\iota_{X_{M_{0}}} \eta_{\btheta}$. Note that by the definition of $\eta_{\btheta}$ we have $\int_\gamma (\eta_{\btheta} + \eta_{-\btheta}) \equiv 0 \mod 2\pi$, and so there exists $q \in C^\infty(M_0)$ of pointwise unit norm such that $i(\eta_{\btheta} + \eta_{-\btheta}) = \frac{dq}{q}$ which similarly to \eqref{eq:conjugacy-operators} implies that $-X_{M_0} + i\iota_{X_{M_0}} \eta_{\btheta}$ is conjugate to $-X_{-\btheta}$. Thus $-X_{-\btheta}$ has a resonance at $\overline{\lambda_{\btheta}}$ which must coincide with $\lambda_{-\btheta}$ by uniqueness of resonances enclosed by $\gamma_0$.

    For the second claim when $M_0 = SN_0$ is the unit tangent bundle of a Riemannian manifold $N_0$, the antipodal map $R : M_0 \to M_0$, $R(x, v) = (x, -v)$ satisfies $R^*X_{M_0} = -X_{M_0}$ and maps $E_u^*$ to $E_s^*$ and vice versa. Also since the pullback map $\pi^*$ (where $\pi : M_0 \to N_0$ is the projection) is an isomorphism from $H^1(N_0, \mathbb{R})$ to $H^1(M_0, \mathbb{R})$, we can also assume that $\eta_{\btheta}$ is a pullback $1$-form, so it is invariant under $R^*$. This shows that $R^*u_{\btheta}$ is also a coresonant state, and so $\overline{\lambda_{\btheta}} = \lambda_{\btheta}$, completing the proof.
\end{proof}

We now compute the first and second order derivatives of the leading resonance at $\btheta = \mathbf{0}$:

\begin{lemma}\label{lemma:resonance-non-degenerate}
The following holds, for any $v \in T_{\mathbf{0}} \mathrm{U}(1)^d \simeq \mathbb{R}^d$
\[
    \lambda'_{\mathbf{0}}(v) = 0, \qquad \lambda''_{\mathbf{0}}(v, v) = - \langle \Pi \iota_{X_{M_{0}}} D_v\eta_{\mathbf{0}}, \iota_{X_{M_{0}}} D_v \eta_{\mathbf{0}}\rangle = - \mathrm{Var}_{\vol_{M_{0}}}(\iota_{X_{M_{0}}} D_v\eta_{\mathbf{0}}).
\]
Also, the we have $-\lambda''_{\mathbf{0}}$ is positive-definite.
\end{lemma}

Here $\lambda'_{\mathbf{0}}(v) = D_v \lambda_{\mathbf{0}}$ denotes the directional derivative at $\btheta = \mathbf{0}$ in the direction $v$, and $\lambda''_{\mathbf{0}}(v, v) = D^2_v \lambda_{\mathbf{0}}$ denotes the Hessian also at $\btheta = \mathbf{0}$ in the direction of $(v, v)$.

\begin{proof}
    We start with the equality
    \[
    (-X_{\btheta}-\lambda_{\btheta})u_{\btheta} = (-X_{M_{0}}-i\iota_{X_{M_0}}\eta_{\btheta}-\lambda_{\btheta}) u_{\btheta} = 0.
    \]
    We will take the directional derivatives $D_v$ in the direction of $v$ and denote the corresponding derivatives at $\btheta$ using the dot notation; note that all objects vary smoothly with respect to $\btheta$. Hence, differentiating twice, we find:
    \begin{align}
        \label{equation:diff1} (i\iota_{X_{M_{0}}} \dot{\eta}_{\btheta} - \dot{\lambda}_{\btheta}) u_{\btheta} + (-X_{M_{0}}-i\iota_{X_{M_{0}}} \eta_{\btheta} - \lambda_{\btheta}) \dot{u}_{\btheta} = 0, \\
        \label{equation:diff2}  (-i\iota_{X_{M_{0}}} \ddot{\eta}_{\btheta} - \ddot{\lambda}_{\btheta}) u_{\btheta} + 2(-i\iota_{X_{M_{0}}} \dot{\eta}_{\btheta} - \dot{\lambda}_{\btheta}) \dot{u}_{\btheta} + (-X_{M_{0}}-i\iota_{X_{M_{0}}} \eta_{\btheta} - \lambda_{\btheta}) \ddot{u}_{\btheta} = 0.
    \end{align}
    We now specialise to $\btheta = \mathbf{0}$. Using that $\lambda_{\mathbf{0}} = 0$, $u_{\mathbf{0}} = \mathbf{1}_{M_0}$, and that $\eta_{\mathbf{0}} = 0$, we find that $-i\iota_{X_{M_{0}}} \dot{\eta}_{\mathbf{0}} - \dot{\lambda}_{\mathbf{0}} = X_{M_{0}}\dot{u}_{\mathbf{0}}$. Integrating over $M_0$ with respect to the volume measure, and using that $\vol(M_0)=1$, as well as that $X_{M_{0}}$ is skew-adjoint, we get:
    \[
        0 = \dot{\lambda}_{\mathbf{0}} = - i \int_{M_0} \iota_{X_{M_{0}}}\dot{\eta}_{\mathbf{0}}\, \dd \vol_{M_0},
    \]
    where the first equality follows from the fact that $\btheta$ is a local maximum of the function $\btheta \mapsto \lambda_{\btheta}$ for $\btheta$ close to $\mathbf{0}$. (Alternatively, apply Lemma \ref{lemma:leading-resonance-symmetry} to get that all odd derivatives of $\lambda_{\btheta}$ vanish at $\btheta = \mathbf{0}$.) This implies that $X_{M_{0}} \dot{u}_{\mathbf{0}} = -i \iota_{X_{M_{0}}} \dot{\eta}_{\mathbf{0}}$ and since $\dot{u}_{\mathbf{0}}$ belongs to an anisotropic Sobolev space $\mc{H}_+^s(M_0)$ for a suitable $s > 0$ we get using \eqref{eq:holomorphic-part-at-zero-identity} that
    \begin{equation}\label{eq:resolvent-id-u-dot}
        \dot{u}_{\mathbf{0}} - \Pi_{\mathbf{0}}^+ \dot{u}_{\mathbf{0}} = R_{\mathbf{0}}^{+, \mathrm{hol}}(0)(i \iota_{X_{M_{0}}} \dot{\eta}_{\mathbf{0}}).
    \end{equation}

    Let $\gamma$ be a closed curve in $M_0$; recall that $\int_\gamma \eta_{\btheta} = \rho(\gamma)\cdot\btheta \mod 2\pi$. Differentiating with respect to $\btheta$ in the direction of $v$, we find that $\int_\gamma \dot{\eta}_{\btheta} = \rho(\gamma) \cdot v$ and $\int_\gamma \ddot{\eta}_{\btheta} = 0$; we conclude that $[\ddot{\eta}_{\btheta}]_{H^1} = 0$ and so $\ddot{\eta}_{\btheta} = 0$. Since $\eta_{\btheta}$ are harmonic with respect to $(M_0, g_0)$ by construction, see Proposition \ref{prop:contractible}, we conclude that $\ddot{\eta}_{\btheta}$ are also harmonic and so $\ddot{\eta}_{\btheta} \equiv 0$.
    Using \eqref{equation:diff2} with $\lambda_{\mathbf{0}} = 0$ and $\dot{\lambda}_{\mathbf{0}} = 0$, and $u_{\mathbf{0}} = \mathbf{1}_{M_0}$, we find 
    \[
        \ddot{\lambda}_{\mathbf{0}} = - 2i \iota_{X_{M_{0}}} \dot{\eta}_{\mathbf{0}} \dot{u}_{\mathbf{0}} - X_{M_{0}}\ddot{u}_{\mathbf{0}}.
    \]
    Integrating with respect to $\dd \vol_{M_0}$, and using \eqref{eq:resolvent-id-u-dot}, we find
    \[
    \ddot{\lambda}_{\mathbf{0}} = -2i \int_{M_0} \iota_{X_{M_{0}}}\dot{\eta}_{\mathbf{0}} \dot{u}_{\mathbf{0}}\, \dd \vol_{M_0} = 2 \langle R^{+, \mathrm{hol}}_{\mathbf{0}}(0) \iota_{X_{M_{0}}} \dot{\eta}_{\mathbf{0}}, \iota_{X_{M_0}} \dot{\eta}_{\mathbf{0}}\rangle,
    \]
    where in the second equality we also used that $\Pi_{\mathbf{0}}^+ \dot{u}_{\mathbf{0}}$ is a constant, that $\iota_{X_{M_0}} \dot{\eta}_{\mathbf{0}}$ integrates to zero, and that $\eta_{\btheta}$ is real-valued. Notice that 
    \[
        \langle R^{+, \mathrm{hol}}_{\mathbf{0}}(0) \iota_{X_{M_0}} \dot{\eta}_{\mathbf{0}}, \iota_{X_{M_0}} \dot{\eta}_{\mathbf{0}}\rangle = \langle \iota_{X_{M_0}} \dot{\eta}_{\mathbf{0}}, R^{-, \mathrm{hol}}_{\mathbf{0}}(0) \iota_{X_{M_0}} \dot{\eta}_{\mathbf{0}}\rangle = \langle R^{-, \mathrm{hol}}_{\mathbf{0}}(0) \iota_{X_{M_0}} \dot{\eta}_{\mathbf{0}}, \iota_{X_{M_0}} \dot{\eta}_{\mathbf{0}}\rangle,
    \]
    where in the first equality we used that $(R^{+, \mathrm{hol}}_{\mathbf{0}}(0))^* = R^{-, \mathrm{hol}}_{\mathbf{0}}(0)$, and in the second equality that $R^{-, \mathrm{hol}}_{\mathbf{0}}(0)$ commutes with the conjugation operator as well as that $\eta_{\btheta}$ is real-valued. Hence, using formula \eqref{eq:variance_pi}, we obtain the claimed result.

    We are left to check that $-\lambda''_{\mathbf{0}}$ is positive definite; so far, we showed that it is non-negative definite. Assume that $\lambda_{\mathbf{0}}''(v, v) = 0$ for some $v \in \mathbb{R}^d$; we will then show that $v = 0$. By \eqref{eq:variance-zero-coboundary}, and the computation above, we there exists $h \in C^\infty(M_0)$ such that $X_{M_0} h = \iota_{X_{M_0}}D_{v} \eta_{\mathbf{0}}$. It follows that $dh - D_v \eta_{\mathbf{0}}$ is invariant by the flow $(\varphi_t^0)_{t \in \mathbb{R}}$ and in the kernel of $\iota_{X_{M_0}}$. Taking an arbitrary vector $v_s \in E^s(x)$, and using the invariance we get
    \[
        \mc{O}(\|d\varphi_t v_s\|) = (dh - D_v \eta_{\mathbf{0}})(\varphi_t x)(d\varphi_t(v_s)) = (dh - D_v \eta_{\mathbf{0}})(x)(v_s),
    \]
    as $t \to \infty$. By definition of Anosov flows, this converges to zero, and we get $(dh - D_v \eta_{\mathbf{0}})(x)(v_s) = 0$. Similarly $(dh - D_v \eta_{\mathbf{0}})(x)(v_u) = 0$ for any $v_u \in E^u(x)$. Thus $dh \equiv D_v \eta_{\mathbf{0}}$, which implies that
    \[
        v \cdot \rho(\gamma) = \int_\gamma D_v \eta_{\mathbf{0}} = \int_\gamma dh = 0,
    \]
    and so since $\rho$ is surjective, this implies that $v = 0$. This completes the proof.
\end{proof}

\subsection{High frequency estimates}

We introduce the notation
\[
\mathbb{B} := \{z \in \C ~:~ \Re(z) \in [0,1]\}.
\]
The following result is key in the proof of Theorem \ref{theorem:main1}:

\begin{theorem}[High frequency estimates]
\label{theorem:high-frequency-ab-ext}
Assume that $d\alpha \neq 0$. Let $s > 0$. Then there exist $C,\vartheta > 0$ such that for all $z \in \mathbb{B}, |z| > 1$ and $\boldsymbol{\theta} \in \mathrm{U}(1)^d$:
\begin{equation} \label{equation:high-freq-estimate}
    \|( \mp X_{\boldsymbol{\theta}}-z)^{-1}\|_{\mathcal{H}_{\pm}^{s}(M_{0}) \to \mathcal{H}_{\pm}^{s}(M_{0})} \leq C \langle \Im(z)\rangle^{\vartheta}.
\end{equation}
\end{theorem}

We will actually prove that for $z \in \mathbb{B}$, $|z| > 1$ and $\btheta \in \mathrm{U}(1)^d$:
\begin{equation}
    \label{equation:toprove}
\|(\mp X_{\btheta}-z)^{-1}\|_{\mathcal{H}_{|\lambda|^{-1}, \pm}^{s}(M_{0}) \to \mathcal{H}_{|\lambda|^{-1}, \pm}^{s}(M_{0})} \leq C \langle \lambda \rangle ^{\ell},
\end{equation}
where $\lambda:=\Im(z)$ and $\ell>\dim M_0+10$. Using \eqref{equation:equivalence-norms}, the previous estimate clearly implies \eqref{equation:high-freq-estimate}.

To keep the notation simple, in what follows we will focus only on the `$+$' case, and we will drop the subindex $+$. For the sake of contradiction, let us assume that \eqref{equation:toprove} does not hold. Then there is a sequence $z_{n} \in \mathbb{B}$, $|z_n|>1$, and $f_{n} \in \mathcal{H}_{|\lambda_{n}|^{-1}}^{s}(M_{0})$, $n \in \mathbb{Z}_{\geq 0}$, such that 
\begin{equation} \label{equation:contradiction-1}
    \|f_{n}\|_{\mathcal{H}_{|\lambda_{n}|^{-1}}^{s}(M_{0})}=1, \quad \|(-X_{\btheta}-z)f_{n}\|_{\mathcal{H}_{|\lambda_{n}|^{-1}}^{s}(M_{0})}=o(\langle \lambda_{n} \rangle^{-\ell}),\quad n \to \infty,
\end{equation}
where $\lambda_n = \Im(z_n)$. We will further assume that $\lambda_{n} \geq 0$ (the other case is treated similarly). In addition, we can also assume (up to extraction of a subsequence) that $(z_n)_{n \geq 0}$ is unbounded, as $X_{\btheta}$ does not have resonances in $\mathbb{B} \cap \{|z|>1\}$. Moreover, up to a further subsequence, by compactness, we can assume that $\Re(z_{n}) \to \nu \in [0,1]$. Finally, we will write $h_{n}=|\lambda_{n}|^{-1}$, and we will drop the subindex on $h$. Thus, we can rewrite \eqref{equation:contradiction-1} as
\begin{equation} \label{equation:quasimodes}
    \|f_{h}\|_{\mathcal{H}_{h}^{s}(M_{0})}=1, \qquad \|P_{h}f_{h}\|_{\mathcal{H}_{h}^{s}(M_{0})}=o(h^{\ell+1}),
\end{equation}
where 
\[
P_{h}=-hX_{\btheta}-h(\nu+o(1))-i.
\]
The principal symbol of $P_{h}$ is $\sigma_{P_{h}}(x,\xi)=-i(\langle \xi,X_{M_{0}}(x)\rangle+1)$. Define
\[
\mathcal{T}=\{-\alpha(x) \mid x \in M_0\} \subset T^*M_0,
\]
where $\alpha$ is the Anosov 1-form. We begin by studying the microsupport of the quasimodes \eqref{equation:quasimodes}.

\begin{lemma}[Microlocalisation at the trapped set] \label{lemma:microsupport-quasimodes}
    Let $E_{h} \in \Psi_{h}^{\mathrm{comp}}(M_{0})$. Then
    \[
    \WF_{h}(E_{h}) \cap \mathcal{T} =\emptyset \Longrightarrow E_{h}f_{h}=o_{L^{2}}(h^{\ell/2})=o_{C^{0}}(h^{\ell/2-n/2}).
    \]
\end{lemma}

\begin{proof}
\emph{Step 1.} First, we deal with points outside of the characteristic set of $\sigma_{P_{h}}$: If $\WF_{h}(E_{h}) \cap \{\langle \xi,X_{M_{0}}(x) \rangle=-1\}=\emptyset$, then it follows from ellipticity of $P_{h}$ on $\WF_{h}(E_{h})$ (Lemma \ref{lemma:elliptic-rgularity}) and \eqref{equation:quasimodes} that $\|E_{h}f_{h}\|_{L^{2}}=o(h^{\ell+1})$. \\

\emph{Step 2.} Second, we propagate the estimate using the source properties of $E_{s}^{*}$. Let us assume that $\WF_{h}(E_{h}) \cap (\mathcal{T}+E_{u}^{*})=\emptyset$. We claim that $E_{h} f_{h}= o_{L^{2}}(h^{\ell})$. This follows from the source estimate (Lemma \ref{lemma:source-sink-estimates}). Indeed, taking $A_h \in \Psi_{h}^0(M_{0})$ elliptic and localized near $E_{s}^{*} \cap \partial_{\infty} T^* M$, we get for some $N > \ell$ that:
\begin{equation} \label{equation:propapagition-quasimodes}
    \|A_{h} f_{h}\|_{H_{h}^{s}} \leq C\left(h^{-1}\|B_h P_h f_{h}\|_{H_{h}^{s}}+h^{N}\|f_{h}\|_{H_{h}^{-N}}\right)=o(h^{\ell}),
\end{equation}
where we used that $B_{h}$ is microlocalized in a neighborhood of $E_{s}^{*}$, where $\mathcal{H}_{+}^{s}$ is microlocally $H^{s}$, so
\[ h^{-1}\|B_h P_h f_{h}\|_{H_{h}^{s}} \le Ch^{-1}\|P_{h}f_{h}\|_{\mathcal{H}_{+}^{s}}=o(h^{\ell}). \]
Observe that for all $(x, \xi) \in\{\langle\xi, X_{M_{0}}(x)\rangle=-1\}$ such that $(x, \xi) \notin \mathcal{T}+E_{u}^{*}$, there is a finite time $T \in \R$ such that $\Phi^{0}_{T}(x, \xi) \in \Ell_{h}(A_{h})$. By assumption, this holds for any $(x, \xi) \in \WF_{h}(E_{h})$. Hence, $A_{h}$ and $E_{h}$ satisfy the hypothesis of the propagation of singularities estimate (Lemma \ref{lemma:propagation-singularities}). This, together with \eqref{equation:propapagition-quasimodes} and \eqref{equation:quasimodes}, imply $E_{h} f_{h}=o_{L^{2}}(h^{\ell})$. \\

\emph{Step 3.} Finally, it remains to deal with $\mathcal{T}+E_{u}^{*}$. Let $A_{h} \in \Psi_{h}^{\text{comp}}(M_{0})$ be formally self-adjoint satisfying: 1) $\WF_{h}(A_{h})$ is localized near $\mathcal{T}+E_{u}^{*}$; 2) $\sigma_{A_{h}} \geq 0$, 3) $\sigma_{A_{h}} \equiv 1$ on a neighborhood of $\mathcal{T}$, 4) $A_h \geq 0$; in the sense that for all $f \in C^{\infty}(M_{0})$, $h > 0$, $\langle A_{h} f, f\rangle_{L^2} \geq 0$.

By \eqref{equation:quasimodes} and the fact that the propagator $t \mapsto e^{t P_{h} / {h}} \in \mathcal{L}(\mathcal{H}_{h}^{s})$ is uniformly bounded for $t \in[0, T]$ in a fixed compact interval for all $h > 0$, by Egorov's Theorem (Lemma \ref{lemma:egorov}), we obtain
\[ 
e^{t P_{h}/h} f_{h}=f_{h}+o_{\mathcal{H}_{h}^{s}}(h^{\ell}), \quad t \in[0, T] .
\]
Hence, for $t \in[0, T]$ we obtain
\[
\begin{split}
    \langle A_{h} f_{h}, f_{h}\rangle_{L^{2}} & =\langle A_{h} e^{t P_{h}/h} f_{h}, e^{tP_{h}/h} f_{h} \rangle_{L^{2}}+o(h^{\ell}) \\
    & =e^{-2(\nu+o(1)) t}\langle e^{t X_{\btheta}} A_{h} e^{-t X_{\btheta}} f_{h}, f_{h}\rangle_{L^{2}}+o(h^{\ell}).
\end{split}
\]
Using the dynamics of $(\Phi_{t}^{0})_{t \in \R}$ (for $t$ large enough the symbol of $A_{h}$ minus the conjugation of $A_h$ with  $e^{tX_{\btheta}}$ is non-negative), for $t$ large enough one can write
\begin{equation}
    \label{equation:sum-of-squares}
A_{h}-e^{-2(\nu+o(1)) t} e^{t X_{\btheta}} A_{h} e^{-t X_{\btheta}}=C_{h}^{*} C_{h}+D_{h},
\end{equation}
where $\WF_{h}(D_{h}) \cap (\mathcal{T}+E_{u}^{*})=\emptyset$, $\WF_{h}(C_{h})$ is close to $\mathcal{T}+E_{u}^{*}$, and $C_{h}$ is elliptic on an open subset $\Omega \subset \mathcal{T}+E_{u}^{*}$ which contains a non-empty region homeomorphic to an annulus, such that:
\begin{enumerate}[label=(\roman*)]
    \item If $\nu >0$, then $\mathcal{T} \subset \Omega$;
    \item If $\nu=0$, then $\Omega \cap \mathcal{T}=\emptyset$.
\end{enumerate}
See \cite[Proposition 5.6]{Bonthonneau-Guillarmou-Hilgert-Weich-25}, and the lines following \cite[Equation (5.13)]{Bonthonneau-Guillarmou-Hilgert-Weich-25}, for a proof of a similar factorization. Applying Step 2 with $D_{h}$ and using Step 1, one has $\langle D_{h} f_{h}, f_{h}\rangle_{L^{2}}=o(h^{\ell})$. Hence $\|C_{h} f_{h}\|_{L^{2}}^{2}=o(h^{\ell})$ and so 
\begin{equation}\label{eq:part-i-to-reference-later}
    \|C_{h} f_{h}\|_{L^{2}}=o(h^{\ell / 2}).
\end{equation} 
Then, using elliptic estimates (Lemma \ref{lemma:elliptic-rgularity}) and propagation of singularities (Lemma \ref{lemma:propagation-singularities}), one can control $\|E_{h} f_{h}\|_{L^{2}}$ by $\|C_{h} f_{h}\|_{L^{2}}$ modulo $o(h^{\ell})$ remainders. This proves the claim.

Finally, the $C^{0}$-estimate follows from Sobolev embeddings: for all $N>n / 2$ (where $n := \dim M_0$), there exists $C>0$ such that for all $u_h \in C^{\infty}(M_{0})$, $\|u_{h}\|_{C^{0}} \leq C h^{-n / 2}\|u_{h}\|_{H_{h}^{N}}$. Since $E_h$ has compact support, the $L^2$-norm in the previous estimates is irrelevant and one has $\|E_{h} f_{h}\|_{H_{h}^{N}}=o(h^{\ell / 2})$ similarly. Therefore, this yields
\[
\|E_{h} f_{h}\|_{C^{0}} \leq C h^{-n/2}\|E_{h} f_{h}\|_{H_{h}^{N}}=o(h^{\ell/2-n/2}).
\]
This concludes the proof.    
\end{proof}

In the following statement, recall that $n$ denotes the dimension of $M_0$ (or $M$).

\begin{lemma} \label{lemma:microsupport-quasimodes-2}
Let $E_{h} \in \Psi_{h}^{\mathrm{comp }}(M_{0})$ such that $E_h \equiv 1$ microlocally near $\mathcal{T}$. Then, there exists $c>1$ such that:
\begin{equation} \label{equation:bounds-E_h}
    P_{h} E_{h} f_{h}=o_{L^{2}}(h^{\ell / 2+1})=o_{C^0}(h^{\ell / 2+1-n / 2}), \quad 1 / c \leq \|E_{h} f_{h}\|_{L^{2}} \leq c.
\end{equation}
Furthermore, $\nu=0$ and $P_{h}=-h X_{\btheta}-z_h$, where $\Re(z_{h}) \geq 0$ and
\[
z_h=i+o(h^{\ell / 2+1}).
\]
\end{lemma}

\begin{proof}
Since $[P_{h},E_{h}] \in h\Psi_{h}^{\mathrm{comp}}(M_{0})$ and $\WF(h^{-1}[P_{h},E_{h}]) \cap \mathcal{T}=\emptyset$, Lemma \ref{lemma:microsupport-quasimodes} and \eqref{equation:quasimodes} give
\[
P_{h} E_{h} f_{h}=E_{h} P_{h} f_{h}+[P_{h}, E_{h}] f_{h}=o_{L^{2}}(h^{\ell+1})+o_{L^{2}}(h^{\ell / 2+1})=o_{L^2}(h^{\ell / 2+1}),
\]
while the $C^{0}$-estimate follows from Sobolev embedding.

We now show the bounds on $E_{h}f_{h}$. Let $A_{h}$ be microlocally supported and elliptic near $E_{u}^{*}$. Then 
\[
\|A_{h} f_{h}\|_{H_{h}^{-s}} \leq C\left(h^{-1}\|B_{h} P_{h} f_{h}\|_{H_{h}^{-s}}+\|A_{h}' f_{h}\|_{H_{h}^{-s}}+h^{N}\|f_{h}\|_{H_h^{-N}}\right),
\]
were $B_{h},A_{h}'$ are as in the sink estimate (Lemma \ref{lemma:source-sink-estimates}). Since $B_{h} \in \Psi_{h}^{0}(M_{0})$ is microlocalized near $E_{u}^{*}$, where $\mathcal{H}_h^s$ is microlocally equivalent to $H_{h}^{-s}$, we get
\[
h^{-1}\|B_{h} P_{h} f_{h}\|_{H_{h}^{-s}} \leq C h^{-1}\|P_{h} f_{h}\|_{\mathcal{H}_{h}^{s}}=o(h^{\ell}).
\]
The operator $A_{h}'$ has wavefront set disjoint from $E_{u}^{*} \cap \partial_{\infty} T^{*}M$. Hence, we can argue as in the proof of Lemma \ref{lemma:microsupport-quasimodes} and apply propagation of singularities (Lemma \ref{lemma:propagation-singularities}) to get $\|A_{h}'f_{h}\|_{H_{h}^{-s}}=o(h^{\ell/2})$. So far, we have
\begin{equation} \label{equation:proof-quasimodes-microsupport1}
    \|A_{h} f_{h}\|_{H_{h}^{-s}}=o(h^{\ell/2}).
\end{equation}
Now, 
\[
1=\|f_{h}\|_{\mathcal{H}_{h}^{s}} \leq \|E_{h}f_{h}\|_{\mathcal{H}_{h}^{s}}+\|(1-E_{h})f_{h}\|_{\mathcal{H}_{h}^{s}}=\|E_{h}f_{h}\|_{\mathcal{H}_{h}^{s}}+o(h^{\ell/2}),
\]
where we applied the estimates from Lemma \ref{lemma:microsupport-quasimodes}, as well as propagation of singularities, and \eqref{equation:proof-quasimodes-microsupport1} to $1 - E_h$, since its wavefront set is supported outside a small neighborhood of $\mathcal{T}$. Taking $h$ small enough, we get the lower bound for $E_{h}f_{h}$. To obtain the upper bound, observe that since $E_{h}$ has compact support, we have 
\[ \| E_{h} f_{h}\|_{L^{2}} \leq C\|f_{h}\|_{\mathcal{H}_{h}^{s}}=C. \]

To show that $\nu=0$ we argue by contradiction. By the proof Lemma \ref{lemma:microsupport-quasimodes}, \eqref{eq:part-i-to-reference-later}, if $\nu>0$ we obtain $\|E_{h}f_{h}\|_{L^{2}}=o(h^{\ell/2})$, contradicting the bounds on $E_{h}f_{h}$. Hence, $\nu=0$.

Finally, we prove that $P_{h}$ has the form stated in the lemma. Call $w_{h} \geq 0$ the $o(h)$ term in the expression of $P_{h}$. We will show that $w_{h}=o(h^{\ell/2+1})$. Since the propagator of $X_{\btheta}$ is unitary on $L^{2}(M_{0})$, we have
\[ 
\|(-X_{\btheta}-z)^{-1}\|_{L^{2} \to L^{2}} \leq \int_{0}^{\infty}\|e^{-tX_{\btheta}}\|_{L^{2} \to L^{2}}e^{-t\Re(z)}dt=1/\Re(z),
\]
for $\Re(z)>0$. Assume for the sake of obtaining a contradiction that the estimate for $w_{h}$ does not hold. Then, for a sequence $h_{n} \to 0$ we have that $w_{h_{n}} \ge Ch_{n}^{\ell/2+1}$ for some $C>0$. Using this and \eqref{equation:bounds-E_h}, we obtain
\[ 
\|E_{h}f_{h}\|_{L^{2}} \leq \frac{1}{w_{h}}\|P_{h}E_{h}f_{h}\|_{L^{2}}=o(1),
\]
contradicting the bounds on $E_{h}f_{h}$ that we have just proved. This completes the proof.
\end{proof} 

The next step in the proof of Theorem \ref{theorem:high-frequency-ab-ext} is to show a certain horocyclic invariance for the quasimodes \eqref{equation:quasimodes}. To this end, we consider the dynamical connection induced by $X_{\btheta}$, that is, $d_{\btheta}:=d+i\eta_{\btheta}$. 

\begin{lemma}[Horocyclic invariance] \label{lemma:horocyclic-invariance}
    Let $E_{h} \in \Psi_{h}^{\mathrm{comp}}(M_{0})$ such that $E_{h}=1$ microlocally near $\mathcal{T}$. Then,
    \[ (-hd_{\btheta} -z_{h}\alpha \wedge)E_{h}f_{h}=o_{L^{2}}(h^{\ell/2})=o_{C^{0}}(h^{\ell/2-n/2}). \]
\end{lemma}

\begin{proof}
By linearity (and since $\ker \alpha=E_{s} \oplus E_{u}$), it is enough to show that 
\[ -hd_{\btheta}^{u/s}E_{h}f_{h}=o_{L^{2}}(h^{\ell/2})=o_{C^{0}}(h^{\ell/2-n/2}), \]
where $d_{\btheta}^{u/s}(\bullet)$ denotes the projection of $d_{\btheta}(\bullet)$ to $E_{s/u}^*$. Before going into the proof, we will check some identities that will be useful for our purposes. We extend the action of $X_{\btheta}$ to smooth sections of $T^{*}M$ by $\mathcal{L}_{X_{M_0}}+i\iota_{X_{M_0}}\eta_{\btheta}$, and we denote the restriction to sections of $E_{s}^{*}$ (smooth along the flow direction) by $X_{\btheta}^{u}$. In the same fashion, we extend $P_{h}$ to $P_{h}^{u} = -hX_{\btheta}^u - h(\nu + o(1)) - i$.  We will need the relations
\begin{equation} \label{equation:commutation-connection-X}
    X_{\btheta}^{u}d_{\btheta}^{u}=d_{\btheta}^{u}X_{\btheta}, \qquad P_{h}^{u}d_{\btheta}^{u}=d_{\btheta}^{u}P_{h}.
\end{equation}
The right equality follows from the left one which, in turn, is a consequence of the commutation formula $X_{\btheta}d_{\btheta}=d_{\btheta}X_{\btheta}$ which we now prove. We have
\[
\begin{split}
    X_{\btheta}d_{\btheta}f &=\mathcal{L}_{X_{M_{0}}}df+i\mathcal{L}_{X_{M_0}}(\eta_{\btheta}f)+i\iota_{X_{M_{0}}}\eta_{\btheta} \wedge df - \iota_{X_{M_0}}\eta_{\btheta} \wedge \eta_{\btheta}f \\
    &=d\mathcal{L}_{X_{M_0}}f+i\iota_{X_{M_0}}(df \wedge\eta_{\btheta})+id(\iota_{X_{M_0}}\eta_{\btheta}f)  +i\iota_{X_{M_0}}\eta_{\btheta} \wedge df-\iota_{X_{M_0}}\eta_{\btheta} \wedge \eta_{\btheta}f \\
    &=dX_{M_0}f+i\eta_{\btheta} \cdot X_{M_0}f+id(\iota_{X_{M_0}}\eta_{\btheta} f)-\iota_{X_{M_0}}\eta_{\btheta} \wedge \eta_{\btheta}f=d_{\btheta}X_{\btheta}f,
\end{split}
\]
where in the second equality we used Cartan's magical formula.

Now we prove the result of the statement. Using \eqref{equation:commutation-connection-X}, we have
\begin{equation}\label{eq:commutator-ph-eh}
-P_{h}^{u}hd_{\btheta}^{u}E_{h}f_{h}=-hd_{\btheta}^{u}P_{h}E_{h}f_{h}=-hd_{\btheta}^{u}E_{h}P_{h}f_{h}+hd_{\btheta}^{u}[E_{h},P_{h}]f_{h}.
\end{equation}
The bound on the quasimodes \eqref{equation:quasimodes} implies
\[ 
    \|hd_{\btheta}^{u}E_{h}P_{h}f_{h}\|_{L^{2}} \leq C \| E_{h}P_{h}f_{h} \|_{H_{h}^{1}} \leq C \|P_{h}f_{h}\|_{\mathcal{H}_{h}^{s}}=o(h^{\ell+1}), 
\]
where in the second inequality we used that $E_h$ is compactly microlocalised and so a map $E_h: \mc{H}^s_h \to H^1_h$. On the other hand, since $\WF_{h}(h^{-1}[E_{h},P_{h}])$ is compact and has empty intersection with $\mathcal{T}$, in light of Lemma \ref{lemma:microsupport-quasimodes}, we have
\[ 
\|h d_{\btheta}^{u} [E_{h},P_{h}]f_{h}\|_{L^{2}} \leq C \|[E_{h},P_{h}] f_{h}\|_{H_{h}^{1}}=o(h^{\ell/2 +1}). 
\]
Note that in the previous two inequalities we used that the multiplication by $\iota_{X_{M_0}} \eta_{\btheta}$ is uniformly bounded in $\btheta$ in the $C^0$-norm. Since we may cover $\mathrm{U}(1)^d$ by a finite number of open contractible sets where we have uniform estimates, we conclude that the estimates are uniform for $\btheta \in \mathrm{U}(1)^d$. Hence by \eqref{eq:commutator-ph-eh},

\begin{equation} \label{equation:L2-bound}
    P_{h}^{u} hd_{\btheta}^{u}E_{h}f_{h}=o_{L^{2}}(h^{\ell/2 +1}). 
\end{equation}
We now obtain a similar bound but in the $C^{0}$ norm. We will denote by $C_{h}^{1}$ the space $C^{1}$ with the semiclassical norm $\|f\|_{C_{h}^{1}}:=\|f\|_{C^{0}}+h\|df\|_{C^{0}}$. Then, for $N>n/2+1$
\[
\|hd_{\btheta}^{u}[E_{h},P_{h}]f_{h}\|_{C^{0}} \leq C\|[E_{h},P_{h}]f_{h}\|_{C_{h}^{1}} \leq Ch^{-n/2} \|[E_{h},P_{h}]f_{h}\|_{H_{h}^{N}}=o(h^{\ell/2+1-n/2}),
\]
where the second inequality follows from the Sobolev embedding, and the last one from Lemma \ref{lemma:microsupport-quasimodes} (since $\WF_{h}(h^{-1}[E_{h},P_{h}]) \cap \mathcal{T}=\emptyset$). Using \eqref{equation:quasimodes}, we also have
\[ 
\|hd_{\btheta}^{u}E_{h}P_{h}f_{h}\|_{C^{0}} \leq C \|E_{h}P_{h}f_{h}\|_{C_{h}^{1}} \leq C\|P_{h}f_{h}\|_{\mathcal{H}_{h}^{s}}=o(h^{\ell+1}). 
\]
Therefore, 
\begin{equation} \label{equation:C0-bound}
    P_{h}^{u}hd_{\btheta}^{u}E_{h}f_{h}=o_{C^{0}}(h^{\ell/2+1-n/2}).
\end{equation}
Note that the only difference between the bounds from the statement and the ones that we have so far is the operator $P_{h}^{u}$. Hence, we study the resolvent:
\[
(-X_{\btheta}^{u}-z)^{-1}=-\int_{0}^{\infty} e^{-tX_{\btheta}^{u}}e^{-tz} \dd t.
\]
We claim that it converges in both the $C^{0}$- and the $L^{2}$-topology on $\{ \Re(z)>-\delta\}$, where $\delta$ is the rate of contraction of $E_{s}^{*}$. Indeed, for $u \in L^2(M,E_{s}^{*})$, 
\[
\begin{split}
    \|(-X_{\btheta}^{u}-z)^{-1}u\|_{L^{2}} &\leq \int_{0}^{\infty}\|e^{-tX_{\btheta}^{u}}\|_{L^{2} \to L^{2}} e^{-t\Re(z)}\dd t \|u\|_{L^{2}} \\
    & \leq \int_{0}^{\infty} \|e^{-t\mathcal{L}_{X_{M_{0}}}|_{E_{s}^{*}}}\|_{L^{2} \to L^{2}} e^{-t\Re(z)}\dd t \|u\|_{L^{2}} \\
    & \leq \int_{0}^{\infty} e^{-\delta t}e^{-t\Re(z)} \dd t \|u\|_{L^{2}} \\
    & \leq (\delta +\Re(z))^{-1} \|u\|_{L^{2}},
\end{split}
\]
and the same bounds work in the $C^{0}$ case. In particular, the resolvent is defined and uniformly bounded for $\{z \in \C : \Re(z) \geq 0\}$ on both $C^{0}$ and $L^{2}$. Hence, $X_{\btheta}^{u}$ is invertible on both $C^{0}$ and $L^{2}$ with norm bounded by $\leq \delta^{-1}$. Thus, the proof is completed after an application of $(P_{h}^{u})^{-1}$, which has norm $\leq Ch^{-1}$, to \eqref{equation:C0-bound} and \eqref{equation:L2-bound}.
\end{proof}

For $E_{h} \in \Psi_{h}^{\mathrm{comp}}(M_{0})$ such that $E_{h}\equiv 1$ microlocally near $\mathcal{T}$, we introduce $v_{h}=E_{h}f_{h}$. To obtain a contradiction, we study some further properties of $v_{h}$.

\begin{lemma}
\label{lemma:dtheta-petit}
    We have that
    \[ d |v_{h}|^{2}=o_{C^{0}}(h^{\ell/2-n-1}). \]
\end{lemma}

\begin{proof}
The estimate is equivalent to showing that $Y|v_{h}|^2=o_{C^{0}}(h^{\ell/2-n-1})$ for all smooth vector fields $Y \in C^{\infty}(M, T M)$ such that $\|Y\|_{C^0} \leq 1$. To show this, note that we pointwise have
\[
\begin{split}
    Y|v_{h}|^{2} &=\langle d_{\btheta}(Y)v_{h},v_{h} \rangle + \langle v_{h},d_{\btheta}(Y)v_{h} \rangle \\
    &= \left\langle \left(d_{\btheta}(Y)+\frac{z_{h}}{h}\alpha(Y)\right)v_{h},v_{h} \right\rangle + \left\langle v_{h},\left(d_{\btheta}(Y)+\frac{z_{h}}{h}\alpha(Y)\right)v_{h} \right\rangle -2\Re \left( \frac{z_{h}}{h} \right) \alpha(Y).
\end{split}
\]
By Lemma \ref{lemma:microsupport-quasimodes-2}, $\Re(z_{h}/h)=o(h^{\ell/2})$, while Lemma \ref{lemma:horocyclic-invariance} gives 
\[
    (d_{\btheta}(Y)+z_{h}\alpha(Y)/h)v_{h}=o_{C^{0}}(h^{\ell/2-n/2-1}),
\]
and by condition $\|f_h\|_{\mc{H}^s_h} = 1$, the Sobolev embedding, and the fact that $E_h$ is compactly microsupported, we get $v_{h} = \mc{O}_{C^{0}}(h^{-n/2})$. The conclusion now follows directly using these bounds and the previous computation.
\end{proof}

We now derive a uniform pointwise lower bound for $|v_h|$.

\begin{lemma}
\label{lemma:vh-bound-below}
    There exists a constant $C>0$ such that for all $x \in M_{0}$, and $h>0$, $|v_{h}(x)|>C$.
\end{lemma}

\begin{proof}
By compactness, using Lemma \ref{lemma:dtheta-petit}, we have, uniformly in $x,y \in M_{0}$ that
\[ |v_{h}(x)|^{2}-|v_{h}(y)|^{2}=o(h^{\ell/2-n-1}). \]
Assume for the sake of a contradiction that the statement of the lemma does not hold. Then, there exists a sequence $x_{h} \in M_{0}$ such that $|v_{h}(x_{h})|^{2} \to 0$ as $h \to 0$. Plugging this into the previous relation, gives $\|v_{h}\|_{C^{0}}^{2}=o(1)$. However, this contradicts the lower bound $\|v_h\|_{L^2}=\|E_{h}f_{h}\|_{L^{2}} \geq 1/c$ obtained in Lemma \ref{lemma:microsupport-quasimodes-2}.
\end{proof}

In our case, we are interested in the (dynamical) connection that is used to obtain the horocyclic invariance, that is, $\nabla:=\nabla^{\mathrm{dyn}}+i\alpha=d_{\btheta}+i\alpha=d+i\eta_{\btheta}+i\alpha$. See \S\ref{sssection:dynamical-connection-linear} for further discussion. In contrast with \cite[Chapter 4]{Cekic-Lefeuvre-24}, here the dynamical connection is smooth. Therefore, the following lemma is standard. 

\begin{lemma}
\label{lemma:curvature-classic}
    Let $\nabla$ be a smooth connection on a complex line bundle $L \to M_{0}$. Then $F_{\nabla} \equiv 0$ if and only if for all homotopically trivial closed loops $\gamma \subset M_{0}$, $\mathrm{Hol}_{\nabla}(\gamma)=1$.
\end{lemma}

Since $\eta_{\btheta}$ is closed, the curvature of $\nabla$ is $F_{\nabla} = d(i\eta_{\btheta}+i\alpha)=id\alpha$. The final step in the proof of Theorem \ref{theorem:high-frequency-ab-ext} is to show that the curvature of $\nabla$ vanishes. This, of course, contradicts the hypothesis that $d\alpha \neq 0$, finishing the proof of the high frequency estimates:

\begin{proof}[Proof of Theorem \ref{theorem:high-frequency-ab-ext}]
    We claim that the combination of Lemmas \ref{lemma:horocyclic-invariance}, \ref{lemma:dtheta-petit}, \ref{lemma:vh-bound-below} and \ref{lemma:curvature-classic} implies that $d\alpha=0$; this is a contradiction and therefore proves \eqref{equation:toprove}, hence Theorem \ref{theorem:high-frequency-ab-ext}. The proof is the same as \cite[Lemma 4.5.7]{Cekic-Lefeuvre-24} and not reproduced here; it shows that the holonomy of $\nabla$ along contractible loops is trivial. Note that in the notation of \cite[Lemma 4.5.7]{Cekic-Lefeuvre-24}, we have $\eta = 1$. Also, in the course of this argument, one needs $\ell>n+10 = \dim M_0+10$.
\end{proof}

As a corollary of Theorem \ref{theorem:high-frequency-ab-ext}, we obtain:

\begin{corollary} \label{corollary:smoothness-resolvent}
Let $\vartheta$ be as in Theorem \ref{theorem:high-frequency-ab-ext}, and let $p\geq 0$. Then:
    \begin{enumerate}[label=(\roman*), itemsep=5pt]
        \item The resolvent 
        \[ 
        \R \ni \lambda \mapsto (\mp X_{\btheta}-i\lambda)^{-1} \in \mathcal{L}(\mathcal{H}_{\pm}^{s}(M_{0}),\mathcal{H}_{\pm}^{s}(M_{0})),
        \]
        is well-defined on the imaginary axis and $C^{p}$-regular with respect to $\lambda$, with the exception that for $\btheta= \mathbf{0}$, at $\lambda=0$, one has to replace $\mathcal{H}_{\pm}^{s}(M_{0})$ by its intersection with the space of distributions with zero average against the measure $\dd\vol_{M_{0}}$.
        \item For any $r > 0$, there exists $C = C(r) > 0$ such that for all $\btheta \in \mathrm{U}(1)^d$, and $\lambda \in \R$ with $|\lambda| > r$:
        \[ 
        \| \partial_{\lambda}^{p}(\mp X_{\btheta}-i\lambda)^{-1} \|_{\mathcal{H}_{\pm}^{s}(M_{0}) \to \mathcal{H}_{\pm}^{s}(M_{0})} \leq C\langle \lambda \rangle ^{(p + 1) \vartheta}.
        \]
    \end{enumerate}
\end{corollary}

In $(ii)$, as in $(i)$ we would get a uniform bound over all $\lambda \in \mathbb{R}$ by replacing the space $\mc{H}^s_\pm(M_0)$ by its intersection with distributions with zero average with respect to the measure $\vol_{M_{0}}$. We emphasise that since the operators $X_{\btheta}$ are only well-defined as a smooth family of operators over an arbitrary good open subset $U$ of $\mathrm{U}(1)^d$, see around Proposition \ref{prop:contractible} above for a discussion, the statements should strictly speaking also be understood only for $\btheta \in U$. 

\begin{proof}
In what follows we consider the case $\btheta \neq \mathbf{0}$. The case $\btheta = \mathbf{0}$ is treated in the same way by noting that $(\mp X_{\btheta} - i\lambda)^{-1}$ agrees with its holomorphic part when acting on distributions with zero average; more precisely, we use that the family $(R^{\pm, \mathrm{hol}}_{\btheta}(z))_{\btheta, z}$ is continuous with values in $\mc{L}(\mc{H}^s_\pm(M_0), \mc{H}^s_\pm(M_0))$ for $|\lambda|$ small and $\btheta$ close to $\mathbf{0}$ and so uniformly bounded on compact sets in $(\lambda, \btheta)$.

By Theorem \ref{theorem:high-frequency-ab-ext}:
\begin{equation} \label{equation:high-freq-1}
    \|(\mp X_{\btheta}-i \lambda)^{-1}\|_{\mathcal{H}_{\pm}^{s}(M_{0}) \to \mathcal{H}_{\pm}^{s}(M_{0})} \leq C\langle\lambda\rangle^{\vartheta}.
\end{equation}
To show smoothness with respect to $\lambda$, is enough to use the resolvent formula
\[  
(\mp X_{\btheta}-i \lambda)^{-1}-(\mp X_{\btheta}-i \lambda')^{-1}=i(\lambda-\lambda')(\mp X_{\btheta}-i \lambda)^{-1}(\mp X_{\btheta}-i \lambda')^{-1},
\]
which shows that
\[
\lambda \mapsto ( \mp X_{\btheta}-i \lambda )^{-1} \in C^{1}(\R, \mathcal{L}(\mathcal{H}_{\pm}^{s}(M_{0}), \mathcal{H}_{\pm}^{s}(M_{0}))),
\]
with 
\[
\partial_{\lambda} (\mp X_{\btheta}-i \lambda )^{-1}=i (\mp X_{\btheta}-i \lambda)^{-2}.
\]
A similar argument for higher order derivatives gives that
\[
\lambda \mapsto (\mp X_{\btheta}-i \lambda)^{-1} \in C^{p}\left(\R, \mathcal{L}(\mathcal{H}_{\pm}^{s}, \mathcal{H}_{\pm}^{s})\right),
\]
with derivatives
\begin{equation} \label{equation:high-freq-2}
    \partial_{\lambda}^{p}(\mp X_{\btheta}-i \lambda )^{-1}=i^p p!(\mp X_{\btheta}-i \lambda)^{-(p+1)}.
\end{equation}
Item (ii) follows from iteration of \eqref{equation:high-freq-1} and \eqref{equation:high-freq-2}.    
\end{proof}

\subsection{Proof of Theorem \ref{theorem:main1}} Recall that $\langle{t}\rangle = \sqrt{1 + t^2}$ is the notation for the Japanese bracket. We begin by proving the following asymptotic for the correlation function of $X_{\btheta}$:

\begin{theorem} \label{theorem:mixing_X_theta}
Assume that $d\alpha \neq 0$. Let $U$ be a neighborhood of $\mathbf{0} \in \mathrm{U}(1)^d$ such that the resonance parametrization $(\lambda_{\btheta})_{\btheta \in U}$ near $z = 0$ exists and depends smoothly on $\btheta$, see \S \ref{ssec:leading-resonance} above. Then, there exists $\vartheta > 0$, such that for all integers $N \geq 0$, there exists $C_N > 0$ such that for all $f,g \in C^{\infty}(M_{0})$ and for all $t \geq 0$ we have:
\begin{enumerate}[label=(\roman*), itemsep=5pt]
    \item If $\btheta \in U$, we have
\[
\begin{split}
\left|\int_{M_{0}} e^{-tX_{\btheta}}f \cdot \overline{g}~ \dd \vol_{M_0} - e^{\lambda_{\btheta}t}\langle \Pi_{\btheta}^{+}f, g \rangle_{L^{2}(M_{0})}\right|\leq C_N \langle{t\rangle}^{-N} \|f\|_{H^{\vartheta(N+2)+4}({M_{0}})}\|g\|_{H^{\vartheta(N+2)+4}({M_{0}})}.
\end{split}
\]
\item If $\btheta \in \mathrm{U}(1)^d \setminus U$, we have
\[
\begin{split}
\left|\int_{M_{0}} e^{-tX_{\btheta}}f \cdot \overline{g}~ \dd \vol_{M_{0}}\right|\leq C_N \langle{t\rangle}^{-N} \|f\|_{H^{\vartheta(N+2)+4}({M_{0}})}\|g\|_{H^{\vartheta(N+2)+4}({M_{0}})}.
\end{split}
\]
\end{enumerate}
\end{theorem}

Note that to get an asymptotic for all negative times, we could use that $\langle{e^{-t X_{\btheta}} f, g}\rangle_{L^2} = \langle{f, e^{t X_{\btheta}}g}\rangle_{L^2}$ and apply Item (i) to get that the principal term in the expansion is equal to $e^{-t \overline{\lambda_{\btheta}}} \langle{\Pi^-_{\btheta}f, g}\rangle_{L^2}$. Observe that by Lemma \ref{lemma:leading-resonance-symmetry} we have $\overline{\lambda_{\btheta}} = \lambda_{-\btheta}$.

\begin{proof}
We first prove $(i)$ and argue for non-zero $\btheta \in U$. The claim for $\btheta = \mathbf{0}$ then follows by continuity. On $\mathcal{D}_{L^2}:=\{f \in L^2(M_{0}) : X_{\btheta} f \in L^2(M_{0})\}$ the operator $iX_{\btheta}$ is self-adjoint. It thus admits a spectral measure $\dd P_{\btheta}(\lambda)$ such that $\id_{L^2(M_{0})}= \int_{\R} \dd P_{\btheta}(\lambda)$. It follows from Stone's formula and Lemma \ref{lemma:no-resonances-imaginary-axis} that $\dd P_{\btheta}(\lambda)$ is analytic, without atoms, and 
\[
\dd P_{\btheta}(\lambda)=-\frac{1}{2 \pi}(R^{+}_{\btheta}(i \lambda)+R^{-}_{\btheta}(-i \lambda)) \dd \lambda,
\]
where $R_{\btheta}^{\pm}(z)$ are is in \eqref{eq:resolvent_def}, see \cite[Proposition 9.2.2]{Lefeuvre-book}. Note that $\dd P$ is smooth by Corollary \ref{corollary:smoothness-resolvent}.

Now, for $f, g \in C^{\infty}(M_{0})$ we have:
\[ 
\begin{split}
\int_{M_{0}} e^{-tX_{\btheta}} f \overline{g}~ \dd \vol_{M_{0}} & = \langle e^{-t X_{\btheta}} f, g \rangle_{L^2(M_{0})} \\
& = \int_{\R} e^{i t \lambda}\langle \dd P_{\btheta}(\lambda) f, g\rangle_{L^2(M_{0})} \\
& = -\frac{1}{2\pi i} \int_{i\R}e^{zt} \langle (R_{\btheta}^{+}(z)+R_{\btheta}^{-}(-z) )f,g\rangle \dd z.    
\end{split}
\]
To compute the last integral, we will use the contour deformation method. Let $r>0$. Let $\gamma$ be a curve that differs from $i\R$ only for $|z|<r$, where it is equal to a smooth graph over the imaginary axis, so that $\lambda_{\btheta}$ is contained in the region of $\mathbb{C}$ to the right of $\gamma$ for any $\btheta \in U$. Call this portion of the curve $\gamma_{r}$. Note that since $\lambda_{\btheta} \in \mathbb{R}_{\leq 0}$, we can take $r > 0$ arbitrarily small. In this way, the only resonance of $-X_{\btheta}$ that $\gamma \cup i\R$ encloses in the interior or its boundary is $\lambda_{\btheta}$, for every $\btheta \in U$. Since we are considering $\btheta$ close to $\mathbf{0}$, we can take $\gamma$ not depending on $\btheta$. Recall that around $\lambda_{\btheta}$ we have the expansion
\[ 
R_{\pm}^{\btheta}(z)=R_{\pm}^{\btheta,\mathrm{hol}}(z)-\frac{\Pi_{\btheta}^{\pm}}{z-\lambda_{\btheta}}.
\]
Then, by the residue theorem, 
\[
\begin{split}
    \int_{\R} e^{i t \lambda}\langle \dd P_{\btheta}(\lambda) f, g\rangle_{L^2(M_{0})}  = & -\frac{1}{2\pi i}\int_{\gamma} e^{zt}\langle  (R_{\btheta}^{+}(z) + R_{\btheta}^-(-z)) f, g\rangle \dd z \\
    & -\operatorname*{Res}_{z=\lambda_{\btheta}} ( e^{zt}\langle (R_{\btheta}^{+}(z)+R_{\btheta}^{-}(-z))f,g \rangle ).
\end{split}
\]
Note that since $R_{\btheta}^{-}(-z)$ is holomorphic over $\gamma$ as $-\gamma$ is a subset of the right half-plane, we have 
\[
\operatorname*{Res}_{z=\lambda_{\btheta}} ( e^{zt}\langle (R_{\btheta}^{+}(z)+R_{\btheta}^{-}(-z))f,g \rangle )= - e^{\lambda_{\btheta}t}\langle \Pi_{\btheta}^{+}f,g \rangle_{L^{2}(M_{0})}.
\]
Now we show that the integral over $\gamma$ is well defined and is $\mathcal{O}(t^{-\infty})$. We record the following identity 
\[ X_{\btheta}(R_{\btheta}^{+}(z)+R_{\btheta}^{-}(-z)) = (R_{\btheta}^{+}(z)+R_{\btheta}^{-}(-z))X_{\btheta} =-z(R_{\btheta}^{+}(z)+R_{\btheta}^{-}(-z)). \]
We also write $\langle \partial_z \rangle^{2}=1+\partial_{z}^2$ and $\langle X_{\btheta}\rangle^2=1-X_{\btheta}^2$. We now treat the integral over $\gamma$ in two steps: large $|z|$ and bounded $|z|$.

Take $N_1 \geq 0$ arbitrary and even, and $N_2$ to be the smallest even integer such that $N_2 >\vartheta\left(N_1+1\right)+1$. Note that over the support of $\chi$, we have $\langle{z}\rangle^2 = 1 - z^2$. We get:
\begin{align*}
    &\int_{\gamma} e^{zt}\langle (R_{\btheta}^{+}(z)+R_{\btheta}^{-}(-z)) f, g\rangle_{L^2(M_{0})}\dd z\\ 
    &=\langle t\rangle^{-N_1} \int_{\gamma} \langle \partial_z \rangle^{N_1} (e^{zt})(1 - z^2)^{-N_2/2}\langle (R_{\btheta}^{+}(z)+R_{\btheta}^{-}(-z)) \langle X_{\btheta}\rangle^{N_2} f, g\rangle \dd z \\
    &=\langle t\rangle^{-N_1} \int_{\gamma} e^{zt}\langle \partial_{z}\rangle^{N_1}\left((1 - z^2)^{-N_2/2} (R_{\btheta}^{+}(z)+R_{\btheta}^{-}(-z))\langle X_{\btheta}\rangle^{N_2} f, g\rangle\right)\dd z,
\end{align*}
where the last equality follows by integration by parts along the curve $\gamma$. Note that in the construction we may assume that the point $z = -1$ is to the left of $\gamma$ so $(1 - z^2)^{-1}$ is well-defined. Indeed, the absolute convergence of the last term can be seen as follows:
\begin{align*}
    &\langle \partial_{z}\rangle^{N_1}\left((1 - z^2)^{-N_2/2}\langle (R_{\btheta}^{+}(z)+R_{\btheta}^{-}(-z))\langle X_{\btheta}\rangle^{N_2} f, g\rangle \right)\\ 
    &= (-1)^{N_1 / 2} (1 - z^2)^{-N_2/2}\langle\partial_{z}^{N_1} (R_{\btheta}^{+}(z)+R_{\btheta}^{-}(-z))\langle X_{\btheta}\rangle^{N_2} f, g \rangle + \text {l.o.t }
\end{align*}
where l.o.t are lower-order terms in the $z$ variable, i.e. they have the same or better decay as $|z| \to \infty$ and can be treated similarly. Observe that similarly to \eqref{equation:high-freq-2}, we have
\begin{align*}
\partial_{z}^{N_1} (R_{\btheta}^{+}(z)+R_{\btheta}^{-}(-z)) &= N_1! ((-X_{\btheta}-z)^{-(N_{1}+1)} + (-1)^{N_{1}}(X_{\btheta}+z)^{-(N_{1}+1)})\\ 
&= N_1! ((R_{\btheta}^{+})^{(N_1 + 1)}(z) + (R_{\btheta}^{-})^{(N_1 + 1)}(-z)),
\end{align*}
where in the second line we used that $N_1$ is even and where ${}^{(N_1 + 1)}$ denotes $(N_1 + 1)$-fold composition. By assumption we may write $-N_2+\vartheta (N_1+1 )=-(1+\eps)$ for some $\eps \in(0,2]$. For $z \in \gamma \setminus \gamma_{r}$ we can write $z=i\lambda$ with $\lambda \in \mathbb{R}$. Note that we have $1 - z^2 = \langle{z}\rangle^2$, and so we get that (recall $s > 0$ is the exponent of the anisotropic space)
\[
\begin{split}
& |\langle z\rangle^{-N_2}\langle\partial_z^{N_1} (R_{\btheta}^{+}(z)+R_{\btheta}^{-}(-z))\langle X_{\btheta}\rangle^{N_2} f, g \rangle_{L^{2}(M_{0})}| \\
& \leq C\langle\lambda\rangle^{-(1+\eps)}\left(\|\langle X_{\btheta}\rangle^{N_2} f\|_{\mathcal{H}_{+}^{s}}\|g\|_{\mathcal{H}_{-}^{s}}+\|\langle X_{\btheta}\rangle^{N_2} f\|_{\mathcal{H}_{-}^{s}}\|g\|_{\mathcal{H}_{+}^{s}}\right) \\
& \leq C\langle\lambda\rangle^{-(1+\eps)}\|\langle X_{\btheta}\rangle^{N_2} f\|_{H^{s}}\|g\|_{H^{s}} \\
& \leq C\langle\lambda\rangle^{-(1+\eps)}\|f\|_{H^{s+N_{2}}}\|g\|_{H^{s}} \\
& \leq C\langle\lambda\rangle^{-(1+\eps)}\|f\|_{H^{s+N_{2}}}\|g\|_{H^{s+N_{2}}}.
\end{split}
\]
Here in first inequality we used Corollary \ref{corollary:smoothness-resolvent}, Item (ii) (which also gives that the estimate is uniform in $\btheta \in U$) as well as that the $L^2$ product is bounded on $\mc{H}^s_+(M_0) \times \mc{H}^s_-(M_0)$, while in the second one we used that $H^{s} \subset \mathcal{H}_{\pm}^{s}$.

We now treat the integral over $\gamma_r$. For $z \in \gamma_{r}$ we similarly get

\[
\begin{split}
& |(1 - z^2)^{-N_2/2}\langle\partial_z^{N_1} (R_{\btheta}^{+}(z)+R_{\btheta}^{-}(-z))\langle X_{\btheta}\rangle^{N_2} f, g \rangle_{L^{2}(M_{0})}| \\
& \leq C\left(\|\langle X_{\btheta}\rangle^{N_2} f\|_{\mathcal{H}_{+}^{s}}\|g\|_{\mathcal{H}_{-}^{s}}+\|\langle X_{\btheta}\rangle^{N_2} f\|_{\mathcal{H}_{-}^{s}}\|g\|_{\mathcal{H}_{+}^{s}}\right) \\
& \leq C\|\langle X_{\btheta}\rangle^{N_2} f\|_{H^{s}}\|g\|_{H^{s}} \\
& \leq C\|f\|_{H^{s+N_{2}}}\|g\|_{H^{s}} \\
& \leq C\|f\|_{H^{s+N_{2}}}\|g\|_{H^{s+N_{2}}},
\end{split}
\]
were the first inequality follows from Lemma \ref{theorem:meromorphic-extension} and the fact that $|(1 -z^2)^{-1}|$ is bounded for $z \in \gamma_r$. The estimates are uniform over $\btheta$ by continuity of $(R^\pm_{\btheta}(z))_{\btheta, z \in \gamma_r}$ in $(\btheta, z)$ and the fact that we estimate over a compact set in $(\btheta, z)$. This shows that the integral over $\gamma$ and also over $\R$ are well defined and that
\begin{align*}
    &\left|\int_\gamma e^{zt} \langle{(R^+_{\btheta}(z) + R^-_{\btheta}(-z))f, g}\rangle \dd z\right|\\ 
    &\leq C\langle t\rangle^{-N_{1}}\|f\|_{H^{s+N_{2}}}\|g\|_{H^{s+N_{2}}}\left| \int_{\gamma_r} e^{t \Re z} \dd z+ \int_{\R \setminus [-r,r]} \langle\lambda\rangle^{-(1+\eps)} \dd \lambda \right|\\
    &\leq C\langle t\rangle^{-N_{1}}\|f\|_{H^{s+N_{2}}}\|g\|_{H^{s+N_{2}}},
\end{align*}
where the first integral is bounded since it is over a compact set and $e^{t \Re (z)} \leq 1$ (recall that we assumed $t \geq 0$ and we have $\gamma_r \subset \{\Re(z) \leq 0\}$), while in the second one we use that the exponent of the integrand is less than $-1$. In particular, for $s\leq 1$ we find that 
\[
    |\langle e^{tX_{\btheta}}f,g \rangle_{L^{2}(M_{0})}| \le C \langle t \rangle^{-N_{1}}\|f\|_{H^{\vartheta(N_{1}+1)+4}} \|g\|_{H^{\vartheta(N_{1}+1)+4}}, 
\]
and the desired bound follows.

For $(ii)$, a similar argument applies. Indeed, in this case we may take $\gamma$ to be defined similarly, but such that now there are no resonance of $X_{\btheta}$ for $\btheta \in \mathrm{U}(1)^d \setminus U$ on $\gamma$ or to the right of $\gamma$. Indeed this can be achieved by Lemma \ref{lemma:no-resonances-imaginary-axis} and compactness, by potentially taking smaller $U$. Then the same argument as above shows the required bound.
\end{proof}

Before proving the main theorem, we establish a bound on the spectral projector to the resonant state $\lambda_{\btheta}$.

\begin{lemma} \label{lemma:regularity_Pi_theta}
Let $U \subset \mathrm{U}(1)^d$ be a good open neighbourhood of the origin $\mathbf{0} \in \mathrm{U}(1)^d$ such that $(\lambda_{\btheta})_{\btheta \in U}$ and $\Pi_{\btheta}^+$ are well-defined, see \S \ref{ssec:leading-resonance}. Let $k \geq 0$ an integer, and $s > 0$ a positive real number such that the construction in \S \ref{ssec:leading-resonance} works. Then, for all multi-indices $\alpha$ such that $|\alpha| \leq k$, there exists $C=C(k,U) > 0$ such that for all $f,g \in B^{s,k}(M)$
\[ \sup_{\btheta \in U}|\partial_{\btheta}^{\alpha} \langle \Pi_{\btheta}^{+}f_{\btheta},g_{\btheta}\rangle_{L^{2}(M_{0})} | \leq C\|f\|_{B^{s, k}(M)}\|g\|_{B^{s, k}(M)}. \]
\end{lemma}

\begin{proof}
In the first place, thanks to \eqref{eq:projector-near-the-origin} observe that for any $j = 1, \dotsc, d$, we have
\[ 
    \partial_{\theta_j}\Pi_{\btheta}^{+}=-\frac{1}{2\pi i}\oint_\gamma R_{\btheta}^{+}(z)(i \iota_{X_{M_0}} \partial_{\theta_j}\eta_{\btheta})R_{\btheta}^{+}(z) \dd z, 
\]
and inductively we see that derivatives of the spectral projector involve derivatives of $\iota_{X_{M_0}}\eta_{\btheta}$ and powers of the resolvent $R^+_{\btheta}(z)$ for $z \in \gamma$, which are uniformly bounded between anisotropic Sobolev spaces for $\btheta \in U$ since there are no resonances crossing $\gamma$ by assumption. Thus, $\|\partial^\alpha_{\btheta} \Pi^+_{\btheta}\|_{\mc{H}^s_+ \to \mc{H}^s_+} \leq C$ for $|\alpha| \leq k$ and $\btheta \in U$. Then, 
\[
\begin{split}
    |\partial_{\btheta}^{\alpha} \langle \Pi_{\btheta}^{+}f_{\btheta},g_{\btheta}\rangle_{L^{2}(M_{0})} | &\leq \sum_{\alpha_{1}+\alpha_{2}+\alpha_{3}=\alpha}C_{\alpha_{1},\alpha_{2},\alpha_{3}} \langle \partial_{\btheta}^{\alpha_{1}}\Pi_{\btheta}^{+} \partial_{\btheta}^{\alpha_{2}}f_{\btheta},\partial_{\btheta}^{\alpha_{3}}g_{\btheta}\rangle_{L^{2}(M_{0})}  \\
    & \leq C\sum_{\alpha_{1}+\alpha_{2}+\alpha_{3}=\alpha}C_{\alpha_{1},\alpha_{2},\alpha_{3}} \| \partial_{\btheta}^{\alpha_{2}}f_{\btheta} \|_{\mathcal{H}_{+}^{s}(M_{0})} \|\partial_{\btheta}^{\alpha_{3}}g_{\btheta}\|_{\mathcal{H}_{-}^{s}(M_{0})} \\
    & \leq C\sum_{\alpha_{1}+\alpha_{2}+\alpha_{3}=\alpha}C_{\alpha_{1},\alpha_{2},\alpha_{3}} \| \partial_{\btheta}^{\alpha_{2}}f_{\btheta} \|_{H^{s}(M_{0})} \|\partial_{\btheta}^{\alpha_{3}}g_{\btheta}\|_{H^{s}(M_{0})} \\
    & \leq C\|f_{\btheta}\|_{C^{k}(U,H^{s}(M_{0}))} \|g_{\btheta}\|_{C^{k}(U,H^{s}(M_{0}))},
\end{split}
\]
where $C_{\alpha_1, \alpha_2, \alpha_3}$ denotes a multinomial coefficient, in the second line we used the boundedness of the $L^2$ pairing on $\mc{H}^s_+ \times \mc{H}^s_-$, and where in the third inequality we used that $H^{s} \subset \mathcal{H}_{\pm}^{s}$. The result now follows from Lemma \ref{lemma:spaces_theta_sobolev}.    
\end{proof}

\begin{proof}[Proof of Theorem \ref{theorem:main1}]
Let $U_0$ be a good neighbourhood of $\mathbf{0} \in \mathrm{U}(1)^d$ such that the conclusion of Theorem \ref{theorem:mixing_X_theta}, Item (i) holds. Consider a finite covering of $\mathrm{U}(1)^d$ by good geodesically convex balls $(U_i)_{i = 0}^k$ and a partition of unity subordinate to this cover $(\chi_i)_{i = 0}^k$ with $\chi_0 = 1$ near the origin; note that $k$ can be taken to be a constant depending only on $d$. Over each $U_i$ we may trivialise the line bundle $(L_{\btheta})_{\btheta \in U_i}$ smoothly by Proposition \ref{prop:contractible}, i.e. there is a smooth family of sections of $L_{\btheta}$ of pointwise unit norm $(s_{\btheta, i})_{\btheta \in U_i}$. We may use the isomorphism \eqref{equation:iso} for $\btheta \in U_i$ identifying sections of $L_{\btheta}$ with functions on $M_0$; the identification goes as $F_{\btheta} = s_{\btheta, i} \pi^*f_{\btheta, i}$ for a section $F_{\btheta}$ of $L_{\btheta}$. With this identification, $(e^{tX_M} f)_{\btheta, i} = e^{tX_{\btheta}} f_{\btheta, i}$ for each $i$, thanks to \eqref{eq:conjugacy-propagators} and Proposition \ref{eq:fourier-theory}. We therefore conclude from Parseval's decomposition \eqref{equation:parseval} that
\begin{align}\label{eq:parseval-in-action}
\begin{split}
    &\langle e^{-tX_M}f,g \rangle_{L^{2}(M)} = (2\pi)^{-d} \int_{\mathrm{U}(1)^d} \langle{e^{-tX_{\btheta}}F_{\btheta}, G_{\btheta}}\rangle_{L^2(M_0, L_{\btheta})} \dd \btheta\\
    =& \underbrace{\sum_{i = 1}^k \int_{U_i} \chi_i(\btheta) \langle e^{-tX_{\btheta}}f_{\btheta, i}, g_{\btheta, i} \rangle_{L^{2}} \dd \btheta + \int_{U_0}\chi_0(\btheta) (\langle e^{-tX_{\btheta}}f_{\btheta, 0}, g_{\btheta, 0} \rangle_{L^{2}} - e^{t \lambda_{\btheta}} \langle{\Pi_{\btheta}^+f_{\btheta, 0}, g_{\btheta, 0}}\rangle_{L^2}) \dd \btheta}_{S(t):=}\\ 
    &+  \underbrace{\int_{U_0} \chi_0(\btheta) e^{t \lambda_{\btheta}} \langle{\Pi_{\btheta}^+f_{\btheta, 0}, g_{\btheta, 0}}\rangle_{L^2} \dd \btheta.}_{Q(t):=}
\end{split}
\end{align}

We now estimate $S(t)$. By Theorem \ref{theorem:mixing_X_theta}, Items (i) and (ii), for all $N > 0$ there is $C = C(N) > 0$ such that (note that technically we take the maximum of the constants that appear for each $U_i$, where $i = 0, \dotsc, k$)
\begin{align*}
    |S(t)| &\leq C \langle{t}\rangle^{-N} \sum_{i = 0}^k \int_{\mathrm{U}(1)^d} \chi_i(\btheta) \|f_{\btheta, i}\|_{H^{\vartheta(N + 2) + 4}} \|g_{\btheta, i}\|_{H^{\vartheta(N + 2) + 4}} \dd \btheta\\ 
    &\leq C \langle{t}\rangle^{-N} \|f\|_{B^{\vartheta(N + 2) + 4, 0}} \|g\|_{B^{\vartheta(N + 2) + 4, 0}},
\end{align*}
where in the last estimate we used Lemma \ref{lemma:spaces_theta_sobolev}.

We now derive an asymptotics for the term $Q(t)$ defined in \eqref{eq:parseval-in-action}; for simplicity we drop the sub-index $0$ in what follows. Since $\lambda_{\btheta}$ has a non-degenerate critical point at $\btheta = \mathbf{0}$ according to Lemma \ref{lemma:resonance-non-degenerate}, and since $\Re(\lambda_{\btheta}) \leq 0$, we may apply the stationary phase lemma (see \cite[Lemma 7.7.5]{Hoermander-I-03}), and so we have, for each $N>0$:
\begin{equation}
\label{equation:tard}
\begin{split}
    \left( \frac{t}{2\pi} \right)^{\frac{d}{2}} Q(t) &= \frac{1}{\sqrt{\det (-\lambda''_{\mathbf{0}})}} \sum_{j=0}^{N-1} t^{-j} L_j (\langle \Pi_{\btheta}^{+}f_{\btheta},g_{\btheta} \rangle_{L^{2}(M_{0})})|_{\btheta = \mathbf{0}} + R(t),
\end{split}
\end{equation}
where $L_j$ is an elliptic differential operator of order $2j$ that can be made explicit, see \cite[Lemma 7.7.5]{Hoermander-I-03}; we have $L_0 = \id$. As $L_{j}$ is a differential operator of order $2j$, we have
\begin{equation}
\label{equation:expression-cj}
    C_j(f, g) := \frac{(2\pi)^{\frac{d}{2}}}{\sqrt{\det \Var_{\vol_{M_{0}}} (\iota_{X_{M_0}} D \eta_{\mathbf{0}})}} L_j \langle{\Pi^+_{\btheta} f_{\btheta}, g_{\btheta}}\rangle|_{\btheta = \mathbf{0}} = \mc{O}(\|f\|_{B^{s, 2j}} \|g\|_{B^{s, 2j}}),
\end{equation}
where in the estimate we used Lemma \ref{lemma:regularity_Pi_theta}. The remainder term is estimated as
\[ 
    |R(t)| \leq C t^{-N} \sum_{|\alpha| \leq 2N+d+1}  \sup_{\btheta \in U_0} |\partial_{\btheta}^{\alpha} \langle \Pi_{\btheta}^{+}f_{\btheta},g_{\btheta}\rangle_{L^{2}(M_{0})} | \leq Ct^{-N}\|f\|_{B^{s,2N+d+1}(M)}\|g\|_{B^{s,2N+d+1}(M)},
\]
where we again used Lemma \ref{lemma:regularity_Pi_theta}. This completes the proof.
\end{proof}

\subsection{Computation of $C_1$.} Continuing the discussion from the Proof of Theorem \ref{theorem:main1}, and using its notation, according to \cite[Lemma 7.7.5]{Hoermander-I-03} for any $j$ we have 
\begin{equation}\label{eq:stat-phase-lemma}
    L_ju = \sum_{\nu - \mu = j} \sum_{2\nu \geq 3\mu \geq 0} i^{-j} 2^{-\nu} (\mu! \nu!)^{-1} \langle{(-i\lambda''_{\mathbf{0}})^{-1}D, D}\rangle^{\nu} (g^\mu u)(\mathbf{0}),\quad u \in C^\infty(U_0),
\end{equation}
is a differential operator of order $2j$ acting on $u$ at $\btheta = \mathbf{0}$, where $g$ vanishes to third order at $\mathbf{0}$ and is defined by 
\[
    ig(\btheta) = \lambda_{\btheta} - \frac{1}{2} \langle{\lambda_{\mathbf{0}}'' \btheta, \btheta}\rangle.
\]
Also, here $D = -i \partial_{\btheta}$ denotes the gradient operator times $-i$. We have $L_0 = \id$ is the identity operator, and from \eqref{eq:stat-phase-lemma} we read that
\begin{align}\label{eq:L1}
\begin{split}
    i L_1 u =& 2^{-1} \langle{(-i\lambda_{\mathbf{0}}'')^{-1}D, D}\rangle u (\mathbf{0}) + 2^{-2} \frac{1}{2!} \langle{(-i\lambda_{\mathbf{0}}'')^{-1}D, D}\rangle^2 (g u) (\mathbf{0})\\ 
    &+  2^{-3} \frac{1}{2! 3!}\langle{(-i\lambda_{\mathbf{0}}'')^{-1}D, D}\rangle^3 (g^2) u(\mathbf{0}),\quad u \in C^\infty_{\comp}(U_0),
\end{split}
\end{align}

Let $u(\btheta) := \langle{\Pi^+_{\btheta} f_{\btheta}, g_{\btheta}}\rangle_{L^2}$; in what follows we compute the first two derivatives of $u$. We first note that
\begin{align*}
    \partial_{\btheta_i} \Pi^+_{\btheta} &= - \frac{1}{2\pi i} \oint_{\gamma_0} R^+_{\btheta}(z) i \iota_X \partial_{\btheta_i} \eta_{\btheta} R_{\btheta}^+(z) \dd z\\
    &= -\frac{1}{2\pi i}\oint_{\gamma_0} \left(R^+_{\hol, \btheta}(z) - \frac{\Pi^+_{\btheta}}{z - \lambda_{\btheta}}\right) i \iota_X \partial_{\btheta_i} \eta_{\btheta} \left(R^+_{\hol, \btheta}(z) - \frac{\Pi^+_{\btheta}}{z - \lambda_{\btheta}}\right) \dd z\\
    &= \Pi^+_{\btheta} i \iota_X \partial_{\btheta_i} \eta_{\btheta} R^+_{\hol, \btheta}(\lambda_{\btheta}) + R^+_{\hol, \btheta}(\lambda_{\btheta}) i \iota_X \partial_{\btheta_i} \eta_{\btheta} \Pi^+_{\btheta},
\end{align*}
where $R^+_{\hol, \btheta}(z)$ denotes the holomorphic part of $R^+_{\btheta}(z)$ near $z = \lambda_{\btheta}$. We therefore have for an index $j \in \{1, \dotsc, d\}$ that
\begin{align*}
    \partial_{\btheta_j} u &= \langle{(\Pi^+_{\btheta} i \iota_X \partial_{\btheta_j} \eta_{\btheta} R^+_{\hol, \btheta}(\lambda_{\btheta}) + R^+_{\hol, \btheta}(\lambda_{\btheta}) i \iota_X \partial_{\btheta_j} \eta_{\btheta} \Pi^+_{\btheta}) f_{\btheta}, g_{\btheta}}\rangle\\ 
    &\hspace{6cm}+ \langle{\Pi^+_{\btheta} \partial_{\btheta_j}f_{\btheta}, g_{\btheta}}\rangle + \langle{\Pi^+_{\btheta} f_{\btheta}, \partial_{\btheta_j} g_{\btheta}}\rangle\\ 
    &= \langle{\Pi^+_{\btheta}(\underbrace{\partial_{\btheta_j} + i \iota_X \partial_{\btheta_j} \eta_{\btheta} R^+_{\hol, \btheta}(\lambda_{\btheta})}_{P_{j}^+:=}) f_{\btheta}, g_{\btheta}}\rangle + \langle{\Pi^+_{\btheta} f_{\btheta}, (\underbrace{\partial_{\btheta_j} + i \iota_{-X} \partial_{\btheta_j} \eta_{\btheta} R^-_{\hol, \btheta}(\overline{\lambda_{\btheta}})}_{P_{j}^-:=})g_{\btheta} }\rangle,
\end{align*}
keeping in mind that $P_i^\pm$ depend on $\btheta$. Iterating this, we compute for $k \in \{1, \dotsc, d\}$ that 
\begin{align*}
    \partial_{\btheta_k} \partial_{\btheta_j}u = \langle{\Pi^+_{\btheta}P^+_{k} P^+_{j} f_{\btheta}, g_{\btheta}}\rangle + \langle{\Pi^+_{\btheta} P^+_{j} f_{\btheta}, P^-_{k} g_{\btheta}}\rangle + \langle{\Pi^+_{\btheta} P_{k}^+ f_{\btheta}, P_{j}^- g_{\btheta} }\rangle + \langle{\Pi^+_{\btheta} f_{\btheta}, P_{k}^- P_{j}^- g_{\btheta} }\rangle.
\end{align*}
For simplicity we now assume that 
\[
    \int_{M} f ~\dd \vol_M = \int_{M_0} f_0 ~\dd \vol_{M_0} = \int_{M} g ~\dd \vol_{M} = \int_{M_0} g_0 ~\dd \vol_{M_0} = 0.
\]
In particular $u(\mathbf{0}) = \mathbf{0}$, and from the expression for $\partial_{\btheta_j} u$ we have $\partial_{\btheta_j} u(\mathbf{0}) = 0$. Using that $g$ vanishes to third order at $\mathbf{0}$, from \eqref{eq:L1} we therefore see that the last two terms vanish, and so 
\begin{align*}
    2L_1 u &= \sum_{j, k} (-\lambda_{\mathbf{0}}'')^{-1}_{jk} \partial_{\btheta_k} \partial_{\btheta_j} u (\mathbf{0})\\
    &= \sum_{j, k} (-\lambda_{\mathbf{0}}'')^{-1}_{jk} \left(\langle{\Pi^+_{\btheta} P^+_{j} f_{\btheta}, P^-_{k} g_{\btheta}}\rangle|_{\btheta = \mathbf{0}} + \langle{\Pi^+_{\btheta} P_{k}^+ f_{\btheta}, P_{j}^- g_{\btheta} }\rangle|_{\btheta = \mathbf{0}}\right)\\
    &= 2\sum_{j, k} (-\lambda_{\mathbf{0}}'')^{-1}_{jk} \Pi^+_{\mathbf{0}} P^+_{j} f_{\btheta}|_{\btheta = \mathbf{0}} \Pi^-_{\mathbf{0}} \overline{P^-_{k}} \overline{g_{\btheta}}|_{\btheta = \mathbf{0}},
\end{align*}
where in the third equality we used that $\lambda''_{\mathbf{0}}$ is symmetric, and that $\vol_{M_0}$ is assumed to be a probability. Note also that $\overline{R^-_{\hol}(0) v} = R^-_{\hol}(0) \overline{v}$ for any $v \in C^\infty(M_0)$. We are left to compute $\Pi^+_{\mathbf{0}} P_j^+ f_{\btheta}|_{\btheta = \mathbf{0}}$ and the symmetric term acting on $g_{\btheta}$.

We compute 
\begin{align*}
    \Pi^+_{\mathbf{0}} P_j^+ f_{\btheta}|_{\btheta = \mathbf{0}} &= \int_{M_0}\left(\partial_{\btheta_j} f_{\btheta}|_{\btheta = \mathbf{0}} + i \iota_{X_{M_0}} \partial_{\btheta_j} \eta_{\btheta} R^+_{\hol}(0) f_{\mathbf{0}}\right) \dd \vol_{M_0}\\
    &= -i \int_{M_0} (fH_j)_{\mathbf{0}} ~\dd \vol_{M_0} + i\int_{M_0} R^-_{\hol}(0) (\iota_{X_{M_0}} \partial_{\btheta_j} \eta_{\btheta}) f_{\mathbf{0}} ~\dd \vol_{M_0}\\
    &= i\int_M f\left(-H_j + \pi^*\iota_{X_{M_0}} R^{1, -}_{\hol}(0) \partial_{\btheta_j} \eta_{\btheta}\right) ~\dd \vol_{M},
\end{align*}
where in the second line we used Proposition \ref{prop:differentiation-btheta}, that $R^+_{\hol}(0)^* = R^-_{\hol}(0)$ as well as that $\eta_{\btheta}$ is real-valued and that $R^\pm_{\hol}(0)$ are real operators, as well as $\pi^* f_{\mathbf{0}}(x) = \sum_{\mathbf{n}} f(\tau_{\mathbf{n}}\widehat{x})$ for an arbitrary $\widehat{x} \in \pi^{-1}(x)$. Finally, in the last line we wrote $R^{1, -}_{\hol}(0)$ to denote the holomorphic part of the resolvent at zero acting on $1$-forms, and we used that $\iota_{X_{M_0}} R^{1, -}_{\hol}(0) = R^{-}_{\hol}(0) \iota_{X_{M_0}}$.

By definition of $H_j$, see \eqref{eq:auxiliary-Hj}, we have that $dH_j = \pi^* \partial_{\btheta_j} \eta_{\btheta}$. Therefore, we have
\begin{align*}
     - dH_j + d\pi^* \iota_{X_{M_0}} R^{1, -}_{\hol}(0) \partial_{\btheta_j} \eta_{\btheta} &= -\pi^* \partial_{\btheta_j} \eta_{\btheta} + \pi^* \mc{L}_{X_{M_0}} R^{1, -}_{\hol}(0) \partial_{\btheta_j} \eta_{\btheta}\\ 
    &= -\pi^* \partial_{\btheta_j} \eta_{\btheta} + \pi^* (\id - \Pi^{1, -}_{\mathbf{0}}) \partial_{\btheta_j} \eta_{\btheta} = -\pi^*\Pi^{1, -}_{\mathbf{0}}  \partial_{\btheta_j} \eta_{\btheta},
\end{align*}
    where in the first equality we used the commutation relation $dR^{-}_{\hol}(0) = R^{1, -}_{\hol}(0) d$, as well as that $\partial_{\btheta_j} \eta_{\btheta}$ is closed, and Cartan's magic formula. In the second equality we used that $\mc{L}_{X_{M_0}} R^-_{\hol}(0) = \id - \Pi^{1, -}_{\mathbf{0}}$ similarly to the relation \eqref{eq:holomorphic-part-at-zero-identity}, where $\Pi^{1, -}_{\mathbf{0}}$ denotes the spectral projector onto coresonant $1$-forms. (Note also that $\Pi^{1, -}_{\mathbf{0}}  \partial_{\btheta_j} \eta_{\btheta} - \partial_{\btheta_j} \eta_{\btheta}$ is exact and that $\Pi^{1, -}_{\mathbf{0}}  \partial_{\btheta_j} \eta_{\btheta} \in \ker \iota_{X_{M_0}}$.) We may thus write (where $d^{-1}$ denotes a primitive of an exact form)
\[
    \Pi_{\mathbf{0}}^+ P_j^+ f_{\btheta}|_{\btheta = \mathbf{0}} = -i \int_M f d^{-1} (\pi^* \Pi_{\mathbf{0}}^{1, -} \partial_{\btheta_j} \eta_{\btheta})~ \dd \vol_M,
\]
which is well-defined since we assume $\int_M f~ \dd \vol_M = 0$ and the action of $d^{-1}$ is well-defined up to a constant. The term corresponding to $\overline{P_k^- g_{\btheta}|_{\btheta = \mathbf{0}}}$ is computed by symmetry (replacing $X_{M_0}$ by $-X_{M_0}$).

We now prove directly that this term is non-zero; we argue by contradiction. This is equivalent to having $-H_j + \pi^*\iota_{X_{M_0}} R^{1, -}_{\hol}(0) \partial_{\btheta_j} \eta_{\btheta}$ is constant, since we had already assumed that $f$ has zero average. In particular, it follows that $R^-_{\hol}(0) \iota_{X_{M_0}}\partial_{\btheta_j}\eta_{\btheta} =: v \in C^\infty(M_0)$. Applying $X_{M_0}$ it follows that
\[
    \iota_{X_{M_0}} \partial_{\btheta_j} \eta_{\btheta} = X_{M_0}v \iff \partial_{\btheta_j} \eta_{\btheta} - dv \in \ker d \cap \ker \iota_{X_{M_0}} \cap C^\infty.
\]
Arguing as in the proof of Lemma \ref{lemma:resonance-non-degenerate} we get that $dv = \partial_{\btheta_j} \eta_{\btheta}$, which contradicts the definition of $\eta_{\btheta}$ by using the surjectivity of the representation $\rho$.

Thus in general, we have
\begin{equation}\label{equation:c1}
\begin{split}
    C_1(f, g) = \frac{(2\pi)^{\frac{d}{2}}}{\sqrt{\det \Var_{\vol}(\iota_{X_{M_0}} D\eta_{\mathbf{0}})}}\left(\sum_{j, k}(-\lambda''_{\mathbf{0}})^{-1}_{jk} \right. &\int_M f d^{-1} (\pi^* \Pi^{1, -}_{\mathbf{0}} \partial_{\btheta_j} \eta_{\btheta}) \dd \vol_M\\ 
    & \left. \cdot\int_M \overline{g} d^{-1} (\pi^* \Pi^{1, +}_{\mathbf{0}} \partial_{\btheta_k} \eta_{\btheta}) \dd \vol_M \right),
\end{split}
\end{equation}
when $\int_M f \dd \vol_M = \int_M g \dd \vol_M = 0$. Note that this pairing is non-trivial, because $\{\mathbf{1}_M, d^{-1} \pi^* \Pi_{\mathbf{0}}^{1, +} \partial_{\btheta_1} \eta_{\btheta}, \dotsc, d^{-1} \pi^*\Pi_{\mathbf{0}}^{1, +} \partial_{\btheta_d} \eta_{\btheta}\}$ are linearly independent by the argument given in the preceding paragraph.

\begin{remark}\rm
    It is possible to write down the full term for $C_1$ (and not only its restriction to zero average functions) by studying further the expressions in \eqref{eq:L1}. To do so, one would compute the terms $P_k^+ P_j^+ f_{\btheta}|_{\btheta = \mathbf{0}}$. We refrain from doing so as the terms would get more complicated quickly; in particular, this would involve the $\btheta$-derivatives of $R_{\hol, \btheta}^+(\lambda_{\btheta})$, which in turn would involve higher order coefficients in the Taylor expansion of $R^+_{\btheta}(\lambda_{\btheta})$ and also derivatives of $\btheta$. A helpful assumption is to assume that $(\varphi_t)_{t \in \mathbb{R}}$ comes from an Anosov geodesic flow because then all odd derivatives of $\lambda_{\btheta}$ at $\btheta = \mathbf{0}$ vanish identically. The pairing $C_2$ can be computed similarly with the assumption that $\Pi^+_{\mathbf{0}} P_j^+ f_{\btheta}|_{\btheta = \mathbf{0}}$ for all indices $1 \leq j \leq d$ and that $f$ has zero average. The computation of $C_\ell$ for $\ell \geq 3$ is less clear but a natural inductive assumption would be to assume that $P^+_\alpha f_{\btheta}|_{\btheta = \mathbf{0}} = 0$ for all $|\alpha| \leq \ell - 1$.
\end{remark}

\subsection{Non-vanishing of the bilinear forms}

Finally, we prove that the bilinear forms $C_j$ defined in \eqref{equation:expression-cj} do not vanish:

\begin{lemma}
    \label{lemma:cj-non-zero}
    For all $j \geq 1$, the bilinear form $C_j \colon C^\infty_{\comp}(M) \times  C^\infty_{\comp}(M) \to \C$ is non zero.
\end{lemma}

\begin{proof}
    The operator $L_j$ in the expression \eqref{equation:expression-cj} of $C_j(f,g)$ is an elliptic differential operator of order $2j$ applied to $\langle\Pi_{\btheta}^+f_{\btheta},g_{\btheta}\rangle$ (and evaluated at $\btheta=\mathbf{0}$). Recall that near $\btheta = \boldsymbol{0}$, $\Pi_{\btheta}^+ = \langle \bullet,v_{\btheta}\rangle u_{\btheta}$, where $u_{\btheta}$ (resp. $v_{\btheta}$) is the corresponding resonant state (resp. coresonant state), see \eqref{eq:projector-near-the-origin} and the discussion below. Therefore, near $\btheta=\mathbf{0}$, for $f,g \in C^\infty_{\comp}(M)$, $\langle\Pi_{\btheta}^+f_{\btheta},g_{\btheta}\rangle = \langle f_{\btheta},v_{\btheta}\rangle \langle u_{\btheta},g_{\btheta}\rangle$. 

    We consider $f,g \in C^\infty_{\comp}(M)$ with $\int_{M_0} f_{\boldsymbol{0}} \dd\vol_{M_0} =1 = \int_{M_0} g_{\boldsymbol{0}} \dd\vol_{M_0}$ and let $c(\btheta) := \langle f_{\btheta},v_{\btheta}\rangle \langle u_{\btheta},g_{\btheta}\rangle$. This is a smooth function of $\btheta$ near $\btheta=\boldsymbol{0}$. As $u_{\mathbf{0}} = v_{\mathbf{0}} = \mathbf{1}_{M_0}$, we find that $c(\boldsymbol{0})=1$. Also note that for $\bk \in \Z^d$, $\tau_{\bk}^*f \in C^\infty_{\comp}(M)$ and $(\tau_{\bk}^*f)_{\btheta} = e^{i\btheta\cdot\bk}f_{\btheta}$. Therefore $\langle\Pi_{\btheta}^+(\tau_{\bk}^*f)_{\btheta},g_{\btheta}\rangle = e^{i\btheta\cdot\bk} c(\btheta)$ near $\btheta = \mathbf{0}$.
    
    We then compute asymptotically $L_j \langle\Pi_{\btheta}^+(\tau_{\bk}^*f)_{\btheta},g_{\btheta}\rangle|_{\btheta=\boldsymbol{0}}$ in the specific case where $\bk=(k,0, \dotsc,0) \in \Z^d$ and $k \to +\infty$. As $L_j$ is a differential operator of order $2j$, we find:
    \[
    \begin{split}
    L_j \langle\Pi_{\btheta}^+(\tau_{\bk}^*f)_{\btheta},g_{\btheta}\rangle|_{\btheta=0}&  = L_j(e^{i\btheta\cdot\bk}c(\btheta))|_{\btheta=\boldsymbol{0}}  =k^{2j} \sigma_{L_j}(\boldsymbol{0}, \mathbf{e}_1^*) c(\boldsymbol{0}) + \mc{O}(k^{2j-1}) \\
    &= k^{2j} \sigma_{L_j}(\boldsymbol{0}, \mathbf{e}_1^*) + \mc{O}(k^{2j-1}),
    \end{split}
    \]
    where $\sigma_{L_j} \in C^\infty(T^*\mathrm{U}(1)^d)$ denotes the (homogeneous) principal symbol of $L_j$, evaluated at the point $\btheta=\boldsymbol{0} \in \mathrm{U}(1)^d$ and covector $\mathbf{e}_1^* \in T^*_{\boldsymbol{0}}\mathrm{U}(1)^d$ (such that $\mathbf{e}_1^*(\partial_{\theta_1})=1$ and $\mathbf{e}_1^*(\partial_{\theta_j})=0$ for $j \geq 2$). As $L_j$ is elliptic, $\sigma_{L_j}(\boldsymbol{0}, \mathbf{e}_1^*) \neq 0$ so the above expression is non zero for $k \gg 1$ large enough. This proves the claim.
\end{proof}

\begin{remark}\rm
    The principal symbol of $L_{j}$ is read-off from its expression in the stationary phase lemma, see \eqref{eq:stat-phase-lemma}. Indeed, only the term for $\mu = 0$ and $\nu = j$ appears in the principal symbol as the others are of lower order thanks to the fact that $g$ vanishes to third order at $\btheta = \mathbf{0}$. Therefore
    \begin{equation}
        \label{equation:expression-symbol-cj}
        \sigma_{L_j}(\mathbf{0}, \xi) = i^{-j - 1} 2^{-j} (j!)^{-1} \langle{(-\lambda''_{\mathbf{0}})^{-1}\xi, \xi}\rangle^j,\quad \xi \in T^*_{\mathbf{0}} \mathrm{U(1)}^d \simeq \mathbb{R}^d.
    \end{equation}
\end{remark}

\subsection{Linear drift}

\label{ssection:linear-drift}

Let $f,g \in C^\infty_{\comp}(M)$. The asymptotic expansion of the correlation function \eqref{equation:decay-correlation-0} applied with $\tau_{\bk}^*f$ and $g$ yields:
\[
\begin{split}
t^{d/2}\langle \tau_{\bk}^*f \circ \varphi_{-t}, g \rangle_{L^2(M)} & = \kappa  \int_M f ~\dd\vol_{M} \int_M g ~\dd\vol_{M} + \sum_{j=1}^{N-1} t^{-j} C_j(\tau_{\bk}^*f,g)  \\
& + \mc{O}\left(\langle t \rangle^{-N}\left(\|\tau_{\bk}^*f\|_{B^{\vartheta N,0}}\|g\|_{B^{\vartheta N,0}} + \|\tau_{\bk}^*f\|_{B^{s,2N+d+1}}\|g\|_{B^{s,2N+d+1}}\right)\right).
\end{split}
\]
Observe that $\|\tau_{\bk}^*f\|_{B^{\vartheta N,0}} = \|f\|_{B^{\vartheta N,0}}$ and $\|\tau_{\bk}^*f\|_{B^{s,2N+d+1}} \leq C \langle \bk \rangle^{2N+d+1}\|f\|_{B^{s,2N+d+1}}$. In addition, using \eqref{equation:expression-cj}, we have
\[
\begin{split}
    \tfrac{\sqrt{\det \Var_{\vol} (\iota_{X_{M_0}} D \eta_{\mathbf{0}})}}{(2\pi)^{\frac{d}{2}}}C_j(\tau_{\bk}^*f, g) & =  L_j \langle{\Pi^+_{\btheta} (\tau_{\bk}^*f)_{\btheta}, g_{\btheta}}\rangle|_{\btheta = \mathbf{0}} \\
    & =  L_j(e^{i\bk\cdot\btheta}\langle{\Pi^+_{\btheta} f_{\btheta}, g_{\btheta}}\rangle)|_{\btheta = \mathbf{0}} \\
    & =  |\bk|^{2j} \left(\sigma_{L_j}(\mathbf{0},\bk/|\bk|) \int_M f ~\dd\vol_{M} \int_M ~g \dd\vol_{M} + \mc{O}(|\bk|^{-1})\right),
    \end{split}
    \]
    where we used in the last line that $L_{2j}$ is a differential operator of order $2j$ and that $\langle{\Pi^+_{\btheta} f_{\btheta}, g_{\btheta}}\rangle|_{\btheta=\boldsymbol{0}}= \int_M f ~\dd\vol_{M} \int_M g ~\dd\vol_{M}$. We thus find:
    \[
    \begin{split}
    t^{d/2}\langle \tau_{\bk}^*f \circ \varphi_{-t}, g \rangle_{L^2(M)} & = A \cdot \sum_{j=0}^{N-1} t^{-j} |\bk|^{2j} \sigma_{L_j}(\mathbf{0},\bk/|\bk|) + \mc{O}(|\bk|^{2j-1})  + \mc{O}(\langle t \rangle^{-N}\langle \bk \rangle^{2N+d+1}),
    \end{split}
    \]
    where
    \[
    A := \frac{(2\pi)^{\frac{d}{2}}\int_M f ~\dd\vol_{M} \int_M g ~\dd\vol_{M}}{\sqrt{\det \Var_{\vol} (\iota_{X_{M_0}} D \eta_{\mathbf{0}})}}.
    \]
We now consider the regime $|\bk| = t^{1/2-\eps}$ with $0 < \eps \leq 1/2$. Then:
\[
\begin{split}
t^{d/2}\langle \tau_{\bk}^*f \circ \varphi_{-t}, g \rangle_{L^2(M)}& = A \cdot \sum_{j=0}^{N-1} t^{-2\eps j}  \sigma_{L_j}(\mathbf{0},\bk/|\bk|) + \mc{O}(t^{-2\eps j -1/2+\eps})  \\
& \qquad  + \mc{O}(\langle t \rangle^{-N + (2N+d+1)(1/2-\eps)}).
\end{split}
\]
In the sum, the remainders $\mc{O}(t^{-2\eps j -1/2+\eps})$ are negligible compared to the main terms in the expansion provided $-2\eps -1/2+\eps = -\eps - 1/2 < -2\eps(N-1)$, that is $\eps < 1/(2(2N-3))$. In addition, the remainder $\mc{O}(\langle t \rangle^{-N + (2N+d+1)(1/2-\eps)})$ is negligible compared to the main terms in the expansion provided $-N + (2N+d+1)(1/2-\eps) < -2(N-1)\eps$ that is $\eps > (d+1)/2(d+3)$. The condition $1/(2(2N-3)) > (d+1)/2(d+3)$ is equivalent to $\tfrac{d+1}{d+3} < \tfrac{1}{2N-3}$. It is always satisfied for $N = 2$, regardless of the value of $d \geq 1$, but it is never satisfied for $N \geq 3$ even for $d=1$. Letting $N=2$, and using the expression \eqref{equation:expression-symbol-cj} for the principal symbol of $L_1$, we then find:
\[
\begin{split}
t^{d/2}\langle \tau_{\bk}^*f \circ \varphi_{-t}, g \rangle_{L^2(M)}  = & \frac{(2\pi)^{\frac{d}{2}}\int_M f ~\dd\vol_M \int_M g ~\dd\vol_M}{\sqrt{\det \Var_{\vol} (\iota_{X_{M_0}} D \eta_{\mathbf{0}})}}\left(1- t^{-2\eps} 2^{-1} \left\langle(-\lambda_{\boldsymbol{0}}'')^{-1}\tfrac{\bk}{|\bk|},\tfrac{\bk}{|\bk|}\right\rangle \right) \\
&  + o(t^{-2\eps}).
\end{split}
\]
This provides a fairly explicit expression for the second term in the expansion of the correlation function.

\section{Abelian covers of isometric extensions}

\label{section:decay-isometric}

Throughout this section, $p \colon P_0 \to M_0$ is a principal $G$-bundle and $\pi \colon M \to M_0$ denotes the $\Z^d$-extension and $P =\pi^*P_0 \to M$ is the pullback $G$-bundle. When the context is clear, we will use the convention that the projection $G$-bundle map $P \to P_0$ is also denoted by $\pi$, and that the projection $P \to M$ is denoted by $p$; similarly, we will also write $\pi$ for the induced map between bundles associated to subgroups of $G$ and $P \to P_0$. Recall from Lemma \ref{lemma:g-z-bundles-relation} that $P\to M_0$ is also a $(G \times \Z^d)$-bundle.

\subsection{Borel-Weil-Floquet theory} \label{ssection:borel-weil-floquet} In this paragraph, we generalize the Floquet theory explained in \S\ref{ssection:floquet-theory} in order to include a $G$-bundle part. Let $f \in C^\infty_{\comp}(P)$. Applying Proposition \ref{eq:fourier-theory}, we may write
\[
f = \dfrac{1}{(2\pi)^d} \int_{\mathrm{U}(1)^d} F_{\btheta} ~\dd\btheta,
\]
where $F_{\btheta} \in C^\infty(P)$ is a $\Z^d$-equivariant function satisfying $\tau_{\mathbf{n}}^*F_{\btheta} = e^{i\n\cdot\btheta}F_{\btheta}$ for all $\n \in \Z^d$; the space of such functions is denoted by $C^\infty_{\btheta}(P)$. After fixing an open contractible set $U \subset \mathrm{U}(1)^d$, for $\btheta \in U$, this function can be written as $F_{\btheta} = (\pi^* f_{\btheta}) s_{\btheta}$, where $f_{\btheta} \in C^\infty(P_0)$ and $s_{\btheta} \in C^\infty_{\btheta}(P)$ has pointwise unit norm. Note that $s_{\btheta}$ can simply be taken to be the pullback by $p$ of a corresponding element of $C_{\btheta}^\infty(M)$; by slightly abusing the notation, we will denote this element by $s_{\btheta}$ as well. In turn, $f_{\btheta}$ can be decomposed using the fiberwise Fourier transform \eqref{equation:fourier-group} on the group $G$, that is $\mc{F}(f_{\btheta}) = (f_{\btheta,\bk, i})_{\bk \in \widehat{G}, i=1,\dotsc, d_{\bk}}$, where $f_{\btheta,\bk,i} \in C^\infty_{\hol}(F_0,\Lk)$ and $F_0 := P_0/T$ is the flag bundle over $M_0$. We note that technically, in the Fourier transform of the group $G$ there is an additional index $i = 1,\dotsc, d_{\bk}$, where we recall $d_{\bk}$ denotes the dimension of the representation $\bk$ and $i$ indexes the isomorphic copies of the same representation. However, since the arguments below are the same for every $i$ we often ignore the index $(\bk, i)$ in the Fourier components and only write $\bk$.

We may also define $F_{\btheta,\bk} := (\pi^* f_{\btheta,\bk}) s_{\btheta} \in C^\infty_{\hol,\btheta}(F, \pi^*\Lk)$, where $F := P/T = \pi^*F_0$ is the flag bundle over $M$, and we see $s_{\btheta}$ as a function on $F$, and, similarly to \eqref{equation:equivariant-space}, the subscript $\btheta$ indicates the equivariance of this section with respect to the natural induced action of $\Z^d$ on $F$ which lifts to a $\Z^d$-action on $C^\infty(F, \pi^*\Lk)$. In what follows to simplify the notation we will write $\Lk \to F$ instead of $\pi^*\Lk \to F$. Then we may write $\mc{F}(f) = (f_{\kk, i})_{\kk \in \widehat{G}, i = 1, \dotsc, d_{\bk}}$, where $f_{\kk, i} \in C^\infty_{\hol}(F, \Lk)$ and
\[
    f_{\kk, i} = \frac{1}{(2\pi)^d} \int_{\mathrm{U}(1)^d} F_{\btheta, \kk, i} ~\dd\btheta.
\]
Note that here $F_{\btheta, \kk, i}$ do not depend on $U$ and the choices made when selecting $(s_{\btheta})_{\btheta \in U}$. In summary, we may either take a partial Fourier transform in $G$, and then in $\Z^d$, or vice versa, and the two procedures commute.

Alternatively, one may choose not to trivialize the line bundle $L_{\btheta} \to M_0$ (using the section $s_{\btheta}$); in this case, $f_{\btheta,\bk}$ should be seen as a fiberwise holomorphic section of $p^* L_{\btheta} \otimes \Lk \to F_0$, where $p : F_0 \to M_0$ is the footpoint projection and $p^*L_{\btheta} \to F_0$ is the pullback bundle. Since the realizations described here are all unitarily equivalent, we will not further distinguish between the corresponding objects.

The following result generalizes Proposition \ref{eq:fourier-theory}; its proof is a straightforward consequence of the previous paragraph, and the fact that the Fourier transform is an $L^2$-isometry, see e.g. \cite[Lemma 2.2.4]{Cekic-Lefeuvre-24}:

\begin{proposition} \label{proposition:borel-floquet-weil}
Let $f \in C^\infty_{\comp}(P)$. Then
\[
f = \dfrac{1}{(2\pi)^d} \int_{\mathrm{U}(1)^d} \mc{F}^{-1}\left(\bigoplus_{\bk \in \widehat{G}, i = 1, \dotsc, d_{\bk}} F_{\btheta,\bk, i}\right) \dd\btheta
\]
In addition, we have the following Parseval identity: for all $f,g \in C^\infty(P)$,
\[
\langle f,g \rangle_{L^2(P)} = \dfrac{1}{(2\pi)^d} \sum_{\bk \in \widehat{G}} d_{\bk}\sum_{i = 1}^{d_{\bk}}\int_{\mathrm{U}(1)^d} \langle F_{\btheta,\bk, i}, G_{\btheta,\bk, i}\rangle_{L^2(F_0, p^*L_{\btheta} \otimes \Lk)} \dd \btheta.
\]
\end{proposition}

Let $\mathbf{X} := X_P$ be the flow generator on $P$ and denote the flow generator on $P_0$ by $X_{P_0}$. It then induces a family of operators
\[
X_{\btheta,\bk} \colon C^\infty_{\hol}(F_0, p^*L_{\btheta} \otimes \Lk) \to C^\infty_{\hol}(F_0, p^*L_{\btheta} \otimes \Lk)
\]
in the Borel-Weil calculus. Alternatively, working over the open contractible subset $U \subset \mathrm{U}(1)^d$, the operator $X_{\btheta}: C^\infty(P_0) \to C^\infty(P_0)$ takes the form 
\[
    X_{\btheta} = X_{P_0} + i p^*(\iota_{X_{M_0}}\eta_{\btheta}),
\]
where $(\eta_{\btheta})_{\btheta \in U}$ is the family of $1$-forms on $M_0$ constructed in Proposition \ref{prop:eta-theta}; see \S \ref{sssection:conjugacy} where this is discussed in the case of the extension to $M \to M_0$ (and where $P_0 = M_0$). It is easily seen that the associated operator $X_{\btheta, \bk}: C^\infty_{\hol}(F_0, \Lk) \to C^\infty_{\hol}(F_0, \Lk)$ takes the form
\begin{equation}\label{eq:X-btheta-kk}
    X_{\btheta, \bk} = (X_{P_0})_{\bk} + i p^*(\iota_{X_{M_0}} \eta_{\btheta}).
\end{equation}

Observe that for $\bk=\mathbf{0}$, one recovers the family of operators $X_{\btheta,\bk=\mathbf{0}} = X_{\btheta}$ acting on $C^\infty(M_0,L_{\btheta})$, or equivalently on $C^\infty(M_0)$ when trivializing $L_{\btheta}$ for $\btheta \in U$ as above. This was studied in \S\ref{section:abelian-extension-anosov}.

\subsection{Resolvents and anisotropic Sobolev spaces}

We recall that $\HH \to F_0$ denotes the horizontal bundle induced by the dynamical connection, see \S \ref{ssection:dynamical-connection}; its dual $\HH^*$ is defined as the annihilator of $\V$, the tangent space to the fibres of $F_0 \to M_0$, see \S \ref{ssection:bw-calculus}. Let
\[
p^*(e^{sG_m}) \in C^\infty(\HH^*)
\]
be the pullback of the escape function introduced in \S\ref{sssection:anisotropic-spaces} to the bundle $\HH^*$ over $F_0$, defined using the isomorphism $dp^\top \colon T^*M_0 \to \HH^*$, see \eqref{equation:isomorphism-h} (we drop the subscript $F_0$ in the footpoint projection to simplify notation). That is, for $(w,\xi) \in \HH^*$:
\[
p^*(e^{sG_m})(w,\xi) := e^{sG_m(x, dp^{-\top}\xi)}, \qquad x := p(w), p \colon F_0 \to M_0.
\]
The (smooth) function $p^*e^{sG_m}$ is a symbol in the anisotropic class $S^{sp^*m(x,\xi)}(\HH^*)$. Following \cite[Lemma 4.3.1]{Cekic-Lefeuvre-24}, this allows to define 
\[
\mathbf{A}(s) \in \Psi^{sm(x,\xi)}_{h,\mathrm{BW}}(P_0, p^* L_{\btheta}), \qquad \mathbf{A}(s) := \Op_{h}^{\mathrm{BW}}(p^* e^{sG_m} \otimes \mathbf{1}_{L_{\btheta}}),
\]
which is a family of pseudodifferential operators in the Borel-Weil calculus, parametrized by $s \geq 0$, and acting for all $h > 0, \bk \in \widehat{G}$ such that $h|\bk| \leq 1$ as
\[
A_{h,\btheta,\bk}(s) \colon C^\infty_{\hol}(F_0, p^*L_{\btheta} \otimes \Lk) \to C^\infty_{\hol}(F_0, p^*L_{\btheta} \otimes \Lk).
\]
Up to lower-order modification of the symbol, the above map can be made an isomorphism.

As in \S\ref{sssection:anisotropic-spaces}, we define the anisotropic Sobolev spaces by setting for $f_{\btheta,\bk} \in C^\infty_{\hol}(F_0, p^*L_{\btheta}\otimes \Lk)$:
\[
\|f_{\btheta,\bk}\|_{\mc{H}^{s}_{h,\pm}(F_0, p^*L_{\btheta}\otimes \Lk)} := \|A_{h,\btheta,\bk}(s)^{\mp 1}f_{\btheta,\bk}\|_{L^2(F_0, p^*L_{\btheta}\otimes \Lk)}.
\]
The space $\mc{H}^{s}_{h,\pm}(F_0, p^*L_{\btheta}\otimes \Lk)$ is then defined as the completion of $C^\infty_{\hol}(F_0, p^*L_{\btheta} \otimes \Lk)$ with respect to the above norm.

When $h = 1/\langle \bk \rangle$, we shall simply write $\mc{H}^{s}_{\pm}(F_0,p^*L_{\btheta}\otimes \Lk)$. Similarly to Lemma \ref{theorem:meromorphic-extension}, the resolvents
\[
z \mapsto (\mp X_{\btheta,\bk}-z)^{-1} \in \mc{L}(\mc{H}^{s}_{h,\pm}(F_0,p^*L_{\btheta}\otimes \Lk))
\]
admit a meromorphic extension from $\{\Re(z) \gg 0\}$ to $\{\Re(z) > -cs\}$, see \cite[Theorem 4.3.1]{Cekic-Lefeuvre-24}. The poles are contained in $\{\Re(z) \leq 0\}$ and independent of choices made in the construction.

Alternatively, as in \S \ref{ssection:borel-weil-floquet} above, taking an open contractible set $U \subset \mathrm{U}(1)^d$, we may trivialise $L_{\btheta}$ for $\btheta \in U$. In this setting, we obtain the usual spaces $\mc{H}^s_{\pm}(F_0, \Lk)$ and norms on them  which are independent of $\btheta$; on these spaces the resolvents $(\mp X_{\btheta, \bk} - z)^{-1}$ act (c.f. Section \ref{section:abelian-extension-anosov} for the case when $P_0 = M_0$).

\subsection{Resonances on the imaginary axis} Recall that $F_1, \dotsc, F_a \in \mc{D}'(M_0, \Lambda^2 T^*M_0)$ are the curvature $2$-forms of the dynamical connection corresponding to the central part. We let $(\psi_t^{F_0})_{t \in \R}$ be the flow on $F_0$ induced by $(\psi_t)_{t \in \R}$, defined on $P_0$.

\begin{lemma}
\label{lemma:no-res-k}
Assume that $(\psi_t^{F_0})_{t \in \R}$ is ergodic on $F_0$ and that $(d\alpha, F_1, \dotsc, F_a)$ are linearly independent over $\R$.
    Let $\bk \in \widehat{G}$, $\bk \neq \mathbf{0}$ and $\btheta \in \mathrm{U}(1)^d$. Then $X_{\btheta,\bk}$ has no resonances on $i\R$.
\end{lemma}

\begin{proof}
Assume for a contradiction that there exists a non-trivial solution
\[
u_{\btheta,\bk} \in \mc{H}^{s}_{\pm}(F_0,p^*L_{\btheta}\otimes \Lk), \qquad (-X_{\btheta,\bk}-i\lambda) u_{\btheta,\bk} = 0,
\]
for some $\lambda \in \R$. By \cite[Proposition 9.2.3, item (ii)]{Lefeuvre-book}, $u_{\btheta,\bk} \in C^\infty_{\hol}(F_0,p^*L_{\btheta}\otimes \Lk)$. Since $X_{F_0}|u_{\btheta,\bk}|^2 = 2\Re(\langle X_{\btheta,\bk} u_{\btheta,\bk}), u_{\btheta,\bk} \rangle) = 0$, we deduce that $|u_{\btheta,\bk}| = c > 0$ is constant by ergodicity of $(\psi_t^{F_0})_{t \in \R}$ on $F_0$. Restricting to a fiber $F_0(x_0) \simeq G/T$ for $x_0 \in M$, we find that $u_{\btheta,\bk}$ is a nowhere vanishing section of $p^*L_{\btheta} \otimes \Lk \to F_0(x_0)$. However, if there exists $k_i \neq 0$ for $i \in \{a+1, \dotsc , d\}$, then $p^*L_{\btheta} \otimes \Lk \to F_0(x_0)$ is a non-trivial line bundle (see \cite[Proposition 2.1.8]{Cekic-Lefeuvre-24}) and this is contradiction. Therefore, $k_{a+1}= \dotsb =k_d=0$.

As $p^*L_{\btheta} \otimes \Lk \to F_0$ is trivial over every fiber of $F_0 \to M_0$, it is the pullback of a line bundle $L_{\btheta} \otimes \Lk_{M_0} \to M_0$ over $M_0$. In addition, fiberwise holomorphic sections in $C^\infty_{\hol}(F_0,p^*L_{\btheta} \otimes \Lk)$ are constant along each fiber and obtained as pullback of sections of $L_{\btheta} \otimes \Lk_{M_0} \to M_0$. That is $u_{\btheta,\bk} = p^* \tilde{u}_{\btheta,\bk}$ for some $\tilde{u}_{\btheta,\bk} \in C^\infty(M_0,L_{\btheta} \otimes \Lk_{M_0})$ (see also the discussion on \cite[Case 2 of Proof of (4.6.1)]{Cekic-Lefeuvre-24} where a similar argument is developed). The flow invariance property $X_{\btheta,\bk} u_{\btheta,\bk} = 0$ translates on the base into $X_{\btheta,\bk,M_0} \tilde{u}_{\btheta,\bk}=0$. After trivializing $L_{\btheta}$ as explained in \S \ref{ssection:borel-weil-floquet} above, thanks to \eqref{eq:X-btheta-kk} we may write $X_{\btheta,\bk,M_0} = (\nabla^{\mathrm{dyn}}_{\kk, M_0})_{X_{M_0}} + i \iota_{X_{M_0}} \eta_{\btheta}$, where $\nabla^{\mathrm{dyn}}_{\kk, M_0}$ denotes the connection on $\Lk_{M_0}$ associated to the dynamical connection. Arguing as in the proof of Lemma \ref{lemma:no-resonances-imaginary-axis}, we find, similarly to \eqref{equation:no-res}, that (see also \cite[Derivation of (4.6.9)]{Cekic-Lefeuvre-24})
\[
(-\nabla^{\mathrm{dyn}}_{\bk,M_0}-i\eta_{\btheta} \wedge -i\lambda\alpha \wedge ) \tilde{u}_{\btheta,\bk}=0.
\]
As $|\tilde{u}_{\btheta,\bk}|=c > 0$, we obtain that the connection $\nabla^{\mathrm{dyn}}_{\bk,M_0} + i \eta_{\btheta} + i\lambda\alpha$ admits a parallel section so its curvature vanishes, that is, using that $\eta_{\btheta}$ is closed and that the curvatures of the line bundles in $\Lk_{M_0}$ can be taken to agree with $i(F_j)_{j = 1}^{a}$,
\[
    k_1 F_1 + \dotsb + k_a F_a + \lambda d\alpha = 0.
\]
By assumption, this is only possible if $k_1 = \dotsb = k_a = 0$ and $\lambda = 0$. That is $\bk = \mathbf{0}$, and we are thus back to the setting of Lemma \ref{lemma:no-resonances-imaginary-axis} which establishes the absence of resonances on the imaginary axis when $\bk = \mathbf{0}$ and $\btheta \neq \mathbf{0}$. This concludes the proof.
\end{proof}

\subsection{High-frequency estimates and decay of correlation}

The key result is the following high-frequency estimate which is analogous to Theorem \ref{theorem:high-frequency-ab-ext} (compare with \cite[Theorem 4.4.2]{Cekic-Lefeuvre-24}).

\begin{theorem}[High frequency estimates, $\bk \neq \mathbf{0}$]
\label{theorem:high-frequency}
Assume that $(\psi_t^{F_0})_{t \in \R}$ is ergodic on $F_0$ and $(d\alpha, F_1, \dotsc, F_a)$ are linearly independent over $\R$. Let $s > 0$. There exist $C,\vartheta > 0$ such that for all $z \in \mathbb{B}, \mathbf{k} \in \widehat{G} \setminus \{0\}, \btheta \in \mathrm{U}(1)^d$:
\begin{equation}
    \label{equation:twisted-hf}
\|(\mp X_{\boldsymbol{\theta},\bk}-z)^{-1}\|_{\mc{H}^s_{1/\langle\mathbf{k}\rangle,\pm}(F_0,L_{\btheta}\otimes\mathbf{L}^{\otimes \mathbf{k}})} \leq C \langle \Im(z)\rangle^{\vartheta} \langle\mathbf{k}\rangle^{\vartheta}.
\end{equation}
\end{theorem}

The proof of \eqref{equation:twisted-hf} relies on the following two estimates which hold for all $z \in \mathbb{B}, \mathbf{k} \in \widehat{G} \setminus \{0\}, \boldsymbol{\theta} \in \mathrm{U}(1)^d$ (below, $\lambda = \Im(z)$):
\begin{equation}
    \label{equation:hf-k}
\begin{split}
\begin{aligned}
& \|(\mp X_{\boldsymbol{\theta},\bk}-z)^{-1}\|_{\mc{H}^s_{1/\langle\mathbf{k}\rangle,\pm}(F_0, p^*L_{\btheta}\otimes \mathbf{L}^{\otimes \mathbf{k}})} \leq C \langle\mathbf{k}\rangle^\ell, \quad &&|\bk| \geq |\lambda|, \\
& \|(\mp X_{\boldsymbol{\theta},\bk}-z)^{-1}\|_{\mc{H}^s_{1/\langle\lambda\rangle,\pm}(F_0, p^*L_{\btheta}\otimes\mathbf{L}^{\otimes \mathbf{k}})} \leq C \langle \Im(z)\rangle^{\ell}, \quad &&|\lambda| \geq |\bk|,
\end{aligned}
\end{split}
\end{equation}
where $\ell = \max(3\dim F_0 + 5, \dim F_0 + 11)$. The case $\btheta = \mathbf{0}$ was established in \cite[Section 4.6]{Cekic-Lefeuvre-24}. The fact that the two estimates above hold uniformly with respect to $\btheta$ follows from a straightforward adaptation of the proof of \cite[Section 4.6]{Cekic-Lefeuvre-24}, using the same strategy as in the proof of Theorem \ref{theorem:high-frequency-ab-ext}. Note that by Lemma \ref{lemma:no-res-k}, the operators $(\mp X_{\btheta,\bk}-z)^{-1}$ have no resonances on the imaginary axis.

This immediately implies the following result:

\begin{proposition}[Expansion of the correlation function for $\bk \neq \mathbf{0}$]
\label{proposition:expansion}
Assume that $(\psi_t^{F_0})_{t \in \R}$ is ergodic on $F_0$ and $(d\alpha, F_1, \dotsc, F_a)$ are linearly independent over $\R$. There exists $\vartheta > 0$ such that for all $N \geq 1$, there exists a constant $C_N > 0$ such that the following holds. Let $f,g \in C^\infty_{\comp}(P)$ such that
\[
F_{\btheta,\bk=\mathbf{0}}=  G_{\btheta, \bk=\mathbf{0}} = 0, \qquad \forall \btheta \in \mathrm{U}(1)^d.
\]
Then:
\[
\left|\langle f \circ \psi_{-t}, g \rangle_{L^2(P)}\right| \leq C_N \langle t \rangle^{-N} \|f\|_{B^{\vartheta N, 0}(P)} \|g\|_{B^{\vartheta N,0}(P)}, \qquad t \geq 0.
\]
\end{proposition}

The constant $\vartheta > 0$ can be made explicit and only depends on the dimension of $M_0$ and $G$. See \cite[Theorem 4.1.1]{Cekic-Lefeuvre-24} for instance. Note that the condition that $F_{\btheta, \bk = \mathbf{0}} = 0$ is zero for all $\btheta$ is the same as $f_{\bk = \mathbf{0}} = 0$, which in turn means that the integrals of $f$ over the $G$-fibres in $P$ are all zero.

\begin{proof}
Using Proposition \ref{proposition:borel-floquet-weil}, we begin by expanding
\[
\langle f \circ \psi_{-t}, g \rangle_{L^2(P)} = \dfrac{1}{(2\pi)^d} \sum_{\bk \in \widehat{G} \setminus \{0\}} d_{\bk} \sum_{i = 1}^{d_{\bk}} \int_{\btheta \in \mathrm{U}(1)^d}  \langle e^{-t X_{\btheta,\bk}}f_{\btheta,\bk, i}, g_{\btheta,\bk, i}\rangle_{L^2(F_0, p^* L_{\btheta}\otimes \Lk)} \dd \btheta.
\]
Then, using the high-frequency estimate \eqref{equation:twisted-hf} and an integration by parts argument as in the proof of Theorem \ref{theorem:mixing_X_theta} (there is no need to perform a contour deformation in this case as there are no resonances on the imaginary axis), one obtains the uniform polynomial decay for each $\btheta$ for the propagator of $X_{\btheta, \bk}$. That is, there exist $C, \vartheta_0 > 0$, such that for every $N \geq 1$ and for all $\bk \neq \mathbf{0}$
\[
    \langle e^{-tX_{\btheta, \bk}} f_{\btheta, \bk}, g_{\btheta, \bk} \rangle_{L^2} \leq  C \langle t \rangle^{-N} \langle \bk \rangle^{\vartheta_0 N} \|f_{\btheta, \bk}\|_{H^{\vartheta_0 N}(F_0, p^*L_{\btheta} \otimes \Lk)} \|g_{\btheta, \bk}\|_{H^{\vartheta_0 N}(F_0, p^*L_{\btheta} \otimes \Lk)}.
\]
Thus we may estimate, for every $N \geq 1$
\begin{align*}
    &|\langle f \circ \psi_{-t}, g \rangle_{L^2(P)}| \leq C\langle t \rangle^{-N} \int_{\btheta \in \mathrm{U}(1)^d} \dd\btheta \sum_{\bk \in \widehat{G}\setminus\{\mathbf{0}\}} d_{\bk} \langle \bk \rangle^{\vartheta_0 N} \sum_{i = 1}^{d_{\bk}} \|f_{\btheta, \bk, i}\|_{H^{\vartheta_0 N}} \|g_{\btheta, \bk, i}\|_{H^{\vartheta_0 N}}\\
    &\leq C\langle t \rangle^{-N} \int_{\btheta \in \mathrm{U}(1)^d} \dd\btheta \left(\sum_{\bk \in \widehat{G}} d_{\bk} \langle \bk \rangle^{\vartheta_0 N} \sum_{i = 1}^{d_{\bk}} \|f_{\btheta, \bk, i}\|_{H^{\vartheta_0 N}}^2\right)^{\frac{1}{2}}\\
    &\hspace{7cm}\cdot\left(\sum_{\bk \in \widehat{G}} d_{\bk} \langle \bk \rangle^{\vartheta_0 N} \sum_{i = 1}^{d_{\bk}} \|g_{\btheta, \bk, i}\|_{H^{\vartheta_0 N}}^2\right)^{\frac{1}{2}}\\
    &\leq C\langle t \rangle^{-N} \int_{\btheta \in \mathrm{U}(1)^d} \dd\btheta \|f_{\btheta}\|_{H^{3\vartheta_0N/2}(P_0, p^*L_{\btheta})} \|g_{\btheta}\|_{H^{3\vartheta_0N/2}(P_0, p^*L_{\btheta})}\\
    &\leq C \langle t \rangle^{-N} \|f\|_{B^{\vartheta N, 0}(P)} \|g\|_{B^{\vartheta N, 0}(P)},
\end{align*}
where in the second line we used the Cauchy-Schwartz inequality, as well as that $f_{\mathbf{k} = \mathbf{0}} = g_{\bk = \mathbf{0}} = 0$ by assumption, in the third line we used the embedding of anisotropic horizontal-vertical spaces into Sobolev spaces on $P_0$ proved in \cite[Lemma 4.4.1]{Cekic-Lefeuvre-24}, and in the last line we used Lemma \ref{lemma:spaces_theta_sobolev} and we introduced $\vartheta := 3\vartheta_0/2$. This completes the proof. (See also \cite[Sections 4.4.1 and 4.4.2]{Cekic-Lefeuvre-24} for similar arguments.)
\end{proof}

Finally, we can complete the proof of Theorem \ref{theorem:main2}:

\begin{proof}[Proof of Theorem \ref{theorem:main2}]
We expand
\[
\begin{split}
\langle f \circ \psi_{-t}, g \rangle_{L^2(P)} & = \underbrace{\dfrac{1}{(2\pi)^d}\int_{\btheta \in \mathrm{U}(1)^d} \langle e^{-tX_{\btheta, \bk = \mathbf{0}}}f_{\btheta, \bk=\mathbf{0}}, g_{\btheta, \bk=\mathbf{0}}\rangle_{L^2(M_0, L_{\btheta})} \dd\btheta}_{=A} \\
& + \underbrace{\dfrac{1}{(2\pi)^d} \sum_{\bk \in \widehat{G} \setminus \{\mathbf{0}\}} d_{\bk} \sum_{i = 1}^{d_{\bk}} \int_{\btheta \in \mathrm{U}(1)^d} \langle e^{-tX_{\btheta,\bk}}f_{\btheta,\bk, i}, g_{\btheta,\bk, i}\rangle_{L^2(F_0, p^*L_{\btheta}\otimes \Lk)} \dd\btheta}_{=B}.
\end{split}
\]
Note that in the term $A$, $f_{\btheta, \bk=\mathbf{0}} \in C^\infty_{\hol}(F_0, p^*L_{\btheta})$ is identified with an element of $C^\infty(M_0,L_{\btheta})$. By Theorem \ref{theorem:main1}, the term $A$ admits an asymptotic expansion in powers of $t$ as claimed. If the transitivity group $H$ is equal to $G$, then $(\psi_t^{F_0})_{t \in \R}$ is ergodic on $F_0$. By Proposition \ref{proposition:expansion}, the term $B$ is bounded by $B \leq C_N \langle t \rangle^{-N} \|f\|_{B^{\vartheta N,0}} \|g\|_{B^{\vartheta N,0}}$ for $t \geq 0$, which concludes the proof.
\end{proof}

\bibliographystyle{alpha}
\bibliography{biblio}

\end{document}